\documentclass[reqno,12pt]{amsart} 
\usepackage{mathrsfs,amsmath,amsthm,amssymb,amsfonts,bbm}
\usepackage{esint}
\usepackage{tikz-cd}

\usepackage{color}
\usepackage[hyperfootnotes=false,pagebackref]{hyperref}
\hypersetup{
  colorlinks,
  citecolor=blue,
  linkcolor=red,
  urlcolor=blue}

\usepackage[left=1.5in, right=1.5in, top=1.2in, bottom = 1.2in]{geometry}
\allowdisplaybreaks[3]

\newtheorem{theorem}{Theorem}[section]
\newtheorem{lemma}[theorem]{Lemma}
\newtheorem{proposition}[theorem]{Proposition}
\newtheorem{corollary}[theorem]{Corollary}
\theoremstyle{definition}
\newtheorem{definition}[theorem]{Definition}

\theoremstyle{remark}
\newtheorem{remark}[theorem]{Remark}

\numberwithin{equation}{section}

\usepackage{comment}

\newcommand{\Ag}[1]{\langle#1\rangle}

\newcommand{\M}{\mathcal{M}}

\newcommand{\Z}{\mathbb{Z}}

\newcommand{\R}{\mathbb{R}}
\newcommand{\T}{\mathbb{T}}
\newcommand{\N}{\mathbb{N}}
\newcommand{\e}{\varepsilon}

\newcommand{\jz}[1]{\textcolor{red}{#1}}

\DeclareMathOperator{\diver}{div}

\newcommand{\doublerightarrow}{%
  \mathrel{%
    \mathchoice
      {\rightarrow\mspace{-10mu}\rightarrow} 
      {\rightarrow\mspace{-10mu}\rightarrow} 
      {\rightarrow\mspace{-10mu}\rightarrow} 
      {\rightarrow\mspace{-10mu}\rightarrow} 
  }%
}

\begin{document}
\title[Infinitely Many Scales]{Quantitative homogenization of elliptic equations with infinitely many scales}

\author{Zhongwei Shen}

\address{Zhongwei Shen: Institute for Theoretical Sciences, Westlake University, No. 600 Dunyu Road, Xihu District, Hangzhou,
Zhejiang 310030, P.R. China.}
\email{shenzhongwei@westlake.edu.cn}

\author{Yao Xu}
\address{Yao Xu: School of Mathematical Sciences, University of Chinese Academy of Sciences, Beijing, 100049, China.}
\email{xuyao89@gmail.com, xuyao@ucas.ac.cn}

\author{Jinping Zhuge}
\address{Jinping Zhuge: Morningside Center of Mathematics, Academy of Mathematics and systems science,
Chinese Academy of Sciences, Beijing 100190, China.}
\email{jpzhuge@amss.ac.cn}

\date{}

\begin{abstract}
    In this paper, we develop a general homogenization theory for elliptic equations with coefficients that  oscillate periodically at infinitely many scales $\varepsilon = (\varepsilon_1, \varepsilon_2, \cdots) \in (0,1)^\infty$, with $\e_1>\e_2>\cdots$ and $\varepsilon_n \to 0$ as $n \to \infty$. Such problems arise naturally in the study of fractal materials and diffusion in fluids. Under suitable scale-separation assumptions, we prove a qualitative homogenization theorem and obtain optimal $L^2$ convergence rates. We also establish interior and boundary Lipschitz estimates that are uniform in $\e$. \\

\end{abstract}

\maketitle


\section{Introduction}

\subsection{Motivations}
In this paper, we study the multiscale homogenization of elliptic equations. Precisely, we consider the boundary value problem,
\begin{equation}\label{eq_infinite}
    \begin{cases}
        -\mathrm{div}(A^{\varepsilon}(x)\nabla u_\varepsilon)=f&\text{in }\Omega,\\
        u_\varepsilon=g& \text{on }\partial\Omega,
   \end{cases}
\end{equation}
where $\Omega$ is a bounded Lipschitz domain in $\R^d$. In the classical homogenization theory of finitely many oscillating scales, one takes the coefficient matrix in the form of $A^\e(x) = A(x/\e_1, x/\e_2,\cdots, x/\e_n)$ with $\e = (\e_1, \e_2,\cdots, \e_n) \in (0,1)^n$ and $A = A(y_1,y_2,\cdots, y_n)$ satisfying certain self-similar structure conditions in each $y_i \in \R^d$. Here $n$ is a fixed integer that represents the number of oscillating scales. 
The standard method for the multiscale homogenization problem is the so-called reiterated homogenization \cite{Bensoussan1978,Allaire1996Multiscale, LLPW01, Niu20_reiterated}. More recently, new approaches have been developed to study uniform regularity \cite{NZ23, NZ24} and optimal convergence rates \cite{Niu25_Optimal}. Nevertheless, all of these results are proved using some inductive arguments on the number of scales $n$, and the constants in the estimates generally blow up as $n \to \infty$.

However, in many theoretical treatments of physical problems, it is natural and crucial to consider fractal-like materials with a large number of or even infinitely many oscillating scales. Particularly in fluid dynamics it is believed that a large collection of vortices across a wide range of scales may be an important factor driving the turbulence evolution, as Lewis F. Richardson summarized in rhyming verse \cite{Ric07} back in 1922: ``\textit{Big whirls have little whirls that feed on their velocity, and little whirls have lesser whirls and so on to viscosity.}'' While the real turbulence dynamics is much more complicated, the multiscale (or perpetual) homogenization theory has been applied in some simplified physical models, in the spirit of renormalization group method\footnote{Roughly speaking, renormalization group is a mathematical procedure to average out a material gradually from small scales to large scales, throw out the unnecessary small details and keep the most important information to get the large-scale properties of the material.}, to understand phase transitions, turbulent mixing, anomalous diffusion, etc; see \cite{Ave96} for a nice introduction in this direction. 
We refer to \cite{AM91,FP94,MK99,BO02,BO03} and references therein for some early developments on this topic. Recently, reiterated homogenization with infinitely many scales has been applied to prove the anomalous dissipation for advection-diffusion equations \cite{AV25,BSW23,BSW26}; and the coarse-graining theory\footnote{The coarse-graining theory in stochastic homogenization \cite{AS16,AKM19}, developed by S. Armstrong, T. Kuusi, et al., can be thought of as a rigorous renormalization group argument.} developed in stochastic homogenization has been used to obtain the superdiffusion for the same type of equations \cite{ABK24, ABK26}.










Due to the important applications described above, there is a practical need to develop a general homogenization theory for elliptic equations and systems\footnote{The methods in this paper do not distinguish between elliptic equations and systems.} with infinitely many scales --- this is the main objective of the present paper. At the same time, regarding multiscale homogenization itself, we also aim to understand the scope and potential limitations of homogenization problems involving infinitely many scales.

\subsection{Set-ups and main results}
Returning to equation \eqref{eq_infinite}, the first question concerns the well-definedness of coefficients involving infinitely many scales.
 Let $\e = (\e_1,\e_2,\cdots) \in (0,1)^\infty$ be the oscillating scales and $1 = \e_0 > \e_1 > \e_2 > \cdots$ with $\lim_{n\to \infty} \e_n = 0$. We define the coefficient matrix $A^\e(x)$ with infinitely many scales as the $L^\infty$-limit of coefficient matrices with finitely many scales; i.e., $A^\e(x) = \lim_{n\to \infty} A_n^\e(x)$, where $A_n^\e(x) = A_n(x,x/\e_1,\cdots, x/\e_n)$ is a coefficient matrix with $n$ oscillating scales, and $A_n = A_n(y_0, y_1,\cdots, y_n) \in C(\Omega \times \T^{d\times n})$ is periodic in each $y_i\in \T^d$ for $1\le i\le n$. If $A_n^\e(x)$ is continuous for each $n$, then the $L^\infty$-limit $A^\e(x)$ is also continuous and well-defined. 
Moreover, if we set
\begin{align*}
    B_n(y_0,y_1,\cdots, y_n) = A_n(y_0,y_1,\cdots, y_n) - A_{n-1}(y_0,y_1,\cdots, y_{n-1}),
\end{align*}
then
\begin{align}
  \label{def_An}
  A_n(y_0, y_1, \cdots, y_n)=\sum_{\ell=0}^nB_\ell(y_0, y_1, \cdots, y_\ell).
\end{align}
Thus, in this paper, we shall consider $A_n$ in a form of \eqref{def_An} and 
\begin{equation}\label{Linfity_limit}
    A^\e(x) = \lim_{n\to \infty} A_n^\e(x) \quad \text{uniformly in } x\in \Omega.
\end{equation}
Moreover, we assume that each $d\times d$ matrix $B_\ell(y_0, y_1, \cdots, y_\ell)$, 
defined on $\Omega \times \mathbb{R}^{d\times\ell}$,
is continuous and periodic in $y_1, \cdots, y_\ell$. 

Before stating our main assumptions and results, 
let us first mention that the coefficient matrix $A^\e(x)$ defined above has a fractal structure with infinitely many scales. Indeed, we can formally write 
\begin{equation*}
    A_\infty(y_0, y_1,\cdots) = \sum_{\ell=0}^\infty B_\ell(y_0, y_1, \cdots, y_\ell),
\end{equation*}
and
\begin{equation*}
    A^\e(x) = A_\infty\Big(x,\frac{x}{\e_1}, \frac{x}{\e_2}, \cdots \Big).
\end{equation*}
At scale $\e_0 = 1$, we see infinitely many oscillating scales in a decreasing order $\e_1 > \e_2 > \cdots $. If we zoom in and look at the coefficients at the scale $\e_m$, the rescaled coefficient matrix takes a form of
\begin{equation*}
    A^\e(\e_m x) = A_\infty\Big(\e_m x,\frac{ x}{\e_1/\e_m}, \frac{x}{\e_2/\e_m}, \cdots, \frac{x }{\e_{m-1}/\e_m}, x, \frac{x}{\e_{m+1}/\e_m}, \cdots \Big).
\end{equation*}
Observe that although the first $m$
oscillating scales are flattened, we are still left with infinitely many oscillating scales: 
 $\frac{\e_{m+1}}{\e_{m}} > \frac{\e_{m+2}}{\e_m} > \cdots $. This means that no matter how small a scale we zoom in to, we always see the oscillations of coefficients at infinitely many scales --- a phenomenon that demonstrates the characteristic of fractal structures.

We now introduce the ellipticity, boundedness and periodicity assumptions for $\{ A_n\}_{n\in \N}$ given in \eqref{def_An}.
\begin{itemize}
    \item Ellipticity: There exists a fixed $\mu > 0$ such that every $n \in \N$,
    \begin{equation}\label{cond_ellipticity}
        \xi \cdot A_n\xi \geq \mu|\xi|^2\quad \text{for any }\xi \in\mathbb{R}^{d}; 
    \end{equation}

    \item Boundedness: 
    \begin{equation}\label{cond_bounded}
        \sum_{\ell = 0}^\infty \| B_\ell \|_{L^\infty} \le \mu^{-1}.
    \end{equation}

    \item 1-Periodicity: For arbitrarily $n \in \N_+:= \N \setminus \{0\}$,
    \begin{equation}\label{cond_periodicity}
        A_n(y_0, y_1 + z_1, \cdots, y_n + z_n ) = A_n(y_0, y_1, \cdots, y_n),
    \end{equation}
    for any $z_j \in \Z^d, j = 1,2,\cdots, n$.
\end{itemize}
Clearly, \eqref{cond_bounded} implies that each $A_n$ satisfies the boundedness condition uniformly in $n$, i.e., 
\begin{equation}\label{cond_bounded-A}
    |A_n \xi| \le \mu^{-1} |\xi| \quad \text{for any } \xi\in \R^d.
\end{equation}
As a result, the limiting matrix $A^\varepsilon(x)$ also satisfies the same ellipticity and boundedness conditions and the Dirichlet problem \eqref{eq_infinite} is well-posed for $f\in H^{-1}(\Omega)$ and $g\in H^{1/2}(\partial \Omega)$. 


Our first result is a qualitative homogenization theorem. The same as in homogenization with finitely many scales \cite{Allaire1996Multiscale}, the qualitative homogenization holds under the scale-separation condition:
\begin{equation}\label{scale-separation}
    \e_1 \to 0 \text{ and } \frac{\e_{k+1}}{\e_k} \to 0 \quad \text{for any } k \ge 1,
\end{equation}
where $\e = (\e_1,\e_2,\cdots)$ is understood as a sequence of elements in $(0,1)^\infty$ without explicit indices. Throughout the paper, for convenience, we will write the condition \eqref{scale-separation} simply as $\e \doublerightarrow 0$.

\begin{theorem}\label{thm.qualitative}
Let $\Omega$ be a bounded Lipschitz domain in $\R^d$. Assume \eqref{cond_ellipticity}, \eqref{cond_bounded} and \eqref{cond_periodicity}. Let $f\in H^{-1}(\Omega)$ and $g\in H^{1/2}(\partial\Omega)$. Then the solution $u_\varepsilon$ of \eqref{eq_infinite} converges to $u_0$ weakly in $H^1(\Omega)$ as $\varepsilon\doublerightarrow 0$, where $u_0$ is the solution to
  \begin{equation}
    \label{eq_homogenized}
    \begin{cases}
      -\mathrm{div}(\widehat{A}\nabla u_0)=f&\text{in }\Omega,\\
      u_0=g&\text{on }\partial\Omega,
    \end{cases}
  \end{equation}
  and $\widehat{A}=\widehat{A}(x)$ is the homogenized coefficient matrix given in Theorem \ref{qual-homog_thm_hatA}. Moreover, as $\varepsilon\doublerightarrow 0$, $A^\varepsilon(x)$ $H$-converges to $\widehat{A}$ in $\mathcal{M}(\mu, \mu^{-3}; \Omega)$ (see Section \ref{sec_compactness-H} for the definition of $H$-convergence). 
\end{theorem}

To establish the optimal convergence rates, we assume that $B_\ell$'s are Lipschitz continuous and decay at a certain rate as $\ell \to \infty$. Specifically, let $\delta = (\delta_0, \delta_1,\delta_2,\cdots) \in [0,\infty)^\infty$ and assume
\begin{align}\label{cond_B}
  \max\Big\{\|B_\ell\|_{L^\infty}, \max_{0\leq j\leq \ell}\|\nabla_{y_j}B_\ell\|_{L^\infty}\Big\}\leq \delta_\ell\quad\text{with}\quad [\delta]_0:=\sum_{\ell=0}^\infty \delta_\ell<\infty.
\end{align}
We will use the notation 
\begin{align}
  \label{def_norm_delta}
  [\delta]_\alpha:=\sum_{\ell=1}^\infty \ell^\alpha \delta_\ell\quad \text{ for } \alpha>  0.
\end{align}
The boundedness of $[\delta]_0$ in \eqref{cond_B} implies \eqref{cond_bounded} and
ensures that $A_n^\varepsilon(x)$ has a uniform limit $A^\varepsilon(x) $ that is continuous in $\Omega$.

\begin{theorem}\label{thm.L2rates}
    Let $\Omega$ be a bounded $C^{1,1}$ domain in $\R^d$. Assume \eqref{cond_ellipticity}, \eqref{cond_periodicity} and \eqref{cond_B}.
    Let $u_\e$ and $u_0$ be the weak solutions of \eqref{eq_infinite} and \eqref{eq_homogenized}, respectively.
    Suppose $ [\delta]_1 < \infty$. Then
        \begin{align}\label{est.rate1}
        \quad \|u_\varepsilon-u_0\|_{L^2(\Omega)}\leq C\Lambda [\delta]_1\sup_{k\geq 1} \frac{\varepsilon_k}{\varepsilon_{k-1}} \cdot\left(\|f\|_{L^2(\Omega)}+\|g\|_{H^{3/2}(\partial\Omega)}\right),
        \end{align}
        where $C$ depends only on $d, \mu, \Omega$ and $\Lambda = \Lambda_\delta$ is a constant depending increasingly on $[\delta]_0$ and $[\delta]_1$.

\end{theorem}

\begin{remark}\label{rmk.typicalcase}
    The convergence rate in Theorem \ref{thm.L2rates} recovers the optimal convergence rates in the case of finitely many scales. In fact, setting $\delta_k = 0$ for all $k > m$, we obtain the same convergence rate as in \cite{Niu20_reiterated} for $m$ scales. Moreover, the convergence rate in \eqref{est.rate1}
    is also optimal in some typical cases of infinitely many scales. In particular, if $\varepsilon_j=\varepsilon_1^j$ with 
    $\e_1 \in (0,1)$, we have
    \begin{align*}
        \sup_{k\geq 1} \frac{\e_k}{\e_{k-1}}=\varepsilon_1. 
    \end{align*}
\end{remark}

\begin{remark}
    The factor $[\delta]_1$ in \eqref{est.rate1} provides additional smallness for the error estimate. For example, consider the Weierstrass-type\footnote{The function in \eqref{W-F} is 
    continuous but differentiable nowhere for suitable choices of $B_k$, $\e_1$ and $\tau$.} functions,
    \begin{equation}\label{W-F}
        A^\e(x) = B_0(x) + \sum_{k =1}^\infty \tau^k B_k(x/\e_1^k),
    \end{equation}
    with $\| B_k\|_\infty + \| \nabla  B_k \|_\infty \le C$.
    Note that  $[\delta]_1 = \sum_{k=1}^\infty C k \tau^k \le C \tau $ (assume $\tau \le 1/2$), which indicates that the oscillating part of $A^\e$ is a relatively small perturbation of the regular principal part $B_0(x)$. By \eqref{est.rate1}, we see that 
    \[
    \|u_\varepsilon-u_0\|_{L^2(\Omega)} \le C\e_1 \tau \left(\|f\|_{L^2(\Omega)}+\|g\|_{H^{3/2}(\partial\Omega)}\right).
    \]
    The improved convergence rate of $O(\e_1 \tau)$ is a combined effect of small perturbation and homogenization. As we shall explain later, this observation is key to controlling the accumulated errors in our proof of uniform Lipschitz regularity
    for the equation \eqref{eq_infinite}.
\end{remark}

The next two theorems give the uniform Lipschitz estimates under a natural scale-separation condition.

\begin{theorem}\label{thm.IntLip}
    Let $\{A_n\}_{n\in \N}$ satisfy the same assumptions as in Theorem \ref{thm.L2rates}. Also assume that  $[\delta]_{1+\rho}<\infty$ for some $\rho>0$. Let $u_\varepsilon$ be a solution to 
    \begin{equation*}
        -\mathrm{div}(A^\varepsilon(x)\nabla u_\varepsilon)=f \quad \text{in } B_R,
    \end{equation*}
    where $0<R\leq 1$ and $f\in L^p(B_{R})$ with $p>d$. Suppose that there exists $N \ge 1$ such that
  \begin{align}\label{int-lip_cond_separation_ratio}
    \bigg(\frac{\varepsilon_{k+1}}{\varepsilon_{k}}\bigg)^N\leq \frac{\varepsilon_{i+1}}{\varepsilon_{i}}, \quad \text{for any } k > i \ge 0.
  \end{align}
    Then there exists a constant $\nu>0$, depending only on $d, \mu, p, \rho, N$, $[\delta]_0 $ and $ [\delta]_1$, such that whenever $\varepsilon_1\leq \nu$, we have
  \begin{align}\label{int-lip_es_Lip}
    \| \nabla u_\e \|_{L^\infty(B_{R/2})} \leq C\bigg\{\bigg(\fint_{B_{R}}|\nabla u_\varepsilon|^2\bigg)^{1/2}+R\bigg(\fint_{B_{R}}|f|^p\bigg)^{1/p}\bigg\},
  \end{align}
where $C$ depends only on $d, \mu, p, [\delta]_0, \rho,$ and $ [\delta]_{1+\rho}$.
\end{theorem}

Let $\Omega$ be a $C^{1,\gamma}$ domain and $0\in \partial \Omega$. Let $D_R := \Omega \cap B_R(0)$ and $\Delta_R := \partial \Omega \cap B_R(0)$. Let $L$ be the $C^{1,\gamma}$ character of $\Omega$ (see \eqref{ext_cond_psi}).
\begin{theorem}\label{thm.BdryLip}
    Let $\{ A_n \}_{n\in \N}, \delta = (\delta_0, \delta_1,\cdots)$ and $\e = (\e_1, \e_2, \cdots )$ satisfy the same assumptions as in Theorem \ref{thm.IntLip}.
    Let $u_\varepsilon$ be a solution to 
    \begin{equation}\label{eq_bdryint}
    \begin{cases}
        -\mathrm{div}(A^{\varepsilon}(x)\nabla u_\varepsilon)=f&\text{in } D_R,\\
        u_\varepsilon=g& \text{on } \Delta_R,
   \end{cases}
\end{equation}
    where $0<R\leq 1$, $f\in L^p(B_{R})$ with $p>d$, $g\in C^{1, \gamma}(\Delta_R)$. Then there exists a constant $\nu>0$, depending only on $d, \mu, p, \rho, N, [\delta]_0, [\delta]_1, \gamma $ and $L$, such that whenever $\varepsilon_1\leq \nu$,  we have
  \begin{align*}
    \| \nabla u_\e \|_{L^\infty(D_{R/2})} \leq C\bigg\{\bigg(\fint_{D_{R}}|\nabla u_\varepsilon|^2\bigg)^{1/2}+R\bigg(\fint_{D_{R}}|f|^p\bigg)^{1/p}+R^{-1}\|g\|_{C^{1, \gamma}(\Delta_R)}\bigg\},
  \end{align*}
  where $C$ depends only on $d, \mu, p, [\delta]_0, \gamma, L, \rho$ and $[\delta]_{1+\rho}$. 
\end{theorem}

\begin{remark}\label{rmk.smallnu}
    The quantitative scale-separation condition \eqref{int-lip_cond_separation_ratio} is essentially the same as in \cite{Niu20_reiterated} for finitely many scales. However,
    for the case of infinite many scales, we need an additional  assumption $\e_1 < \nu$, which, together with \eqref{int-lip_cond_separation_ratio}, implies 
    $(\e_{k+1}/\e_k) < \nu^{1/N}$ for all $k \ge 1$. This is due to a technical reason in our approach, and actually ill-behaved examples can be constructed; see Remark \ref{rmk.illExample}. Note that in view of the rescaling property of the equation, $\e_1 = \e_1/\e_0 < \nu$ can be replaced by $\e_{m}/\e_{m-1} < \nu$, where $m$ is the smallest integer satisfying this condition (i.e., the first $m-1$ scales are not separated well). In this case, the uniform Lipschitz estimate can be established below $\e_m$-scale with $\e_m \ge \nu^m$, and thus a full-scale estimate depends additionally on the number $m$. If no such $m$ exists, then $\e_k \ge \nu^k$ for all $k \ge 1$ and Lemma \ref{int-lip_thm_Dini} yields the Lipschitz estimate.
    
    On the other hand, the smallness assumption of $\e_1$ can be removed in some interesting cases. In fact, if there exists $N \ge 1$ such that
    \begin{equation*}
        \bigg(\frac{\varepsilon_{k+1}}{\varepsilon_{k}}\bigg)^N\leq \frac{\varepsilon_{i+1}}{\varepsilon_{i}} \le \bigg(\frac{\varepsilon_{k+1}}{\varepsilon_{k}}\bigg)^{1/N}, \quad \text{for any } k > i \ge 0,
    \end{equation*}
    then Theorems \ref{thm.IntLip} and \ref{thm.BdryLip} hold without the assumption $\e_1 < \nu$; see Theorem \ref{thm.typicalLipEst} and Remark \ref{rmk.typical}. This, of course, includes the case $\e_j = \e_1^j$ for any $\e_1 \in (0,1)$.
\end{remark}

\begin{remark}
    For the uniform Lipschitz estimates, it is possible to weaken the Lipschitz regularity assumption \eqref{cond_B} on the coefficients to the H\"{o}lder continuity by an approximation argument. However, we will not pursue this generality to avoid additional complexity.
\end{remark}

\subsection{Key ideas in the proofs}

Our proofs are  based on reiterated homogenization argument --- homogenizing the coefficients in order from small scales to large scales --- which at a high level, resembles the renormalization group or coarse-graining argument in the deterministic and nondegenerate setting. But since the scales tend to zero, there is no smallest scale for the original equation \eqref{eq_infinite}. The natural starting point is the approximate equation with coefficient matrix $A_n^\e(x)$ for arbitrarily large $n$:
\begin{equation}
  \label{eq_approximate}
  \begin{cases}
    -\mathrm{div}(A_n^{\varepsilon}(x)\nabla u_{n, \varepsilon})=f&\text{in }\Omega,\\
        u_{n, \varepsilon}=g& \text{on }\partial\Omega,
  \end{cases}
\end{equation}
where $A_n = A_n(y_0,y_1,\cdots, y_n)$ is the partial sum given by \eqref{def_An} and $A_n^\e(x) = A_n(x,x/\e_1,\cdots, x/\e_n)$.
Since this problem has a finite number of scales,  a classical reiterated homogenization argument will work and we have $A_n^\e(x)$ $H$-converges to some homogenized matrix $\widehat{A}_n$ and $u_{n,\e}$ converges to $u_{n,0}$ weakly in $H^1(\Omega)$ as $\e \doublerightarrow 0$, where $u_{n,0}$ is the weak solution of
\begin{align}\label{eq_approx_homo}
  \begin{cases}
    -\mathrm{div}(\widehat{A}_n(x)\nabla u_{n, 0})=f&\text{in }\Omega, \\
    u_{n,0}=g &\text{on }\partial\Omega. 
  \end{cases}
\end{align}
Moreover, the stability of $H$-convergence and the assumption \eqref{cond_bounded} imply that $\widehat{A}_n$ converges in $L^\infty(\Omega)$ to a limiting matrix, denoted by $\widehat{A}$, as $n\to \infty$, which is defined as the homogenized matrix for the original problem \eqref{eq_infinite}. Thus, the qualitative homogenization theorem can be proved by a procedure demonstrated in the following diagram.
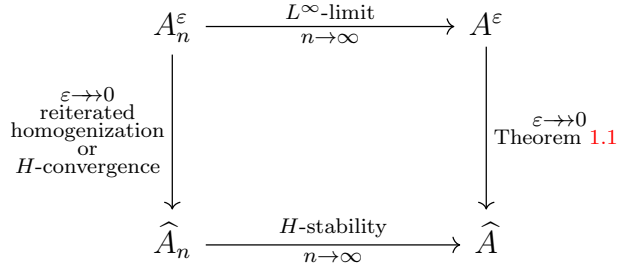
\begin{figure}[h]
    \centering
    \begin{tikzcd}[column sep=8em, row sep=5em]
    A_n^\e \arrow[r, " L^\infty\text{-limit}", " n\to \infty"'] \arrow[d, " \substack{ {\e \doublerightarrow 0}\\  \text{reiterated} \\  \text{homogenization} \\  \text{or} \\  H\text{-convergence}  }"'] 
    & A^\e \arrow[d, "\substack{ \e \doublerightarrow 0 \\  \text{Theorem \ref{thm.qualitative} }}"] \\
    \widehat{A}_n \arrow[r, " H\text{-stability}"," n\to \infty"'] 
    & \widehat{A}
    \end{tikzcd}
    \caption{Qualitative homogenization}
    \label{fig:Qualitative}
\end{figure}

\noindent\textbf{Quantitative convergence rates.}
As for qualitative homogenization, to establish the quantitative convergence rates, we also start by considering the approximate equation \eqref{eq_approx_homo} with arbitrary $n$ scales. The key is to establish a convergence rate that is independent of the number of scales $n$, which allows us to take the limit as $n\to \infty$ to obtain the convergence rate for the original equation with infinitely many scales.
This is by no means a direct consequence of the classical reiterated homogenization method; new ideas are required, as we now explain below.

Let us first mention that, the reiterated homogenization process can be broken down into a sequence of $H$-convergence of intermediate matrices, homogenizing the smallest scale in each step, i.e.,
\begin{equation}\label{ReteratedHomo-H-conv}
\begin{aligned}
    A_n(x,x/\e_1,\cdots, x/\e_n) & \xrightarrow[\e_n\to 0]{H} A_n^{n-1}(x,x/\e_1,\cdots, x/\e_{n-1}) \\
    & \xrightarrow[\e_{n-1}\to 0]{H} A_n^{n-2}(x,x/\e_1,\cdots, x/\e_{n-2}) \\
    & \qquad \cdots \\
    & \xrightarrow[\e_1\to 0]{H} A^0_n(x) = \widehat{A}_n(x).
\end{aligned}
\end{equation}
Our method will take advantage of the special structure of the multiscale coefficients $A_n^\e(x)$, written as
\begin{equation*}
    A_n^\e(x) = A_{n-1}\Big(x,\frac{x}{\e_1},\cdots, \frac{x}{\e_{n-1}}\Big) + B_n\Big(x, \frac{x}{\e_1}, \cdots, \frac{x}{\e_{n-1}}, \frac{x}{\e_n} \Big).
\end{equation*}
The key observation is that the first term on the right-hand side, called the principal part, oscillates relatively slowly with a large amplitude; while the second term, called the tail part, oscillates relatively faster with a smaller amplitude. This structure appears not only in $A_n^\e(x)$, but also in every intermediate matrix $A_n^{k,\e}(x), 1\le k\le n-1$, taking the form of
\begin{equation*}
    A_n^{k,\e}(x) = A_{k-1}\Big(x,\frac{x}{\e_1},\cdots, \frac{x}{\e_{k-1}}\Big) + B^n_k\Big(x, \frac{x}{\e_1}, \cdots, \frac{x}{\e_{k-1}}, \frac{x}{\e_n} \Big),
\end{equation*}
where $\{B^n_k\}_{1\le k \le n-1}$ are given by the backward recursive equations \eqref{iden_Blk} and \eqref{expression_Ank}. Using the $H$-stability and a careful estimate of the recursive equations, we are able to show that 
\begin{equation*}
    \| B^n_k \|_{L^\infty} + \sup_{1\le j\le k}\| \nabla_{y_j} B^n_k \| \le C \sum_{\ell = k}^n \delta_\ell;
\end{equation*}
see Lemma \ref{qual-homog_lemma_bkn} and Lemma \ref{conver_lemma_dkn}.
Under a certain decay assumption on $\delta_k$ (say, $[\delta]_{1} < \infty$), we see that $B^n_k$, though it oscillates faster, is a relatively small perturbation to the principal part $A_{k-1}$, for $k$ large enough. Due to this generic structure, it is natural to consider a base situation with only two scales $\e_0 = 1$ and $\e \ll 1$. In fact, we consider the coefficient matrix,
\begin{equation*}
    A(x, x/\e) = E_0(x) + E_1(x,x/\e),
\end{equation*}
where $E_0(x)$ is the principal part with slow oscillation and large amplitude, and $E_1(x,x/\e)$ is the tail part with fast oscillation and small amplitude. Let $u_\e$ and $u_0$ be the weak solutions corresponding to the coefficient matrix $A(x,x/\e)$ and its homogenized matrix $\widehat{A}(x)$; see \eqref{eq_n=1_general} and \eqref{eq_n=1.homo}. In this case, we prove an improved error estimate that takes into account of the smallness of $E_1$:
\begin{equation*}
    \| u_\e - u_0 \|_{L^2(\Omega)} \le C \e (\| E_1 \|_{L^\infty} + \| \nabla_x E_1 \|_{L^\infty}) + \text{similar terms};
\end{equation*}
see Theorem \ref{thm_n=1_general} for the precise statement.
This improvement - exploiting the smallness of $E_1$
  in the one-step homogenization - allows us to carry out the reiterated homogenization argument as in  \eqref{ReteratedHomo-H-conv} with accumulated errors that can be controlled uniformly in $n$ under the additional assumption $[\delta]_{\alpha} <\infty$ for some $\alpha > 0$; see Theorem \ref{conver_thm_rate_suboptimal}. It turns out that the optimal $L^2$ convergence rate can be derived if $\alpha = 1$,  as in Theorem \ref{thm.L2rates}.





\noindent\textbf{Uniform Lipschitz estimates.}
The proof of the uniform Lipschitz estimates proceeds largely within the framework of \cite{S17}, a structured refinement of \cite{AS16} based on the $L^2$ convergence rates. 
Since the essential difficulty is not reduced by considering the approximate problem \eqref{eq_approx_homo} and establishing the Lipschitz estimates uniform in $n$, we will work directly on the original problem with infinitely many scales. 

In multiscale homogenization, the uniform Lipschitz estimates are established through an iterative argument, in contrast to the proof of convergence rate, proceeding from the largest scale down to the smaller scales. To illustrate our main ideas, we take the simplified setting $\mathrm{div}(A^\e \nabla u_\e) = 0$ in $B_1$. Recall that in \cite{Niu20_reiterated}, under appropriate scale-separation conditions, it was shown that for any $r\in (\e_{k+1}, \e_k)$
\begin{equation}\label{est.largeLip}
    \bigg( \fint_{B_{r}} |\nabla u_\e|^2 \bigg)^{1/2} \le C \bigg( \fint_{B_{\e_k}} |\nabla u_\e|^2 \bigg)^{1/2},
\end{equation}
where $C$ is a large constant. Iterating this estimate, we obtain
\begin{equation*}
    \bigg( \fint_{B_{\e_n}} |\nabla u_\e|^2 \bigg)^{1/2} \le C^n \bigg( \fint_{B_{1}} |\nabla u_\e|^2 \bigg)^{1/2}.
\end{equation*}
This large-scale estimate is satisfactory for a fixed number of scales. However, for infinitely many scales considered in the present paper, as $n\to \infty$ and $\e_n \to 0$, the constant $C^n \to \infty$ and the previous method fails. The failure indicates that an essential modification is needed for infinitely many scales.

Recall the well-known quantity
\begin{equation}\label{def.Hr}
    H(r) = H(r; u_\varepsilon) := \frac{1}{r}\inf_{P\in \mathcal{P}}\bigg(\fint_{B_r}|u_\varepsilon-P|^2\bigg)^{1/2},
\end{equation}
where $\mathcal{P}$ denotes the linear space of affine functions. This quantity, measuring how flat a function is, is frequently used to establish the large-scale Lipschitz estimates in elliptic homogenization. Let $h(r) = |\nabla P_r|$ with $P_r \in \mathcal{P}$ being the minimizer in \eqref{def.Hr}. Roughly speaking,  $h(r)$ is comparable to the $L^2$ average of $|\nabla u_\e|$ in $B_r$ and generally larger than $H(r)$. In our situation, these two quantities combined can be used to establish \eqref{est.largeLip} between two successive scales. The new idea here is  establishing an iterative formula for $H(r)$ and $h(r)$, across all scales, with improved errors again due to the special structure of the coefficients. More precisely, under certain scale-separation conditions, we show that for any $r \in [\e_{m+1}, \e_m]$,
\begin{equation}\label{eq.Hrhr}
    \begin{cases}
        \displaystyle H(r) \le C_0 \Big( \frac{r}{\e_m} \Big)^\lambda H(\e_m) + C_1 M_m \big\{ H(\e_m) + h(\e_m) \big\}, \\
        \displaystyle |h(r) - h(\e_m)| \le C_0 H(\e_m) + C_1 M_m \big\{ H(\e_m) + h(\e_m) \big\},
    \end{cases}
\end{equation}
where $\lambda \in (0,1), C_0, C_1$ are constants independent of $\e$ and $m$. The constants $\{M_m \}_{m\in \N}$ satisfy
\begin{align*}
    M_m \le \max\bigg\{\sum_{j=0}^{m}\frac{\varepsilon_m}{\varepsilon_j}\sum_{k=j}^{\infty}\delta_k, \sum_{j=0}^m\frac{\varepsilon_m}{\varepsilon_j}\cdot\sum_{k=1}^\infty k^\alpha\delta_{k+m}\bigg\}.
\end{align*}
For large $m$, the extra factor $M_m$ becomes small, ensuring $\sum_{m = 1}^\infty M_m < \infty$, under certain scale-separation condition on $\e$ and decay condition on $\delta$. This crucial improvement actually follows from the extra factor $[\delta]_\alpha$ in the $L^2$ convergence rate of Theorem \ref{conver_thm_rate_suboptimal} as well as  the rescaling properties that gradually shift the oscillation to the small tail part of the coefficients. Finally, with $\sum_{m = 1}^\infty M_m < \infty$, we can iterate \eqref{eq.Hrhr} as $m\to \infty$ and obtain  the uniform boundedness of $H(r) + h(r)$ for all $r\in (0,1)$.

The remainder of the paper is organized as follows.  Section 2 addresses the qualitative homogenization of the Dirichlet problem \eqref{eq_infinite} and
gives the proof of Theorem \ref{thm.qualitative}.
Section 3 is devoted to the convergence rates and contains the proof of Theorem \ref{thm.L2rates}.
The interior Lipschitz estimates (Theorem \ref{thm.IntLip}) are established in Section 4, and
 the boundary Lipschitz estimates (Theorem \ref{thm.BdryLip}) are proved in Section 5.

 \textbf{Acknowledgement.} Y. Xu is partially supported by NNSF of China (No. 12201604, 12371106). J. Zhuge is partially supported by NNSF of China (No. 12494541, 12288201, 12471115).

\section{Qualitative homogenization}\label{sec_qualitative-homogenization}

\subsection{Reiterated homogenization}\label{sec_approximate-problems}
As discussed in the introduction, we shall start with the approximate problem \eqref{eq_approximate} for arbitrary but fixed $n \ge 1$.
In this subsection, we review the method of reiterated homogenization to analyze \eqref{eq_approximate} and provide some formal calculations adapted to our settings that will be referenced throughout the subsequent sections. 



By the method of reiterated homogenization, the homogenized matrix $\widehat{A}_n(x)$ under $\e \doublerightarrow 0$ is obtained by repeatedly homogenizing one scale at a time in order from the smallest scale to the largest. 
Precisely, we first fix $\e_1,\e_2,\cdots, \e_{n-1}$ in \eqref{eq_approximate} and let $\e_n \to 0$. Then the fastest variable of equation \eqref{eq_approximate} will be homogenized and the solution $u_{n,\e} \to u^{n-1}_{n,\e}$ weakly in $H^1(\Omega)$, where $u^{n-1}_{n,\e}$ is the weak solution of
\begin{equation}\label{eq.approx.n-1}
    \begin{cases}
    -\mathrm{div}(A_n^{n-1,\varepsilon}(x)\nabla u^{n-1}_{n, \varepsilon})=f&\text{in }\Omega,\\
        u^{n-1}_{n, \varepsilon}=g& \text{on }\partial\Omega,
  \end{cases}
\end{equation}
 $A_n^{n-1,\varepsilon}(x) = A^{n-1}_n(x,x/\e_1,\cdots, x/\e_{n-1})$, and $A^{n-1}_n(y_0, y_1, \cdots, y_{n-1})$ is the $(n-1)$-scale intermediate matrix, given by
\begin{align}\label{eq.Ann-1}
\begin{aligned}
    & A_n^{n-1}(y_0, y_1, \cdots, y_{n-1}) \\
    & =\fint_{\mathbb{T}^d} \Big\{ A_n(y_0, \cdots, y_n)+A_n(y_0, \cdots, y_n)\nabla_{y_n}\chi_{A_n}(y_0, \cdots, y_n) \Big\} dy_n.
\end{aligned}
\end{align}
Here $\chi_{A_n} = (\chi_{A_n}^j)_{1\le j\le d}$ are the correctors satisfying the equations w.r.t. the variable $y_n$,
\begin{align*}
  \begin{cases}
    -\mathrm{div}_{y_n}(A_n\nabla_{y_n} \chi_{A_n}^j)=\mathrm{div}_{y_n}(A_ne^j)=\mathrm{div}_{y_n}(B_ne^j)\quad \text{in }\mathbb{T}^d,\\
  \chi_{A_n}^j \text{ is $1$-periodic in }y_n \text{ and }\int_{\mathbb{T}^d}\chi_{A_n}dy_n=0.
  \end{cases}
\end{align*}
By \eqref{eq.Ann-1} and \eqref{def_An}, and a simple calculation,
\begin{equation}\label{def_Ank}
\aligned
    & A_n^{n-1}(y_0, y_1, \cdots, y_{n-1})\\
    & = B_0(y_0)+B_1(y_0, y_1)+\cdots+B_{n-2}(y_0, \cdots, y_{n-2})\\
    & \qquad + B_{n-1}(y_0, \cdots, y_{n-1})+\fint_{\mathbb{T}^d}B_ndy_n+\fint_{\mathbb{T}^d}B_n\nabla_{y_n}\chi_{A_n} dy_n.
    \endaligned
\end{equation}
Recall that the matrix $A_n^{n-1}$ satisfies the ellipticity, boundedness and periodicity conditions as in \eqref{cond_ellipticity}--\eqref{cond_periodicity}. 


Now the equation \eqref{eq.approx.n-1} has $n-1$ scales, and we apply the same procedure to homogenize the smallest scale $\e_{n-1}$. We obtain the $(n-2)$-scale intermediate matrix in a form of $A_n^{n-2, \e}(x) = A_n^{n-2}(x,x/\e_1,\cdots,x/\e_{n-2})$ and repeat this procedure. In general, for $0\le k \le n-1$, the $k$-scale intermediate matrix is derived via the backward recursive equations:
\begin{align}\label{iden_Blk}
\begin{split}
B_{k}^n(y_0, \cdots, y_{k})&=B_{k}(y_0, \cdots, y_{k})+\fint_{\mathbb{T}^d}B_{k+1}^n(y_0, \cdots, y_{k+1})dy_{k+1}\\&+\fint_{\mathbb{T}^d}B_{k+1}^n(y_0, \cdots, y_{k+1})\nabla_{y_{k+1}} \chi_{A_n^{k+1}}(y_0, \cdots, y_{k+1})dy_{k+1},
\end{split}
\end{align}
and
\begin{align}\label{expression_Ank}
  A_n^{k}(y_0, \cdots, y_{k})=\sum_{\ell=0}^{k-1} B_\ell(y_0, \cdots, y_\ell)+B_{k}^{n}(y_0, \cdots, y_{k}), 
\end{align}
where we have used the convention $B^n_n=B_n$, $A_n^n=A_n$ (to include the case $k = n-1$), and $\chi_{A_n^{k+1}} = (\chi_{A_n^{k+1}}^j)_{1\le j\le d}$ are the corresponding correctors corresponding to $A_n^{k+1}$ w.r.t. the variable $y_{k+1}$:
\begin{align}\label{eq.chi.k+1}
  \begin{cases}
    -\mathrm{div}_{y_{k+1}}(A_n^{k+1}\nabla_{y_{k+1}} \chi_{A_n^{k+1}}^j) = \mathrm{div}_{y_{k+1}}(A_n^{k+1} e^j)=\mathrm{div}_{y_{k+1}}(B^n_{k+1} e^j)\quad \text{in }\mathbb{T}^d,\\
  \chi_{A_n^{k+1}}^j \text{ is $1$-periodic in }y_{k+1} \text{ and }\int_{\mathbb{T}^d}\chi_{A_n^{k+1}}^j dy_{k+1} = 0.
  \end{cases}
\end{align}
The final matrix $A^0_n(y_0)$ is the homogenized matrix obtained by the method of reiterated homogenization, and we write $\widehat{A}_n(y_0)=A_n^0(y_0)$.

It was proved in \cite{Allaire1996Multiscale, Niu20_reiterated} that if $\e \doublerightarrow 0$, then
$u_{n, \varepsilon}$ converges weakly to $u_{n, 0}$ in $H^1(\Omega)$ and $A_n^\varepsilon \nabla u_{n, \varepsilon}$ converges weakly to $\widehat{A}_n\nabla u_{n, 0}$ in $L^2(\Omega; \mathbb{R}^d)$, where $u_{n, 0}$ is the solution of \eqref{eq_approx_homo}.
In this qualitative reiterated homogenization theorem, we only require that $\Omega$ is a bounded Lipschitz domain, $f\in H^{-1}(\Omega)$ and $g\in H^{1/2}(\partial\Omega)$.

\subsection{\texorpdfstring{Compactness of $H$-convergence}{Compactness of H-convergence}}\label{sec_compactness-H}

An effective way to universally control the size of the intermediate and homogenized matrices is using the notion of $H$-convergence (or $G$-convergence in the case of symmetric coefficient matrices).
Let $\Omega\subset \R^d$ be a bounded domain. For $0<\alpha\le \beta < \infty$, let $\mathcal{M}(\alpha,\beta;\Omega)$ denote the set of $A\in L^\infty(\Omega;\R^{d\times d})$ that satisfy
\begin{equation}\label{qual-homog_def_Hconvergence}
	A(x)\xi\cdot \xi \ge \alpha |\xi|^2,\quad A^{-1}(x)\xi\cdot \xi \ge \beta^{-1}|\xi|^2, \quad \text{a.e. } x\in \Omega.
\end{equation}

\begin{definition}[\cite{T09}]\label{H-convergence}
	We say a sequence $\{A_k\}_{k\ge 1} \subset \mathcal{M}(\alpha,\beta;\Omega)$ $H$-converges to $\bar{A}\in \mathcal{M}(\alpha',\beta';\Omega)$ (denoted by $A_k \xrightarrow[k\to \infty]{H} \bar{A}$) for some $0<\alpha'\le \beta'<\infty$, if for all $f\in H^{-1}(\Omega)$, the sequence of solutions $\{u_k\}_{k\ge 1} \subset H^1_0(\Omega)$ of $-\mathrm{div} (A_k \nabla u_k) = f$ converges to $\bar{u}$ weakly in $H^1_0(\Omega)$; and the sequence $A_k \nabla u_k$ converges to $\bar{A} \nabla \bar{u}$ weakly in $L^2(\Omega;\R^d)$, where $\bar{u}$ is the weak solution of $-\mathrm{div} (\bar{A} \nabla \bar{u}) = f$ in $\Omega$.
\end{definition}

Recall that the $H$-limit of an $H$-converging sequence is unique. Also observe that the condition \eqref{qual-homog_def_Hconvergence} in the class $\mathcal{M}(\alpha, \beta; \Omega)$ is equivalent to the condition \eqref{cond_ellipticity} and \eqref{cond_bounded-A}. Indeed, if $A\in \mathcal{M}(\alpha, \beta; \Omega)$, by the Cauchy-Schwarz inequality, $\beta^{-1}|\xi|^2 \le \xi\cdot A^{-1}\xi \le |\xi| |A^{-1}\xi|$ for a.e. $x\in\Omega$. Thus, $|A^{-1}\xi| \ge \beta^{-1}|\xi|$ for any $\xi$. Replacing $\xi$ by $A\xi$ in this inequality, we get $|A\xi| \le \beta |\xi|$, which gives the upper bound of $A$. Conversely, if $A$ satisfies \eqref{cond_ellipticity} and \eqref{cond_bounded-A}, then we can show $\xi \cdot A^{-1}\xi \ge \mu^3 |\xi|^2$ and thus $A \in \mathcal{M}(\mu, \mu^{-3}; \Omega)$.
 
The importance of the class $\mathcal{M}(\alpha,\beta;\Omega)$ is reflected by its compactness under $H$-convergence.
\begin{lemma}[{\cite[Theorem 6.5]{T09}}]\label{case-n=1_lem_Hcompact}
  For any sequence $\{A_k\}\subset \mathcal{M}(\alpha, \beta; \Omega)$ there exists a subsequence $\{A_{k_\ell}\}$ and an element $\bar{A}\in \mathcal{M}(\alpha, \beta; \Omega)$ such that $A_{k_\ell}$ $H$-converges to $\bar{A}$. 
\end{lemma}
This compactness can be used to establish the uniform upper and lower bounds of the intermediate and homogenized coefficient matrices at each reiterated step. 
Consider the coefficient matrix $A_n(y_0, y_1,\cdots, y_{n-1}, x/\e_n)$ as a function of $x \in \R^d$, where $(y_0, y_1, \cdots, y_{n-1})\in \Omega\times \mathbb{R}^{d(n-1)}$ are fixed parameters. We know from \eqref{cond_ellipticity}, \eqref{cond_bounded-A}, and the previous argument that 
$$
A_n(y_0, y_1, \cdots, y_{n-1}, x/\varepsilon_n)\in \mathcal{M}(\mu, \mu^{-3}) \quad \text{for each }y_0, y_1, \cdots, y_{n-1}.
$$
Hereafter we write $\mathcal{M}(\alpha, \beta) = \mathcal{M}(\alpha, \beta; \Omega)$ if the coefficient matrices are defined in the entire $\R^d$ and drop in every $\mathcal{M}(\alpha, \beta; \Omega)$ for any bounded $\Omega \subset \R^d$. As $\e_n \to 0$, by the classical homogenization theorem with only one oscillating scale, it holds
\[
A_n(y_0, y_1, \cdots, y_{n-1}, x/\e_n) \xrightarrow[\e_n\to 0]{H} A_n^{n-1}(y_0, y_1, \cdots, y_{n-1}),
\]
where $A_n^{n-1}$ is given by \eqref{eq.Ann-1}.
Therefore, by Lemma \ref{case-n=1_lem_Hcompact}, we have
\begin{align*}
  A_n^{n-1}(y_0, y_1, \cdots, y_{n-1})\in \mathcal{M}(\mu, \mu^{-3})\quad\text{for each } y_0, y_1, \cdots, y_{n-1}.
\end{align*}
Then, by setting $y_{n-1}=x/\varepsilon_{n-1}$, we get $$A_n^{n-1}(y_0, y_1, \cdots, y_{n-2}, x/ \varepsilon_{n-1})\in \mathcal{M}(\mu, \mu^{-3})\quad \text{for each }y_0, y_1, \cdots, y_{n-2}. $$
This allows us to apply the same argument repeatedly to conclude that for $0\leq k\leq n$
\begin{align}\label{es_Ank_bound}
  A_n^k(y_0, \cdots, y_k)\in \mathcal{M}(\mu, \mu^{-3}) \quad \text{for each }y_0, \cdots, y_k. 
\end{align}
In other words, each $A_n^k$, including $\widehat{A}_n=A_n^0$, satisfies the ellipticity and boundedness conditions with the constant depending only on $\mu$. 
 In particular, by \eqref{expression_Ank},
 \begin{align}
   \label{es_Bkn_bound}
 \|B_k^n\|_{L^\infty} = \|A^k_n - A_{k-1}\|_{L^\infty} \le \|A^k_n \|_{L^\infty} + \|A_{k-1}\|_{L^\infty} \leq C, 
 \end{align}
where $C$ depends only on $\mu$ (thus uniform in $n$ and $k$).

\subsection{\texorpdfstring{Stability of $H$-convergence}{Stability of H-convergence}}\label{sec_stability-H}

In this subsection, we use the stability result of $H$-convergence, demonstrated in \cite[etc.]{Boccardo1982, Ene1997, T09}, to establish the convergence of $\widehat{A}_n$ as $n\to \infty$, as well as their Lipschitz regularity uniform in $n$.

\begin{lemma}[{\cite[Lemma 10.9]{T09}}]\label{qual-homog_lem_stability}
  Let $\{A_k\}_{k\geq 1}$ and $\{A_k'\}_{k\geq 1}$ be two sequences of matrices in $\mathcal{M}(\alpha, \beta; \Omega)$ such that for each $k\geq 1$,
  \begin{align*}
    \|A_k-A_k'\|_{L^\infty(\Omega)}\leq \tau\quad \text{for some }\tau>0.
  \end{align*}
  Suppose that $A_k$ $H$-converges to $\bar{A}$ and $A_k'$ $H$-converges to $\bar{A}'$ as $k\rightarrow \infty$. Then
  \begin{align*}
    \|\bar{A}-\bar{A}'\|_{L^\infty(\Omega)}\leq C\tau,
  \end{align*}
where $C = \beta/\alpha$. 
\end{lemma}

\begin{lemma}\label{qual-homog_lem_difference}
    Assume \eqref{cond_ellipticity}--\eqref{cond_periodicity}.
    Let $\widehat{A}_n$ denote the homogenized matrix in \eqref{eq_approx_homo} derived through reiterated homogenization. Then there exists a constant $C>0$, depending only on $\mu$, such that for any $n,m \in \N$, $n< m$
    \begin{equation}\label{est.hatAn-Am}
        \|\widehat{A}_n- \widehat{A}_m\|_{L^\infty(\Omega)} \le C \sum_{k=n+1}^m \|B_k\|_{L^\infty}.
     \end{equation}
    If in addition assume
    \begin{equation}\label{cond_Dy0B}
        \sum_{k=0}^\infty\|\nabla_{y_0}B_k\|_{L^\infty}<\infty,
    \end{equation}
    then for any $n\in\mathbb{N}$ and any $x_1, x_2\in \Omega$,
    \begin{align}\label{est.hatAn-Lip}
        |\widehat{A}_n(x_1)-\widehat{A}_n(x_2)|\leq C\sum_{k=0}^n\|\nabla_{y_0}B_k\|_{L^\infty}|x_1-x_2|. 
    \end{align}
\end{lemma}
\begin{proof}

By the reiterated homogenization described as in Section \ref{sec_approximate-problems}, we have $A_n^\varepsilon\xrightarrow[\varepsilon\doublerightarrow 0]{H} \widehat{A}_n$ and $A_m^\varepsilon\xrightarrow[\varepsilon\doublerightarrow 0]{H} \widehat{A}_m$. Note that by \eqref{cond_ellipticity} and \eqref{cond_bounded-A}, both $A_n^\varepsilon$ and $A_m^\varepsilon$ belong to $\mathcal{M}(\mu, \mu^{-3}; \Omega)$. Due to \eqref{def_An}, 
\begin{align*}
  \|A_n^\varepsilon-A_m^\varepsilon\|_{L^\infty(\Omega)}\leq \sum_{k=n+1}^m\|B_k\|_{L^\infty}. 
\end{align*}
By Lemma \ref{qual-homog_lem_stability}, this yields \eqref{est.hatAn-Am}. 

For the estimate \eqref{est.hatAn-Lip}, we set $y_0$ in $A_n(y_0, \cdots, y_n)$ to be $x_1$ and $x_2$ respectively, and denote
\begin{align*}
  A_{n, i}^\varepsilon(x)=A_n(x_i, x/\varepsilon_1, \cdots, x/\varepsilon_n),\quad i=1, 2. 
\end{align*}
In this situation, by the process of reiterated homogenization,
\begin{align*}
  A_{n ,i}^\varepsilon(x)\xrightarrow[\varepsilon\doublerightarrow 0]{H} \widehat{A}_n(x_i),
\end{align*}
and $A_{n, i}^\varepsilon\in \mathcal{M}(\mu, \mu^{-3})$. Observe that
\begin{align*}
  |A_{n, 1}^\varepsilon(x) - A_{n, 2}^\varepsilon(x) |\leq \|\nabla_{y_0}A_n\|_{L^\infty}|x_1-x_2|\leq \sum_{k=0}^n\|\nabla_{y_0}B_k\|_{L^\infty}|x_1-x_2|. 
\end{align*}
Therefore, Lemma \ref{qual-homog_lem_stability} implies \eqref{est.hatAn-Lip}. 
\end{proof}

\begin{theorem}\label{qual-homog_thm_hatA}
    Assume \eqref{cond_ellipticity}--\eqref{cond_periodicity}. Then there exists a homogenized matrix $\widehat{A} \in \M(\mu,\mu^{-3};\Omega)$ such that $\widehat{A}_n \to \widehat{A}$ uniformly in $\Omega$ as $n\to \infty$, and
    \begin{equation}\label{est.hatAn-hatA}
        \| \widehat{A}_n - \widehat{A} \|_{L^\infty(\Omega)} \le C \sum_{k=n+1}^\infty \|B_k\|_{L^\infty},
    \end{equation}
where $C$ is the constant given in Lemma \ref{qual-homog_lem_difference}.   
Moreover, if in addition \eqref{cond_Dy0B} holds,
then $\widehat{A}$ is Lipschitz in $\Omega$ and
    \begin{align}\label{est.hatA-Lip}
      |\widehat{A}(x_1)-\widehat{A}(x_2)|\leq C\sum_{k=0}^\infty\|\nabla_{y_0}B_k\|_{L^\infty}|x_1-x_2|. 
    \end{align}
  \end{theorem}
\begin{proof}
    According to Lemma \ref{qual-homog_lem_difference}, $\{\widehat{A}_n\}_n$ is a Cauchy sequence in $L^\infty(\Omega)$ under our assumption, whose limit, denoted by $\widehat{A}$, belongs to $\M(\mu,\mu^{-3};\Omega)$ as $\widehat{A}_n\in \M(\mu,\mu^{-3};\Omega)$ by \eqref{es_Ank_bound}. Thus, taking $m\to\infty$ in \eqref{est.hatAn-Am}, we obtain \eqref{est.hatAn-hatA}. Moreover, taking $n\to \infty$ in \eqref{est.hatAn-Lip}, we obtain \eqref{est.hatA-Lip}. 
\end{proof}

For further applications in quantitative homogenization via reiterated homogenization, the size estimates of the intermediate matrices $A^k_n$ and $B^n_k$ are crucial.
  \begin{lemma}\label{qual-homog_lemma_bkn}
  Let $B_k^n$ and $A_n^k$ be matrices given in Section \ref{sec_approximate-problems}. We have for each $0\leq k\leq n$,
  \begin{align}\label{est.Bnk.Linfty}
    \|B_k^n\|_{L^\infty}\leq C\sum_{\ell=k}^n\|B_\ell\|_{L^\infty},
  \end{align}
  and for each $0\leq j\leq k$,
  \begin{align}
    \|\nabla_{y_j}A_n^k\|_{L^\infty}\leq C\sum_{\ell=j}^n\|\nabla_{y_j}B_\ell\|_{L^\infty},\nonumber\\
    \|\nabla_{y_j}B_k^n\|_{L^\infty}\leq C\sum_{\ell=j}^n\|\nabla_{y_j}B_\ell\|_{L^\infty}, \label{est.DjBnk.Linfty}
  \end{align}
  where $C$ depends only on $\mu$. 
\end{lemma}

\begin{proof}
  The proof is similar to that of Lemma \ref{qual-homog_lem_difference}. Recall that for fixed parameters $y_0, y_1, \cdots, y_k$,
  \begin{align}\label{qual-homog_es_An-Hconver}
    A_n(y_0, \cdots, y_k, x/\varepsilon_{k+1}, \cdots, x/\varepsilon_n)\xrightarrow[\varepsilon\doublerightarrow 0]{H} A_n^k(y_0, \cdots, y_k),
  \end{align}
  and $A_{k-1}(y_0, \cdots, y_{k-1})$ (as a constant matrix in $x$) $H$-converges to itself. Also note that
  \begin{align*}
    |A_n(y_0, \cdots, y_k, x/\varepsilon_{k+1}, \cdots, x/\varepsilon_n)-A_{k-1}(y_0, \cdots, y_{k-1})|\leq \sum_{\ell=k}^n\|B_\ell\|_{L^\infty}.
  \end{align*}
  As a result, Lemma \ref{qual-homog_lem_stability} yields
  \begin{align*}
    |A_n^k(y_0, \cdots, y_k)-A_{k-1}(y_0, \cdots, y_{k-1})|\leq C\sum_{\ell=k}^n\|B_\ell\|_{L^\infty},
  \end{align*}
  which, together with \eqref{expression_Ank} (which can be written as $A^k_n = A_{k-1} + B^n_k$), gives
  \begin{align*}
    |B_k^n(y_0, \cdots, y_k)|\leq C\sum_{\ell=k}^n\|B_\ell\|_{L^\infty},
  \end{align*}
  where $C$ depends only on $\mu$.

  On the other hand, since for $0\leq j\leq k$ and any $y_j, y_j' \in \R^d$,
  \begin{align*}
    &\Big|A_n\Big(y_0, \cdots, y_j,\cdots, y_k, \frac{x}{\varepsilon_{k+1}}, \cdots, \frac{x}{\varepsilon_n}\Big)-A_n\Big(y_0, \cdots, y_j', \cdots, y_k, \frac{x}{\varepsilon_{k+1}}, \cdots, \frac{x}{\varepsilon_n}\Big)\Big|\\&\leq \sum_{\ell=j}^n\|\nabla_{y_j}B_\ell\|_{L^\infty}|y_j-y_j'|,
  \end{align*}
  we deduce from \eqref{qual-homog_es_An-Hconver} that
  \begin{align*}
    \|\nabla_{y_j}A_n^k\|_{L^\infty}\leq C\sum_{\ell=j}^n\|\nabla_{y_j}B_\ell\|_{L^\infty}. 
  \end{align*}
Notice that, for $0\leq j\leq k$, $\nabla_{y_j}B_k^n=\nabla_{y_j}A_n^k-\nabla_{y_j}A_{k-1}$. The estimate of $\|\nabla_{y_j}B_k^n\|_{L^\infty}$ follows by the triangle inequality. 
\end{proof}

\subsection{Qualitative homogenization theorem}

\begin{proof} [Proof of Theorem \ref{thm.qualitative}]
    We start with the approximate equation \eqref{eq_approximate}. The classical reiterated homogenization theory yields that $A_n^\varepsilon\xrightarrow[\varepsilon\doublerightarrow 0]{H} \widehat{A}_n$. This means that $u_{n,\e} \to u_{n,0}$ weakly in $H^1(\Omega)$ and $A_{n,\e} \nabla u_{n,\e} \to \widehat{A_{n}} \nabla u_{n,0}$ weakly in $L^2(\Omega)$, where $u_{n,0}$ is the weak solution of \eqref{eq_approx_homo}.
    Thanks to Theorem \ref{qual-homog_thm_hatA}, we know that $\widehat{A}_n \to \widehat{A}$ uniformly, where $\widehat{A}$ satisfies the ellipticity and boundedness conditions as in \eqref{cond_ellipticity} and \eqref{cond_bounded-A} and
  \begin{align*}
    \| \widehat{A}_n - \widehat{A} \|_{L^\infty(\Omega)} \le C \sum_{k=n+1}^\infty \|B_k\|_{L^\infty}.
  \end{align*}
    Combining the equations for $u_{n,0}$ \eqref{eq_approx_homo} and $u_0$ \eqref{eq_homogenized}, we get the equation for $u_{n,0} - u_0$:
    \begin{equation*}
    \begin{cases}
      -\diver(\widehat{A}\nabla (u_{n,0} - u_0)= \diver ((\widehat{A}_n - \widehat{A}) \nabla u_{n,0}) &\text{in }\Omega,\\
      u_{n,0} - u_0 = 0&\text{on }\partial\Omega.
    \end{cases}
  \end{equation*}
  It follows by the energy estimates that 
  \begin{align}\label{qual-homog_conver_u0}
  \begin{aligned}
      \|\nabla(u_{n, 0}-u_0)\|_{L^2(\Omega)} & \leq C\|\widehat{A}_n-\widehat{A}\|_{L^\infty(\Omega)} \| \nabla u_{n,0}\|_{L^2(\Omega)} \\
      & \leq C \Big(\| f\|_{H^{-1}(\Omega)} + \| g\|_{H^{1/2}(\partial \Omega)} \Big) \sum_{k=n+1}^\infty \|B_k\|_{L^\infty},
  \end{aligned}
  \end{align}
  and
  \begin{align}\label{est.flux.An-A0}
  \begin{split}
    &\|\widehat{A}_n\nabla u_{n, 0}-\widehat{A}\nabla u_0\|_{L^2(\Omega)}\\
    &\leq \|(\widehat{A}_n-\widehat{A})\nabla u_{n, 0}\|_{L^2(\Omega)}+\|\widehat{A}(\nabla u_{n, 0}-\nabla u_0)\|_{L^2(\Omega)}\\
    &\leq C\Big(\| f\|_{H^{-1}(\Omega)} + \| g\|_{H^{1/2}(\partial \Omega)} \Big) \sum_{k=n+1}^\infty \|B_k\|_{L^\infty},
    \end{split}
  \end{align}
  where we have used \eqref{est.hatAn-hatA} in both of these inequalities.
  
  On the other hand, by
  \begin{align*}
    \|A^\varepsilon-A_n^\varepsilon\|_{L^\infty(\Omega)}\leq C\sum_{k=n+1}^\infty\|B_k\|_{L^\infty}
  \end{align*}
  and a similar energy argument, we obtain
  \begin{align}\label{qual-homog_conver_uepsilon}
    \|\nabla (u_\varepsilon-u_{n, \varepsilon})\|_{L^2(\Omega)} \leq C\Big(\| f\|_{H^{-1}(\Omega)} + \| g\|_{H^{1/2}(\partial \Omega)} \Big) \sum_{k=n+1}^\infty\|B_k\|_{L^\infty},
  \end{align}
  and
  \begin{align}\label{est.flux.Ane-Ae}
   \begin{split}
   &\quad\|A_n^\varepsilon\nabla u_{n, \varepsilon}-A^\varepsilon\nabla u_\varepsilon\|_{L^2(\Omega)}\\&\leq  C \Big(\| f\|_{H^{-1}(\Omega)} + \| g\|_{H^{1/2}(\partial \Omega)} \Big) \sum_{k=n+1}^\infty \|B_k\|_{L^\infty}.
   \end{split}
  \end{align}

  Finally, we show $\nabla u_\e \to \nabla u_0$ weakly in $L^2(\Omega)$ and $A^\e \nabla u_\e \to \widehat{A} \nabla u_0$ weakly in $L^2(\Omega)$ as $\e \doublerightarrow 0$. In fact, for any $\phi \in L^2(\Omega)^d$,
  \begin{equation*}
  \begin{aligned}
      \bigg| \int_{\Omega} (\nabla u_\e - \nabla u_0) \cdot \phi \bigg| & \le \bigg| \int_{\Omega} (\nabla u_{n,\e} - \nabla u_\e) \cdot \phi \bigg| \\
      & \quad  + \bigg| \int_{\Omega} (\nabla u_{n,\e} - \nabla u_{n,0}) \cdot \phi \bigg| + \bigg| \int_{\Omega} (\nabla u_{n,0} - \nabla u_0) \cdot \phi \bigg| \\
      & \le \bigg| \int_{\Omega} (\nabla u_{n,\e} - \nabla u_{n,0}) \cdot \phi \bigg| \\
      & \quad + C\| \phi \|_{L^2(\Omega)} \Big(\| f\|_{H^{-1}(\Omega)} + \| g\|_{H^{1/2}(\partial \Omega)} \Big) \sum_{k=n+1}^\infty \|B_k\|_{L^\infty},
  \end{aligned}
  \end{equation*}
  where we have used \eqref{qual-homog_conver_u0} and \eqref{qual-homog_conver_uepsilon} in the last inequality.
  Since for each fixed $n\ge 1$, $\e \doublerightarrow 0$ implies
  \begin{equation*}
      \sup_{1\le i \le n} \frac{\e_i}{\e_{i-1}} \to 0.
  \end{equation*}
  It follows from \cite{Allaire1996Multiscale, Niu20_reiterated} that $u_{n,\e} \to u_{n,0}$ weakly in $H^1(\Omega)$ as $\e \doublerightarrow 0$. Thus, 
  \begin{align*}
      & \limsup_{\e \doublerightarrow 0} \bigg| \int_{\Omega} (\nabla u_\e - \nabla u_0) \cdot \phi \bigg| \\
      &\le C\| \phi \|_{L^2(\Omega)} \left(\| f\|_{H^{-1}(\Omega)} + \| g\|_{H^{1/2}(\partial \Omega)} \right) \sum_{k=n+1}^\infty \|B_k\|_{L^\infty}.
  \end{align*}
  By taking $n\to \infty$ and using the assumption \eqref{cond_bounded}, we obtain
  \begin{equation*}
      \lim_{\e \doublerightarrow 0} \int_{\Omega} (\nabla u_\e - \nabla u_0) \cdot \phi = 0.
  \end{equation*}
  This, together with the fact $u_\e = u_0$ on $\partial \Omega$, implies that $u_\e \to u_0$ weakly in $H^1(\Omega)$. By a similar argument using \eqref{est.flux.An-A0} and \eqref{est.flux.Ane-Ae}, we can show that $A^\e \nabla u_\e \to \widehat{A} \nabla u_0$ weakly in $L^2(\Omega)$ as $\e \doublerightarrow 0$. It follows  that as $\varepsilon\doublerightarrow 0$, $A^\varepsilon(x)$ $H$-converges to $\widehat{A}$ in $\mathcal{M}(\mu, \mu^{-3}; \Omega)$. The proof is complete.
\end{proof}

\section{Convergence rates}\label{sec_convergence-rates}

\subsection{Recursive systems}\label{sec_recursive-systems}
To obtain the convergence rates for the equation \eqref{eq_infinite}, it suffices to consider the approximate problem \eqref{eq_approximate} with $n$ scales and establish convergence rates independent of $n$. Thus, we will fix $n >1$. 
In view of the process of reiterated homogenization described in Section \ref{sec_approximate-problems}
or \eqref{ReteratedHomo-H-conv}, we need to consider all intermediate problems with coefficient matrices $A_n^{k,\e}(x)$.
Recall that $A^{k}_n$ is obtained through the backward recursive system \eqref{expression_Ank} and \eqref{iden_Blk}, which can be written in a concise form
\begin{equation*}
\begin{cases}
    A_n^k = A_{k-1} + B^n_k,\\
    B^n_k = B^n_{k+1} + \Ag{ B_{k+1}^n (I + \nabla_{y_{k+1}} \chi_{A^{k+1}_n}) }_{y_{k+1}},
\end{cases}
\end{equation*}
for $1\le k \le n-1$ with $B_n^n = B_n$ and $A^n_n = A_n$, where $\Ag{\cdot}_{y_{k+1}}$ denotes the mean value over $y_{k+1} \in \T^d$. To prove a convergence rate uniform in the number of scales, we need accurate size estimates of $B^n_k$ that could be accumulated through the recursive system. The $L^\infty$ bound of $B^n_k$ has been given by \eqref{est.Bnk.Linfty}, which is sufficient for us. We also need the $L^\infty$ bound of $\nabla_{y_j} B^n_k$ decaying as $n, k \to \infty$, which is not implied by \eqref{est.DjBnk.Linfty}.

Denote for $0\leq k\leq n$
$$\delta_k^n=\max_{0\leq j\leq k}\|\nabla_{y_j}B_k^n\|_{L^\infty}. $$
We have the following backward recursive estimates for $\delta^n_k$.

\begin{lemma}\label{recursive_lem_recursiveformula}
  It holds that $\delta_n^n\leq \delta_n$, and for $0\le k\leq n-1$,
  \begin{align}\label{conver_ineq_recursive}
    \delta_{k}^{n}\leq \delta_{k}+\delta_{k+1}^n+C_0\sum_{\ell=k+1}^n\delta_\ell\delta_{k+1}^n+C_0[\delta]_0\bigg(\sum_{\ell=k+1}^n\delta_\ell\bigg)^2,
  \end{align}
  where $C_0$ depends only on $d, \mu$.
\end{lemma}
\begin{proof}
  The estimate for $\delta_n^n$ is trivial, since $B_n^n=B_n$. For $0\leq j\leq k\leq n-1$, by the equation \eqref{eq.chi.k+1} of $\chi_{A_n^{k+1}}$, one can see that  for each $(y_0, \cdots, y_k)$,
  \begin{align*}
    \|\nabla_{y_{k+1}}\chi_{A_n^{k+1}}\|_{L^2_{y_{k+1}}(\mathbb{T}^d)}\leq C\|B_{k+1}^n\|_{L^2_{y_{k+1}}(\mathbb{T}^d)},
  \end{align*}
  and
  \begin{align*}
\|\nabla_{y_{k+1}}\nabla_{y_j}\chi_{A_n^{k+1}}\|_{L^2_{y_{k+1}}(\mathbb{T}^d)} & \leq C\|\nabla_{y_j}B_{k+1}^n\|_{L^2_{y_{k+1}}(\mathbb{T}^d)}\\&\quad+C\|\nabla_{y_j}A_n^{k+1}\|_{L^\infty_{y_{k+1}}(\mathbb{T}^d)}\|\nabla_{y_{k+1}}\chi_{A_n^{k+1}}\|_{L^2_{y_{k+1}}(\mathbb{T}^d)},
  \end{align*}
 where $C$ depends only on $d, \mu$. In view of \eqref{iden_Blk}, this implies  that
  \begin{align*}
    &|\nabla_{y_j}B_{k}^n|\leq |\nabla_{y_j}B_{k}|+\fint_{\mathbb{T}^d}|\nabla_{y_j}B_{k+1}^n|dy_{k+1}\\
&\qquad\qquad+\fint_{\mathbb{T}^d}
\Big( |\nabla_{y_j}B_{k+1}^n||\nabla_{y_{k+1}} \chi_{A_n^{k+1}}|+|B_{k+1}^n||\nabla_{y_j}\nabla_{y_{k+1}} \chi_{A_n^{k+1}}|\Big)dy_{k+1}\\
&\leq |\nabla_{y_j}B_{k}|+\|\nabla_{y_j}B_{k+1}^n\|_{L^2_{y_{k+1}}(\mathbb{T}^d)}+C\|\nabla_{y_j}B_{k+1}^n\|_{L^2_{y_{k+1}}(\mathbb{T}^d)}\|B_{k+1}^n\|_{L^2_{y_{k+1}}(\mathbb{T}^d)}\\&\qquad+C\|B_{k+1}^n\|_{L^2_{y_{k+1}}(\mathbb{T}^d)}^2\|\nabla_{y_j}A_n^{k+1}\|_{L^\infty_{y_{k+1}}(\mathbb{T}^d)}.
  \end{align*}
Since by Lemma \ref{qual-homog_lemma_bkn}, 
\begin{align*}
  \|\nabla_{y_j}A_n^{k+1}\|_{L^\infty}\leq C\sum_{\ell=j}^{n}\|\nabla_{y_j}B_\ell\|_{L^\infty},
\end{align*}
we deduce that
\begin{align*}
  \|\nabla_{y_j}B_{k}^n\|_{L^\infty}&\leq \|\nabla_{y_j}B_{k}\|_{L^\infty}+\|\nabla_{y_j}B_{k+1}^n\|_{L^\infty}+C\|B_{k+1}^n\|_{L^\infty}\|\nabla_{y_j}B_{k+1}^n\|_{L^\infty}\\&\quad+C\|B_{k+1}^n\|_{L^\infty}^2\sum_{\ell=j}^{n}\|\nabla_{y_j} B_\ell\|_{L^\infty}.
\end{align*}
Taking the supremum over $0\le j\le k$, we obtain 
  \begin{align*}
    \delta_{k}^{n}&\leq \delta_{k}+\delta_{k+1}^n+C\sum_{\ell=k+1}^n\delta_\ell\delta_{k+1}^n+C\bigg(\sum_{\ell=k+1}^n\delta_\ell\bigg)^2\sum_{\ell=0}^{n}\delta_\ell\\&\leq \delta_{k}+\delta_{k+1}^n+C\sum_{\ell=k+1}^n\delta_\ell\delta_{k+1}^n+C[\delta]_0\bigg(\sum_{\ell=k+1}^n\delta_\ell\bigg)^2,
  \end{align*}
  where we have used Lemma \ref{qual-homog_lemma_bkn} as well as the assumption \eqref{cond_B}. The constant $C$ depends only on $d$ and $\mu$. The proof is complete. 
\end{proof}

The next lemma gives a good control for the recursive system \eqref{conver_ineq_recursive}.

\begin{lemma}\label{conver_lemma_dkn}
  Let $\delta_k^n$ satisfy the recursive conditions in Lemma \ref{recursive_lem_recursiveformula}. Suppose $[\delta]_1<\infty$. Then for any $n\geq 1$ and $0\le k\leq n$, 
  \begin{align}\label{re-0}
    \delta_k^n\leq \exp\big\{C_0[\delta]_1\big\}\cdot \big\{ 1+C_0[\delta]_0[\delta]_1\big\}\cdot R_k^n,
  \end{align}
  where
  \begin{align}
    \label{conver_def_Rkn}
    R_k^n=\sum_{j=k}^n\delta_j. 
  \end{align}
\end{lemma}
\begin{proof}
For fixed $n$, we write \eqref{conver_ineq_recursive} as
\begin{align*}
  \delta_k^n\leq a_k\delta_{k+1}^n+b_k,
\end{align*}
where
\begin{align*}
    a_k=1+C_0R_{k+1}^n,\quad b_k=\delta_k+C_0[\delta]_0 (R_{k+1}^n)^2.
\end{align*}
By using a backward induction on $k$, we can show that  
\begin{align*}
  \delta_k^n\leq \prod_{\ell=k}^{n-1}a_\ell \delta_n^n+\sum_{j=k}^{n-1} \prod_{\ell=k}^{j-1}a_\ell b_j,
\end{align*}
where we have employed the convention that $\prod_{\ell=k}^{j-1}a_\ell=1$ if $j-1<k$. To bound the first term, we note that  for each $k$, 
\begin{align*}
    \prod_{\ell=k}^{n-1} a_\ell&=\prod_{\ell=k}^{n-1}(1+C_0R_{\ell+1}^n)=\exp\bigg\{\sum_{\ell=k}^{n-1} \log(1+C_0R_{\ell+1}^n)\bigg\}\\
    &\leq \exp\bigg\{\sum_{\ell=k}^{n-1} C_0R_{\ell+1}^n\bigg\}
    \leq \exp\bigg\{C_0\sum_{j=k+1}^n\sum_{\ell=k}^{j-1} \delta_j\bigg\}\leq \exp\big\{C_0[\delta]_1\big\},
  \end{align*}
which gives
\begin{align*}
  \prod_{\ell=k}^{n-1} a_\ell \delta_n^n\leq \exp\big\{C_0[\delta]_1\big\}\delta_n^n.
\end{align*}
  For the second term, since $\prod_{\ell=k}^{j-1} a_\ell\leq \prod_{\ell=k}^{n-1}a_\ell$, we have
  \begin{align*}
    \sum_{j=k}^{n-1}\prod_{\ell=k}^{j-1}a_\ell b_j&\leq \exp\big\{C_0[\delta]_1\big\}\cdot \sum_{j=k}^{n-1} \left\{\delta_j+C_0[\delta]_0 (R_{j+1}^n)^2\right\}\\
    &\leq \exp\big\{C_0[\delta]_1\big\}\cdot \bigg(R_k^{n-1}+C_0[\delta]_0 R_k^n\sum_{j=k}^{n-1} R_{j+1}^n\bigg)\\&\leq \exp\big\{C_0[\delta]_1\big\}\cdot \big\{R_k^{n-1}+C_0[\delta]_0[\delta]_1R_k^n\big\}. 
  \end{align*}
Combining these two estimates,  we obtain the desired inequality \eqref{re-0}.
\end{proof}

For future use, we introduce the notation 
\begin{align}
  \label{conver_def_Rk}
  R_k:=\sum_{\ell=k}^\infty \delta_\ell\quad\text{for }k\geq 0.
\end{align}
Notice  that
\begin{align*}
  \sum_{k=1}^\infty k^\alpha R_k=\sum_{k=1}^\infty k^\alpha \sum_{j=k}^\infty \delta_j=\sum_{j=1}^\infty \sum_{k=1}^j k^\alpha \delta_j,
\end{align*}
and that there exist $c_\alpha, C_\alpha$ depending only on $\alpha\in [0,\infty)$ such that
\begin{align*}
  c_\alpha j^{\alpha+1}\leq \sum_{k=1}^j k^\alpha\leq C_\alpha j^{\alpha+1}.
\end{align*}
It follows that 
\begin{align}\label{recursive_es_Rk}
  c_\alpha [\delta]_{\alpha+1}\leq \sum_{k=1}^\infty k^\alpha R_k\leq C_\alpha [\delta]_{\alpha+1}.
\end{align}

\subsection{Improved one-step convergence rates}\label{sec_case-n=1}
In this section, we establish an improved convergence rate in a two-scale problem, which reflects the homogenization effect of a small perturbation.

Consider the Dirichlet problem,
\begin{equation}
  \label{eq_n=1_general}
  \begin{cases}
    -\mathrm{div}(A(x, x/\e)\nabla u_\varepsilon)=f&\text{in }\Omega, \\
    u_\varepsilon=g &\text{on }\partial\Omega,
  \end{cases}
\end{equation}
where $0<\varepsilon\leq 1$, the coefficient matrix $A(x, y)=E_0(x)+E_1(x, y)$ 
is $1$-periodic in $y = (y^1,\cdots, y^d) \in \T^d$ 
and satisfies
\begin{gather}\label{conver_cond_case-n=1}
\left\{
  \begin{split}
    & \xi \cdot A(x, y)\xi \geq \mu|\xi|^2\quad \text{for any }\xi \in\mathbb{R}^{d},\\
  & \|E_0\|_{L^\infty}+\|E_1\|_{L^\infty}\leq \mu^{-1},\\
  & \|\nabla_x E_0\|_{L^\infty}+\|\nabla_x E_1\|_{L^\infty}<\infty.
  \end{split}
  \right.
\end{gather}
The homogenized equation of \eqref{eq_n=1_general} is given by
\begin{align}\label{eq_n=1.homo}
  \begin{cases}
    -\mathrm{div}(\widehat{A}\nabla u_0)=f&\text{in }\Omega, \\
    u_0=g &\text{on }\partial\Omega.
  \end{cases}
\end{align}
In the following, we provide the detailed equations for the corrector, flux corrector and the homogenized matrix, with emphasis on the role played by the structure of $A(x,y)$.

Denote by $\chi(x, y)$ and $\phi(x, y)$ the corrector and flux corrector of equation \eqref{eq_n=1_general}. More precisely, the corrector $\chi=(\chi^j)_{1\le j\le d}$ satisfies 
\begin{equation}\label{eq_one-scale_chi}
  \begin{cases}
    -\mathrm{div}_y(A\nabla_y \chi^j)=\mathrm{div}_y(A e^j)=\mathrm{div}_y (E_1 e^j)&\text{in }\mathbb{T}^d,\\
    \chi^j\text{ is $1$-periodic in $y$ and }\int_{\mathbb{T}^d}\chi^j dy=0.
  \end{cases}
\end{equation}
In the second equality, we have used the fact $\nabla_y A = \nabla_y E_1$, as $E_0$ is independent of $y$. This simple fact allows us to transfer the smallness of $E_1$ to the correctors.
The homogenized matrix is given by
\begin{align}\label{n=1_def_hatA}
\begin{split}
    \widehat{A}(x) & = \fint_{\mathbb{T}^d} \left\{ A(x, y)+A(x, y)\nabla_{y}\chi(x, y) \right\} dy\\&=E_0(x)+\fint_{\mathbb{T}^d} \left\{ E_1(x, y)+E_1(x, y)\nabla_{y}\chi(x, y)\right\}  dy.
\end{split}
\end{align}
Let
\begin{equation}\label{F}
    F_{ij} = A_{ij}+ \sum_{\ell = 1}^d A_{i\ell}\partial_{y^\ell}\chi^j-\widehat{A}_{ij}.
\end{equation}
The flux corrector $\phi=(\phi_{kij})_{1\le k,i,j\le d}$ is given by
\begin{align}\label{eq.phi}
  \begin{cases}
    \Delta_y \phi_{kij}=\partial_{y^k} F_{ij} -\partial_{y^i}F_{kj} \quad \text{in }\mathbb{T}^d,\\
  \phi_{kij}\text{ is $1$-periodic in $y$ and }\int_{\mathbb{T}^d}\phi_{kij} dy=0. 
  \end{cases}
\end{align}
Note that $\phi_{kij}$ is skew-symmetric in $k$ and $i$, and
\begin{align}\label{conver_iden_flux}
\sum_{k=1}^d \partial_{y^k}\phi_{kij} = F_{ij}.
\end{align}
We also denote by $\chi^*(x, y)$ and $\phi^*(x, y)$ the corrector and flux corrector for the corresponding equation with the coefficient matrix $A^*$. They share the same properties as $\chi$ and $\phi$.
\begin{lemma}\label{lem_chi_E1}
  Under the assumptions above, we have
  \begin{equation}\label{est_E1}
      \|\chi\|_{L^\infty H^1}+\|\phi\|_{L^\infty H^1}+\|A-\widehat{A}\|_{L^\infty}\leq C\|E_1\|_{L^\infty},
  \end{equation}
  and
  \begin{equation}\label{est_chi_DE1}
  \begin{aligned}
    & \|\nabla_x\chi\|_{L^\infty H^1}+\|\nabla_x\phi\|_{L^\infty H^1}+\|\nabla_x(A-\widehat{A})\|_{L^\infty}\\
    &\quad\leq C(\|\nabla_xA\|_{L^\infty}\|E_1\|_{L^\infty}+\|\nabla_x E_1\|_{L^\infty}),
    \end{aligned}
  \end{equation}
  where $L^\infty H^1=L^\infty(\Omega; H^1(\mathbb{T}^d))$ and $C$ depends only on $d $ and $ \mu$. 
\end{lemma}

\begin{proof}
By the energy estimates, it follows from equation \eqref{eq_one-scale_chi} that for each $x\in\Omega$,
\begin{align}\label{est.Dychi}
  \|\nabla_{y}\chi\|_{L_{y}^2(\mathbb{T}^d)}\leq C\|E_1\|_{L_{y}^2(\mathbb{T}^d)}.
\end{align}
Using the Poincar\'{e} inequality and taking the supremum over $x\in \Omega$, we obtain the estimate of $\|\chi\|_{L^\infty H^1}$ in \eqref{est_E1}. Also, 
by taking the derivative in $x$ in  the equation \eqref{eq_one-scale_chi}, we obtain 
\begin{equation*}
    -\mathrm{div}_y(A\nabla_y \partial_{x^k}\chi^j)= \mathrm{div}(\partial_{x^k} E_1 e^j) + \mathrm{div}_y(\partial_{x^k} A\nabla_y \chi^j).
\end{equation*}
Thus, the energy estimate yields
\begin{align}\label{est.DyDxchi}
  \begin{split}
  \|\nabla_y\nabla_x\chi\|_{L_{y}^2(\mathbb{T}^d)}&\leq C\|\nabla_xE_1\|_{L_y^2(\mathbb{T}^d)}+C\|\nabla_xA\nabla_y\chi\|_{L_y^2(\mathbb{T}^d)}\\&\leq C\|\nabla_x E_1\|_{L_y^2(\mathbb{T}^d)}+C\|\nabla_x A\|_{L^\infty}\|E_1\|_{L_y^2(\mathbb{T}^d)},
  \end{split}
\end{align}
which leads to  the estimate of $\| \nabla_x \chi \|_{L^\infty H^1}$ in \eqref{est_chi_DE1}. 

Next, note that
\begin{align*}
  F = A+A\nabla_{y}\chi-\widehat{A}=E_1+A\nabla_{y}\chi-\fint_{\mathbb{T}^d}E_1 dy-\fint_{\mathbb{T}^d}E_1\nabla_{y} \chi dy.
\end{align*}
We obtain that  for each $x$,
\begin{align*}
    \|F\|_{L^2_{y}(\mathbb{T}^d)}&\leq C\|E_1\|_{L^2_{y}(\mathbb{T}^d)},\\
    \|\nabla_x F\|_{L^2_{y}(\mathbb{T}^d)}&\leq C\|\nabla_xE_1\|_{L^2_{y}(\mathbb{T}^d)}+C\|\nabla_xA\|_{L^\infty}\|E_1\|_{L_y^2(\mathbb{T}^d)},
\end{align*}
where we have used  \eqref{est.Dychi} and \eqref{est.DyDxchi}.
As a result, by the energy estimate for \eqref{eq.phi}, the flux corrector $\phi$ satisfies
\begin{align*}
  \|\phi\|_{H^1_{y}(\mathbb{T}^d)}&\leq C\|E_1\|_{L^2_{y}(\mathbb{T}^d)},\\
  \|\nabla_x\phi\|_{H^1_{y}(\mathbb{T}^d)}&\leq C\|\nabla_xE_1\|_{L^2_{y}(\mathbb{T}^d)}+C\|\nabla_xA\|_{L^\infty}\|E_1\|_{L_y^2(\mathbb{T}^d)}.
\end{align*}

Finally, note that
\begin{align*}
  A-\widehat{A}=E_1-\fint_{\mathbb{T}^d}E_1 dy-\fint_{\mathbb{T}^d}E_1\nabla_{y} \chi dy,
\end{align*}
from which, together with \eqref{est.Dychi} and \eqref{est.DyDxchi}, we  deduce that
\begin{align*}
  \|A-\widehat{A}\|_{L^\infty}&\leq C\|E_1\|_{L^\infty},\\
  \|\nabla_x(A-\widehat{A})\|_{L^\infty}&\leq C\|\nabla_xE_1\|_{L^\infty}+C\|\nabla_xA\|_{L^\infty}\|E_1\|_{L^\infty}. 
\end{align*}
The proof is complete.
\end{proof}

In this section, for a  function $g(x,y)$ defined on $\Omega \times \T^d$, we write $g^\e(x) = g(x,x/\e)$.
Thus, $A^\e(x) = A(x,x/\e), \chi^\e(x) = \chi(x,x/\e), \phi^\e(x) = \phi(x,x/\e)$.
 Define the smoothing operator by 
$$S_\varepsilon(g^\varepsilon)(x):= g(\cdot, x/\e)* \varphi_\e(x) = \int_{\mathbb{R}^d}g(z,x/\varepsilon)\varphi_\varepsilon(x-z)dz,$$ 
where $\varphi_\varepsilon(x)=\varepsilon^{-d}\varphi(x/{\varepsilon})$ and $\varphi$ is a function in $C_0^\infty(B(0, 1/2))$ such that  $\varphi\geq 0$ and $\int_{\mathbb{R}^d}\varphi =1$. Note that the smoothing at $\e$-scale is only done w.r.t the non-periodic slow variable.

We recall two lemmas taken from \cite[Lemmas 2.2 and 2.3]{Niu20_reiterated}.

\begin{lemma}
    Suppose that $h = h(x,y) \in L^\infty(\R^d; L^2(\T^d))$ and $f = f(x) \in L^2(\R^d)$. Then
    \begin{equation*}
        \| S_\e(h^\e f) \|_{L^2(\R^d)} \le C \| f\|_{L^2(\R^d)} \sup_{x\in \R^d} \| h(x,\cdot)\|_{L^2(\T^d)},
    \end{equation*}
    where $C$ depends only on $d$.
\end{lemma}

\begin{lemma}\label{lem.h-Seh}
    Suppose that $h = h(x,y) \in L^\infty(\R^d \times \T^d), \nabla_x h \in L^\infty(\R^d \times \T^d)$ and $f = f(x) \in H^1(\R^d)$. Then
    \begin{equation*}
        \| h^\e f - S_\e (h^\e f) \|_{L^2(\R^d)} \le C\e \Big\{ \| \nabla_x h\|_{L^\infty} \| f\|_{L^2(\R^d)} + \| h\|_{L^\infty} \| \nabla f\|_{L^2(\R^d)} \Big\},
    \end{equation*}
    where $C$ depends only on $d$.
\end{lemma}

Let $\eta_\varepsilon\in C_0^\infty(\Omega)$ be a cutoff function satisfying $0\leq \eta_\varepsilon\leq 1$, $\eta_\varepsilon=1$ on $\Omega\setminus \Omega_{4\varepsilon}$, $\eta_\varepsilon=0$ on $\Omega_{3\varepsilon}$, where $$\Omega_t:=\{x\in \Omega: \mathrm{dist}(x, \partial\Omega)<t\}.$$

\begin{theorem}\label{thm_n=1_general}
  Let $\Omega$ be a bounded $C^{1, 1}$ domain in $\mathbb{R}^d$. Let $u_\e$ and $u_0$ be the weak solutions of \eqref{eq_n=1_general} and \eqref{eq_n=1.homo}, respectively. Then
  \begin{align*}
    \begin{split}
         \|u_\varepsilon-u_0\|_{L^2(\Omega)}& \leq C\varepsilon\big\{\|E_1\|_{L^\infty}+\|\nabla_xE_1\|_{L^\infty}+\varepsilon\|\nabla_xE_1\|_{L^\infty}^2\\
         &\qquad +(\|E_1\|_{L^\infty}+\varepsilon\|E_1\|_{L^\infty}\|\nabla_xE_1\|_{L^\infty})\|\nabla_xA\|_{L^\infty}\\
         &\qquad +\varepsilon \|E_1\|_{L^\infty}^2\|\nabla_xA\|_{L^\infty}^2\big\}\times (\|f\|_{L^2(\Omega)}+\|g\|_{H^{3/2}(\partial\Omega)}), 
    \end{split}
  \end{align*}
  where $C$ depends only on $d, \mu$ and $\Omega$. 
\end{theorem}
\begin{proof}
    The proof is a refinement of that in \cite[Lemmas 3.2 and 4.1]{Niu20_reiterated}. Let
    \begin{equation}\label{def.we}
        w_\varepsilon=u_\varepsilon-u_0-\varepsilon S_\varepsilon(\eta_\varepsilon\chi^\varepsilon\nabla u_0).
    \end{equation}
    Then
    \begin{align*}
        -\mathrm{div}(A^\varepsilon \nabla w_{\varepsilon}) & =\mathrm{div}[(A^\varepsilon-\widehat{A})\nabla u_0]+\mathrm{div}[A^\varepsilon S_\varepsilon\big(\eta_\varepsilon(\nabla_y\chi)^\varepsilon\nabla u_0\big)]\\
        &\qquad+\varepsilon\mathrm{div}\big[A^\varepsilon S_\varepsilon\big([\nabla_x(\eta_\varepsilon\chi\nabla u_0)]^\varepsilon\big)\big]\\
        & = \mathrm{div}\big[(A^\varepsilon-\widehat{A})\nabla u_0-S_\varepsilon\big((A^\varepsilon-\widehat{A})\eta_\varepsilon \nabla u_0\big)\big]\\
        &\qquad + \mathrm{div}[A^\varepsilon S_\varepsilon\big(\eta_\varepsilon(\nabla_y\chi)^\varepsilon\nabla u_0\big)-S_\varepsilon\big(A^\varepsilon(\nabla_y\chi)^\varepsilon \eta_\varepsilon \nabla u_0\big)]\\
        &\qquad + \mathrm{div}\big[S_\varepsilon\big( F(x,x/\e) \eta_\varepsilon \nabla u_0\big)\big]\\
        & \qquad + \varepsilon\mathrm{div}\big[A^\varepsilon S_\varepsilon\big([\nabla_x(\eta_\varepsilon\chi\nabla u_0)]^\varepsilon\big)\big],
    \end{align*}
  where $F=F(x, y)$ is given by \eqref{F}.
    It follows that  for any test function $\psi\in H_0^1(\Omega)$,
    \begin{align*}
        &\left|\int_\Omega A^\varepsilon\nabla w_\varepsilon\cdot\nabla\psi dx\right|\\&\leq\int_{\Omega}|(A^{\varepsilon}-\widehat{A})\nabla u_0-S_{\varepsilon}\big((A^{\varepsilon}-\widehat{A})\eta_{\varepsilon}\nabla u_0\big)||\nabla\psi|dx\\
        &\quad + \int_{\Omega}\left|A^\varepsilon S_\varepsilon\left(\eta_\varepsilon(\nabla_y\chi)^\varepsilon\nabla u_0\right)-S_\varepsilon\left(\eta_\varepsilon A^\varepsilon(\nabla_y\chi)^\varepsilon\nabla u_0\right)||\nabla\psi|dx\right.\\
        &\quad + \bigg| \int_{\Omega} S_\varepsilon\big( F(x,x/\e) \eta_\varepsilon \nabla u_0\big) \cdot \nabla\psi dx \bigg| + C \varepsilon\int_{\Omega}\big|S_\varepsilon\big([\nabla_x(\eta_\varepsilon\chi\nabla u_0)]^\varepsilon\big)\big||\nabla\psi|dx \\
        & = J_1 + J_2 + J_3 + J_4.
    \end{align*}
    We will estimate $J_1$-$J_4$ separately, skipping some detailed calculations, which can be found in \cite[Lemma 3.2]{Niu20_reiterated}.
    
    By Lemma \ref{lem.h-Seh}, we obtain
    \begin{align*}
        &J_1 \le  C \|A-\widehat{A}\|_{L^\infty} \|\nabla u_0\|_{L^2(\Omega_{4\varepsilon})} \|\nabla\psi\|_{L^2(\Omega_{5\varepsilon})}\\
         & + C\varepsilon\big\{\|\nabla_x(A-\widehat{A})\|_{L^\infty}\| \nabla u_0\|_{L^2(\Omega)} + \|A-\widehat{A}\|_{L^\infty} \|\nabla^2u_0\|_{L^2(\Omega\setminus\Omega_{3\varepsilon})}\big\} \|\nabla\psi\|_{L^2(\Omega)},
    \end{align*}
    and
    \begin{align*}
        J_2 \le C\varepsilon\|\nabla_xA\|_{L^\infty}\|\nabla_y\chi\|_{L^\infty L^2}  \|\nabla u_0\|_{L^2(\Omega)} \|\nabla\psi\|_{L^2(\Omega)}.
    \end{align*}
    For $J_3$, using the skew-symmetry of the flux corrector $\phi$, we have
    \begin{equation*}
    \begin{aligned}
        J_3 & = \e \bigg| \int_{\Omega} S_\e\big( [\nabla_x( \phi \nabla u_0 \eta_\e)]^\e \big) \cdot \nabla \psi dx \bigg| \\
        & \le C\|\phi\|_{L^\infty L^2} \|\nabla u_0\|_{L^2(\Omega_{4\varepsilon})} \|\nabla\psi\|_{L^2(\Omega_{5\varepsilon})} + C \varepsilon \|\nabla_x\phi\|_{L^\infty L^2} \|\nabla u_0\|_{L^2(\Omega)} \|\nabla\psi\|_{L^2(\Omega)}\\ 
        &\qquad + C\varepsilon \|\phi\|_{L^\infty L^2} \|\nabla^2u_0\|_{L^2(\Omega\setminus\Omega_{3\varepsilon})} \|\nabla\psi\|_{L^2(\Omega)}.
    \end{aligned}
    \end{equation*}
    For $J_4$, we have
    \begin{align*}
        J_4 & \le C\|\chi \|_{L^\infty L^2} \|\nabla u_0\|_{L^2(\Omega_{4\varepsilon})} \|\nabla\psi\|_{L^2(\Omega_{5\varepsilon})} + C \varepsilon \|\nabla_x\chi\|_{L^\infty L^2} \|\nabla u_0\|_{L^2(\Omega)} \|\nabla\psi\|_{L^2(\Omega)}\\ 
        &\qquad + C\varepsilon \|\chi \|_{L^\infty L^2} \|\nabla^2u_0\|_{L^2(\Omega\setminus\Omega_{3\varepsilon})} \|\nabla\psi\|_{L^2(\Omega)}.
    \end{align*}
    Summing up $J_1$-$J_4$, we obtain
    \begin{align}\label{es_wpsi}
    \begin{split}
        & \left|\int_\Omega A^\varepsilon\nabla w_\varepsilon\cdot\nabla\psi dx\right|\\
        & \leq C\varepsilon\Big\{\big(\|A-\widehat{A}\|_{L^\infty}+\|\phi\|_{L^\infty L^2}+\|\chi\|_{L^\infty L^2}\big)\|\nabla^2u_0\|_{L^2(\Omega\setminus\Omega_{3\varepsilon})}\\
        & \qquad\qquad  + \big(\|\nabla_x(A-\widehat{A})\|_\infty+\|\nabla_xA\|_{L^\infty}\|\nabla_y\chi\|_{L^\infty L^2}\\
        & \qquad\qquad  + \|\nabla_x\phi\|_{L^\infty L^2}+\|\nabla_x\chi\|_{L^\infty L^2}\big)\|\nabla u_0\|_{L^2(\Omega)}\Big\} \|\nabla\psi\|_{L^2(\Omega)}\\
        &\quad + C\big(\|A-\widehat{A}\|_{L^\infty}+\|\phi\|_{L^\infty L^2}+\|\chi\|_{L^\infty L^2}\big) \|\nabla u_0\|_{L^2(\Omega_{4\varepsilon})} \|\nabla\psi\|_{L^2(\Omega_{5\varepsilon})}\\
        &\leq C\varepsilon\Big\{\|E_1\|_{L^\infty}\|\nabla^2u_0\|_{L^2(\Omega\setminus\Omega_{3\varepsilon})}  \\
        & \qquad \qquad +\big(\|\nabla_xA\|_{L^\infty}\|E_1\|_{L^\infty}+\|\nabla_x E_1\|_{L^\infty}\big)\|\nabla u_0\|_{L^2(\Omega)}\Big\}  \|\nabla\psi\|_{L^2(\Omega)} \\ 
        &\qquad + C\|E_1\|_{L^\infty}\|\nabla u_0\|_{L^2(\Omega_{4\varepsilon})} \|\nabla\psi\|_{L^2(\Omega_{5\varepsilon})},
    \end{split}
    \end{align}
    where we have used Lemma \ref{lem_chi_E1} in the second inequality, and $C$ depends only on $d $ and $\mu$. Setting $\psi=w_\varepsilon$ in the above inequality, we deduce that
\begin{align}\label{conver_es_wH1_0}
  \begin{split}
    \|\nabla w_\varepsilon\|_{L^2(\Omega)} & \leq C \Big\{\|E_1\|_{L^\infty}\big(\varepsilon\|\nabla^2 u_0\|_{L^2(\Omega\setminus \Omega_{3\varepsilon})}+\| \nabla u_0\|_{L^2(\Omega_{4\varepsilon})}\big)\\
    &\qquad + \varepsilon\big(\|\nabla_xA\|_{L^\infty}\|E_1\|_{L^\infty}+\|\nabla_x E_1\|_{L^\infty}\big)\|\nabla u_0\|_{L^2(\Omega)}\Big\}.
  \end{split}
\end{align}
Note that by the energy estimate and $H^2$ estimate in $C^{1,1}$ domains,
\begin{align}
    & \| \nabla u_0 \|_{L^2(\Omega)} \le C (\|f\|_{L^2(\Omega)}+\|g\|_{H^{1/2}(\partial\Omega)}), \label{est.Du0.L2} \\
    & \| \nabla^2 u_0 \|_{L^2(\Omega)} \le C(\| \nabla \widehat{A} \|_{L^\infty}+1) (\|f\|_{L^2(\Omega)}+\|g\|_{H^{3/2}(\partial\Omega)}), \label{est.DDu0.L2}
\end{align}
and by the boundary layer estimate,
\begin{align}\label{est.Du0,layer}
    \| \nabla u_0 \|_{L^2(\Omega_{4\e})} \le C\e^{1/2} \big(\| \nabla u_0 \|_{L^2(\Omega)} + \| \nabla u_0 \|_{L^2(\Omega)}^{1/2} \| \nabla^2 u_0 \|_{L^2(\Omega)}^{1/2}\big).
\end{align}
Moreover, it follows from the first line of \eqref{n=1_def_hatA} and the equation \eqref{eq_one-scale_chi} that
\begin{align}\label{conver_es_A-Lip}
\begin{split}
    \|\nabla \widehat{A}\|_{L^\infty} & \leq C \big\{\|\nabla_xA\|_{L^\infty}+\|\nabla_xA\|_{L^\infty}\|\chi\|_{L^\infty H^1}+\|A\|_{L^\infty}\|\nabla_x\chi\|_{L^\infty H^1} \big\} \\
    & \leq C\|\nabla_xA\|_{L^\infty},
\end{split}
\end{align}
where we have used the fact $\|A\|_{L^\infty}\leq \mu^{-1}$. By inserting \eqref{est.Du0.L2}-\eqref{conver_es_A-Lip} into \eqref{conver_es_wH1_0}, we obtain
\begin{align}\label{conver_es_w_H1}
  \begin{split}
    \|\nabla w_\varepsilon\|_{L^2(\Omega)} & \leq C\varepsilon^{1/2} \Big\{ \|E_1\|_{L^\infty}+\|E_1\|_{L^\infty}\|\nabla_xA\|_{L^\infty}^{1/2}+\varepsilon^{1/2}\|\nabla_xE_1\|_{L^\infty}\\
    &\qquad + \varepsilon^{1/2}\|E_1\|_{L^\infty}\|\nabla_xA\|_{L^\infty} \Big\} \cdot \Big(\|f\|_{L^2(\Omega)}+\|g\|_{H^{3/2}(\partial \Omega)}\Big). 
  \end{split}
\end{align}

Next,  we apply a duality argument as in the proof of \cite[Lemma 4.1]{Niu20_reiterated}. Let $v_\varepsilon$ be the solution of the dual problem
\begin{align*}
  \begin{cases}
    -\mathrm{div}(A^*(x,{x}/{\varepsilon})\nabla v_\varepsilon) = G &\text{in }\Omega,\\
    v_\varepsilon=0&\text{on }\partial\Omega,
  \end{cases}
\end{align*}
and $v_0$ the corresponding homogenized solution. Define 
$$\widetilde{w}_\varepsilon=v_\varepsilon-v_0-\varepsilon S_\varepsilon\big(\widetilde{\eta}_\varepsilon (\chi^*)^\varepsilon \nabla v_0\big)\in H^1_0(\Omega),$$
where $\widetilde{\eta}_\varepsilon\in C_0^\infty(\Omega)$ is a cutoff function such that $0\leq \widetilde{\eta}_\varepsilon\leq 1$, $\widetilde{\eta}_\varepsilon=1$ on $\Omega\setminus \Omega_{10\varepsilon}$, $\widetilde{\eta}_\varepsilon=0$ on $\Omega_{8\varepsilon}$. Note that $\widetilde{w}_\varepsilon$ satisfies the same estimates as $w_\varepsilon$. Then
\begin{align}\label{eq.weG}
\begin{aligned}
  \bigg|\int_\Omega w_\varepsilon\cdot G dx\bigg|&=\bigg|\int_\Omega A^\varepsilon \nabla w_\varepsilon\cdot \nabla v_\varepsilon dx\bigg|\\&\leq \bigg|\int_{\Omega}A^{\varepsilon}\nabla w_{\varepsilon}\cdot\nabla\widetilde{w}_{\varepsilon}dx\bigg|+\bigg|\int_{\Omega}A^\varepsilon\nabla w_{\varepsilon}\cdot\nabla v_{0}dx\bigg|\\&\qquad+\bigg|\int_{\Omega}A^\varepsilon\nabla w_{\varepsilon}\cdot\nabla\big[\varepsilon S_{\varepsilon}\big(\eta_\varepsilon(\chi^*)^\varepsilon \nabla v_0\big)\big]dx\bigg| \\
  & = I_1 + I_2 + I_3.
\end{aligned}
\end{align}
By a calculation as in \cite[Lemma 4.1]{Niu20_reiterated}, using \eqref{conver_es_w_H1}, we see that
\begin{align}\label{est.I1}
\begin{aligned}
  I_1 \le & C\varepsilon \big\{ \|E_1\|_{L^\infty}^2+\|E_1\|_{L^\infty}^2\|\nabla_xA\|_{L^\infty} + \varepsilon\|\nabla_xE_1\|_{L^\infty}^2\\
  & \quad + \varepsilon\|E_1\|_{L^\infty}^2\|\nabla_xA\|_{L^\infty}^2 \big\} \cdot  \left(\|f\|_{L^2(\Omega)} + \|g\|_{H^{3/2}(\partial \Omega)}\right) \| G \|_{L^2(\Omega)}.
  \end{aligned}
\end{align}
Using \eqref{es_wpsi}, we have
\begin{align}\label{est.I2}
\begin{aligned}
  I_2 & \le C \varepsilon\big\{\|E_1\|_{L^\infty}+\|\nabla_xE_1\|_{L^\infty}+\|E_1\|_{L^\infty}\|\nabla_xA\|_{L^\infty}\big\}\\
  & \qquad \cdot (\|f\|_{L^2(\Omega)}+\|g\|_{H^{3/2}(\partial \Omega)}) \| G \|_{L^2(\Omega)},
  \end{aligned}
\end{align}
and
\begin{align}\label{est.I3}
\begin{aligned}
  I_3 & \le C \varepsilon\big\{ \|E_1\|_{L^\infty}+\|\nabla_xE_1\|_{L^\infty} + \|E_1\|_{L^\infty}\|\nabla_xA\|_{L^\infty}\big\} \\
  & \qquad \cdot \big\{ \|E_1\|_{L^\infty} + \varepsilon\|\nabla_xE_1\|_{L^\infty} + \varepsilon\|E_1\|_{L^\infty}\|\nabla_xA\|_{L^\infty}\big\} \\
  & \qquad \cdot (\|f\|_{L^2(\Omega)}+\|g\|_{H^{3/2}(\partial \Omega)}) \| G \|_{L^2(\Omega)}.
  \end{aligned}
\end{align}
Inserting the estimates for $I_1$-$I_3$ into \eqref{eq.weG}, we obtain 
\begin{equation*}
\begin{aligned}
    \bigg|\int_\Omega w_\varepsilon\cdot G dx\bigg| & \le C\varepsilon \Big\{ \|E_1\|_{L^\infty}+\|\nabla_xE_1\|_{L^\infty}+\varepsilon\|\nabla_xE_1\|_{L^\infty}^2 \\
    & \qquad + (\|E_1\|_{L^\infty}+\varepsilon\|E_1\|_{L^\infty}\|\nabla_xE_1\|_{L^\infty})\|\nabla_xA\|_{L^\infty}\\
    & \qquad + \varepsilon \|E_1\|_{L^\infty}^2\|\nabla_xA\|_{L^\infty}^2\Big\} \cdot (\|f\|_{L^2(\Omega)}+\|g\|_{H^{3/2}(\partial\Omega)}) \| G \|_{L^2(\Omega)}.
  \end{aligned}
\end{equation*}
By duality and \eqref{def.we}, we conclude  that
\begin{align*}
  \|u_\varepsilon-u_0\|_{L^2(\Omega)}& \le C\varepsilon \Big\{ \|E_1\|_{L^\infty}+\|\nabla_xE_1\|_{L^\infty}+\varepsilon\|\nabla_xE_1\|_{L^\infty}^2 \\
    & \qquad + (\|E_1\|_{L^\infty}+\varepsilon\|E_1\|_{L^\infty}\|\nabla_xE_1\|_{L^\infty})\|\nabla_xA\|_{L^\infty}\\
    & \qquad + \varepsilon \|E_1\|_{L^\infty}^2\|\nabla_xA\|_{L^\infty}^2\Big\} \cdot (\|f\|_{L^2(\Omega)}+\|g\|_{H^{3/2}(\partial\Omega)}),
\end{align*}
where we have also used the estimate
\begin{align}\label{conver_es_Su0}
  \begin{split}
    \|\varepsilon S_\varepsilon(\eta_\varepsilon \chi^\varepsilon \nabla u_0)\|_{L^2(\Omega)}&\leq C\varepsilon \|E_1\|_{L^\infty}\|\nabla u_0\|_{L^2(\Omega)}\\&\leq C\varepsilon \|E_1\|_{L^\infty} (\|f\|_{L^2(\Omega)}+\|g\|_{H^{1/2}(\partial\Omega)}).
  \end{split}
\end{align}
The proof is complete. 
\end{proof}

\subsection{Reiterated homogenization: quantitative version}
In this subsection, we prove the optimal convergence rates in Theorem \ref{thm.L2rates} by a quantitative version of reiterated homogenization.



\begin{proof}[Proof of Theorem \ref{thm.L2rates}]
In view of the proof of Theorem \ref{thm.qualitative}, we know that $u_{n,\e} \to u_\e$ in $H^1(\Omega)$ and $u_{n,0} \to u_0$ in $H^1(\Omega)$ as $n\to \infty$, where $u_{n,\e}$ is the weak solution of the approximate problem \eqref{eq_approximate} and $u_{n,0}$ is the solution of the corresponding homogenized equation \eqref{eq_approx_homo}. If we can establish a convergence rate for $u_{n,\e}$, which is uniform in $n$, then we obtain the same convergence rate of $u_\e$ by taking $n\to \infty$.

Fix $n \ge 1$ and consider $1\leq k\leq n$. Let $u_{n, \varepsilon}^{k}$ be the weak solution to
  \begin{align*}
    \begin{cases}
      -\mathrm{div}(A_n^{k}(x, \frac{x}{\varepsilon_1}, \cdots, \frac{x}{\varepsilon_{k}})\nabla u_{n, \varepsilon}^{k})=f&\text{in }\Omega,\\
      u_{n, \varepsilon}^{k}=g &\text{on }\partial\Omega,
    \end{cases}
  \end{align*}
  where $A_n^{k}$ is given by \eqref{expression_Ank} with the convention $A_n^n = A_n$. 
  We will apply Theorem \ref{thm_n=1_general} and the method of reiterated homogenization. 
  
    Denote $A'(x, y)=E_0(x)+E_1(x, y)$, where
    \begin{align*}
        E_0(x) & =\sum_{\ell=0}^{k-1}B_\ell\Big(x, \frac{x}{\varepsilon_1}, \cdots, \frac{x}{\varepsilon_{\ell}}\Big),\\ E_1(x, y) & =B_k^n\Big(x, \frac{x}{\varepsilon_1}, \cdots, \frac{x}{\varepsilon_{k-1}}, y\Big).
    \end{align*}
    Then $A_n^k\big(x, {x}/{\varepsilon_1}, \cdots, 
    {x}/{\varepsilon_{k}}\big)=A'\big(x, {x}/{\varepsilon_k}\big)$. Thanks to \eqref{es_Ank_bound} and \eqref{es_Bkn_bound}, $A'$ satisfies the conditions in the first two lines of \eqref{conver_cond_case-n=1}, with a universal ellipticity constant $\mu'$ depending only on $d$ and $\mu$. Moreover, by a direct calculation, we deduce from Lemmas \ref{qual-homog_lemma_bkn} and  \ref{conver_lemma_dkn} that
    \begin{align*}
    \begin{split}
      \|E_1\|_{L^\infty}&\leq \|B_k^n\|_{L^\infty}\leq C\sum_{\ell=k}^n\delta_\ell\leq C R_k,\\
    \|\nabla_x E_1\|_{L^\infty}&\leq \sum_{j=0}^{k-1}\varepsilon_j^{-1}\|\nabla_{y_j}B_k^n\|_{L^\infty}\leq \sum_{j=0}^{k-1}\varepsilon_j^{-1}\delta_k^n\\&\leq \exp\big\{C_0[\delta]_1\big\}\big\{ 1+C_0[\delta]_0[\delta]_1\big\}\cdot R_k^n\sum_{j=0}^{k-1}\varepsilon_j^{-1} \\&=\Lambda_0R_k\sum_{j=0}^{k-1}\varepsilon_j^{-1},\\
    \|\nabla_x A'\|_{L^\infty}&\leq \sum_{j=0}^{k-1}\varepsilon_j^{-1}\|\nabla_{y_j}A_n^k\|_{L^\infty}\leq C\sum_{j=0}^{k-1}\varepsilon_j^{-1}\sum_{\ell=j}^{n}\delta_\ell\leq C[\delta]_0\sum_{j=0}^{k-1}\varepsilon_j^{-1},
    \end{split}
    \end{align*}
    where $R_k^n$ and $R_k$ are defined by \eqref{conver_def_Rkn} and \eqref{conver_def_Rk} respectively, $C$ depends only on $d$ and $\mu$, and 
    \begin{align*}
        \Lambda_0:=\exp\big\{C_0[\delta]_1\big\}\big\{ 1+C_0[\delta]_0[\delta]_1\big\}.
    \end{align*}

  Observe that, during the one-step homogenization, $u_{n, \varepsilon}^k$ converges to $u_{n, \varepsilon}^{k-1}$ as $\varepsilon_k\rightarrow 0$. Applying Theorem \ref{thm_n=1_general} to this case, we have
  \begin{align}\label{conver_es_unk}
    \begin{split}
      &\|u_{n, \varepsilon}^k-u_{n, \varepsilon}^{k-1}\|_{L^2(\Omega)} \\
      & \leq C\varepsilon_k \Big\{ \|E_1\|_{L^\infty}+\|\nabla_xE_1\|_{L^\infty}+\varepsilon_k\|\nabla_xE_1\|_{L^\infty}^2\\
      &\qquad +(\|E_1\|_{L^\infty}+\varepsilon_k\|E_1\|_{L^\infty}\|\nabla_xE_1\|_{L^\infty})\|\nabla_xA'\|_{L^\infty}\\
      &\qquad + \varepsilon_k \|E_1\|_{L^\infty}^2\|\nabla_xA'\|_{L^\infty}^2 \Big\} \cdot\Big(\|f\|_{L^2(\Omega)}+\|g\|_{H^{3/2}(\partial\Omega)}\Big)\\
      &\leq C\varepsilon_k\bigg\{R_k+\Lambda_0 R_k\sum_{j=0}^{k-1}\varepsilon_j^{-1}+\varepsilon_k\bigg(\Lambda_0 R_k\sum_{j=0}^{k-1}\varepsilon_j^{-1}\bigg)^2 \\ & \qquad  +\bigg(R_k+\varepsilon_k\Lambda_0 R_k^2 \sum_{j=0}^{k-1}\varepsilon_j^{-1}\bigg)[\delta]_0\sum_{j=0}^{k-1}\varepsilon_j^{-1} + \varepsilon_k\bigg([\delta]_0 R_k\sum_{j=0}^{k-1}\varepsilon_j^{-1}\bigg)^2\bigg\}\\
      & \qquad \qquad \cdot \Big(\|f\|_{L^2(\Omega)}+\|g\|_{H^{3/2}(\partial\Omega)}\Big) \\ 
      &\leq C\bigg\{\Lambda R_k\sum_{j=0}^{k-1}\frac{\varepsilon_k}{\varepsilon_j}+\bigg(\Lambda R_k\sum_{j=0}^{k-1}\frac{\varepsilon_k}{\varepsilon_j}\bigg)^2\bigg\}\cdot\Big(\|f\|_{L^2(\Omega)}+\|g\|_{H^{3/2}(\partial\Omega)}\Big),
    \end{split}
  \end{align}
  where
  \begin{align}\label{def.Lambda}
    \Lambda=\Lambda_0+[\delta]_0,
  \end{align}
and $C$ depends only on $d, \mu $ and $\Omega$. The constant $\Lambda \ge 1$ will be fixed throughout. Note that $u_{n, \varepsilon}^n=u_{n, \varepsilon}$ and $u_{n, \varepsilon}^0=u_{n, 0}$. On the other hand, by the Poincar\'{e} inequality and the energy estimate, we trivially have
\begin{align}\label{conver_es_unk_energy}
\|u_{n, \varepsilon}^k-u_{n, \varepsilon}^{k-1}\|_{L^2(\Omega)} \leq C (\|f\|_{L^2(\Omega)}+\|g\|_{H^{3/2}(\partial\Omega)}),
\end{align}
where $C$ depends only on $d, \mu $ and $\Omega$. Combining \eqref{conver_es_unk} and \eqref{conver_es_unk_energy}, and by considering $\Lambda R_k\sum_{j=0}^{k-1}\frac{\varepsilon_k}{\varepsilon_j}\geq 1$ and $\Lambda R_k\sum_{j=0}^{k-1}\frac{\varepsilon_k}{\varepsilon_j}\leq 1$ separately, the estimate \eqref{conver_es_unk} can be improved to
\begin{align}\label{conver_es_unk_combine}
  \|u_{n, \varepsilon}^k-u_{n, \varepsilon}^{k-1}\|_{L^2(\Omega)}\leq C\Lambda R_k\sum_{j=0}^{k-1}\frac{\varepsilon_k}{\varepsilon_j}\cdot\Big(\|f\|_{L^2(\Omega)}+\|g\|_{H^{3/2}(\partial\Omega)}\Big).
\end{align}
Summing \eqref{conver_es_unk_combine} in $k$ from $1$ to $n$, we obtain that
\begin{equation}\label{est.une-un0}
\aligned
  \|u_{n,\varepsilon}-u_{n, 0}\|_{L^2(\Omega)}&\leq \sum_{k=1}^n\|u_{n, \varepsilon}^k-u_{n, \varepsilon}^{k-1}\|_{L^2(\Omega)}\\&\leq C\sum_{k=1}^n \Lambda R_k\sum_{j=0}^{k-1}\frac{\varepsilon_k}{\varepsilon_j}\cdot\Big(\|f\|_{L^2(\Omega)}+\|g\|_{H^{3/2}(\partial\Omega)}\Big)\\&\leq C \Lambda[\delta]_1\sup_{1\le k\le n} \sum_{j = 0}^{k-1} \frac{\e_k}{\e_j}\cdot \Big(\|f\|_{L^2(\Omega)}+\|g\|_{H^{3/2}(\partial\Omega)}\Big),
  \endaligned
\end{equation}
where we have used \eqref{recursive_es_Rk} with $\alpha = 0$ in the last step. It follows from \eqref{qual-homog_conver_u0} and \eqref{qual-homog_conver_uepsilon} that $u_{n, \varepsilon}\rightarrow u_\varepsilon$ and $u_{n, 0}\rightarrow u_0$ in $H^1(\Omega)$ as $n\rightarrow\infty$. Hence, taking the limit in \eqref{est.une-un0} as $n\to \infty$, one gets
\begin{align}\label{es_rate_sumratio}
  \|u_\varepsilon-u_0\|_{L^2(\Omega)}\leq C\Lambda[\delta]_1\sup_{k\geq 1} \sum_{j = 0}^{k-1} \frac{\e_k}{\e_j}\cdot\Big(\|f\|_{L^2(\Omega)}+\|g\|_{H^{3/2}(\partial\Omega)}\Big).
\end{align}

Finally, we show that \eqref{es_rate_sumratio} implies the desired estimate. Let $\mathcal{E}:= \sup_{k\geq 1}\frac{\varepsilon_k}{\varepsilon_{k-1}}$. If $\mathcal{E} < 1/2$, then $\e_k \le \mathcal{E}^{k-j} \e_j$ for any $k > j \ge 0$.
It follows that
\begin{equation*}
    \sum_{j=0}^{k-1} \frac{\e_k}{\e_j} \le \sum_{j=0}^{k-1} \mathcal{E}^{k - j} \le \frac{\mathcal{E}}{1 - \mathcal{E}} \le 2\mathcal{E}.
\end{equation*}
which, together with \eqref{es_rate_sumratio}, implies  \eqref{est.rate1}. Now if $\mathcal{E} \ge 1/2$, then \eqref{est.rate1} is a consequence of $H$-stability. In fact, by the observation $\widehat{A}_0 = B_0$ and Theorem \ref{qual-homog_thm_hatA}, we have
\begin{align*}
  \|A^\varepsilon-\widehat{A}\|_{L^\infty}\leq \|A^\varepsilon-B_0\|_{L^\infty}+\|B_0-\widehat{A}\|_{L^\infty}\leq C\sum_{k=1}^\infty \delta_k \le C[\delta]_1,
\end{align*}
from which the energy estimate yields
\begin{equation*}\label{conver_es_rate_energy}
    \|u_\varepsilon-u_0\|_{L^2(\Omega)} \leq C\Lambda [\delta]_1 \{\|f\|_{L^2(\Omega)}+\|g\|_{H^{1}(\partial\Omega)}\}.
\end{equation*}
This concludes \eqref{est.rate1} for the case $\mathcal{E} \ge 1/2$ and therefore ends the proof.
\end{proof}

\subsection{Suboptimal convergence rate in Lipschitz domains}


To prove the uniform Lipschitz estimate, we establish a suboptimal algebraic convergence rate in a Lipschitz domain with less regular boundary data. Again, we begin with a one-step convergence rate.
\begin{theorem}\label{thm_suboptimal}
  Suppose $\Omega$ is a bounded Lipschitz domain. Let $u_\varepsilon$ be the solution of equation \eqref{eq_n=1_general} and $u_0$ the solution of the homogenized equation \eqref{eq_n=1.homo}. Then there exists $\sigma>0$, depending only on $\mu $ and $\Omega$, such that
  \begin{align*}
    \|u_\varepsilon-u_0\|_{L^2(\Omega)}&\leq C\big\{ \varepsilon^{\sigma}\|E_1\|_{L^\infty}+\varepsilon\|\nabla_x E_1\|_{L^\infty}+\varepsilon\|E_1\|_{L^\infty}\|\nabla_xA\|_{L^\infty} \big\} \\
    & \qquad \times (\|f\|_{L^2(\Omega)} + \|g\|_{H^1(\partial\Omega)} ),
  \end{align*}
  where $C$ depends only on $d, \mu$ and $\Omega$. 
\end{theorem}

\begin{proof}
    First of all, using the Poincar\'{e} inequality as well as \eqref{conver_es_Su0}, we derive from \eqref{conver_es_wH1_0} that
    \begin{align}\label{conver_es_n=1_suboptimal}
    \begin{split}
        &\|u_\varepsilon-u_0\|_{L^2(\Omega)}\leq C\Big\{ \|E_1\|_{L^\infty}\big(\varepsilon\|\nabla^2 u_0\|_{L^2(\Omega\setminus \Omega_{3\varepsilon})}+\|\nabla u_0\|_{L^2(\Omega_{4\varepsilon})}\big) \\
        &\qquad + \varepsilon\big(\|\nabla_xA\|_{L^\infty}\|E_1\|_{L^\infty} + \|\nabla_x E_1\|_{L^\infty}+\|E_1\|_{L^\infty}\big)\|\nabla u_0\|_{L^2(\Omega)} \Big\}.
    \end{split}
    \end{align}
    The estimates for $\|\nabla^2 u_0\|_{L^2(\Omega\setminus \Omega_{3\varepsilon})}$ and $\|\nabla u_0\|_{L^2(\Omega_{4\varepsilon})}$ are more or less standard. We provide the details with emphasis on the dependence of $A$ as required in the desired estimates.
    
    According to Meyers' estimate (see \cite{Meyers1963, Giaquinta1983_book}), there exists $s>2$, depending only on $\mu$ and $\Omega$, such that
    \begin{align}\label{conver_es_Meyers}
        \|\nabla u_0\|_{L^s(\Omega)}\leq C( \|f\|_{L^2(\Omega)}+\|g\|_{H^1(\partial\Omega)}),
    \end{align}
    where $C$ depends only on $\mu$ and $\Omega$. By H\"{o}lder's inequality, we have
    \begin{align}\label{est.Du0.Oe}
    \begin{aligned}
        \|\nabla u_0\|_{L^2(\Omega_{4\varepsilon})}& \leq C\varepsilon^{\frac{1}{2}-\frac{1}{s}}\|\nabla u_0\|_{L^{s}(\Omega)} \\
        & \leq C\varepsilon^{\frac{1}{2}-\frac{1}{s}}\big\{\|f\|_{L^2(\Omega)}+\|g\|_{H^1(\partial\Omega)}\big\}. 
    \end{aligned}
    \end{align}
    By the interior $H^2$ estimate for elliptic equations with Lipschitz coefficients, it holds that for any $x\in \Omega$,
    \begin{align*}
        & \fint_{B(x, \omega(x)/8)}|\nabla^2 u_0|^2 \\
        & \leq C\Big([\omega(x)]^{-2}+\|\nabla\widehat{A}\|_{L^\infty}^2\Big)\fint_{B(x, \omega(x)/4)}|\nabla u_0|^2+C\fint_{B(x, \omega(x)/4)}|f|^2,
    \end{align*}
    where $\omega(x)=\mathrm{dist}(x, \partial \Omega)$. Integrating both sides of the above inequality in $x$ over $\Omega\setminus\Omega_{3\varepsilon}$, we get
    \begin{align*}
        \int_{\Omega\setminus\Omega_{3\varepsilon}}|\nabla^2 u_0|^2&\leq C\int_{\Omega\setminus\Omega_{\varepsilon}}[\omega(y)]^{-2}|\nabla u_0|^2dy+C\|\nabla_xA\|_{L^\infty}^2\int_\Omega|\nabla u_0|^2+C\int_\Omega |f|^2\\&\leq C\Big(\int_\Omega|\nabla u_0|^{s}\Big)^{2/s}\Big(\int_{\Omega\setminus\Omega_\varepsilon}[\omega(y)]^{-\frac{2s}{s-2}}dy\Big)^{\frac{s-2}{s}}\\&\qquad+C\|\nabla_xA\|_{L^\infty}^2\int_\Omega|\nabla u_0|^2+C\int_\Omega|f|^2\\&\leq C\varepsilon^{-1-\frac{2}{s}}\|\nabla u_0\|_{L^{s}(\Omega)}^2+C\|\nabla_xA\|_{L^\infty}^2\int_\Omega|\nabla u_0|^2+C\int_\Omega|f|^2,
    \end{align*}
    where we have used \eqref{conver_es_A-Lip} in the first step. Therefore,
    \begin{align}
    \begin{aligned} \label{est.DDu0.Oec}
        \varepsilon\|\nabla^2 u_0\|_{L^2(\Omega\setminus \Omega_{3\varepsilon})}&\leq C\varepsilon^{\frac{1}{2}-\frac{1}{s}}\|\nabla u_0\|_{L^{s}(\Omega)}+C\varepsilon\|\nabla_xA\|_{L^\infty}\|\nabla u_0\|_{L^2(\Omega)}\\&\qquad+C\varepsilon\|f\|_{L^2(\Omega)}\\&\leq C(\varepsilon^{\frac{1}{2}-\frac{1}{s}}+\varepsilon\|\nabla_xA\|_{L^\infty})\big\{\|f\|_{L^2(\Omega)}+\|g\|_{H^1(\partial\Omega)}\big\},
    \end{aligned}
    \end{align}
        where we have used \eqref{conver_es_Meyers}. 

        Finally, substituting \eqref{est.Du0.Oe} and \eqref{est.DDu0.Oec} into \eqref{conver_es_n=1_suboptimal}, we obtain the desired result with $\sigma=\frac{1}{2}-\frac{1}{s}$. 
\end{proof}

\begin{theorem}\label{conver_thm_rate_suboptimal}
Let $\Omega$ be a bounded Lipschitz domain, $u_\varepsilon$ and $u_0$ be the solutions to \eqref{eq_infinite} and \eqref{eq_homogenized}, respectively.
Suppose that $\delta = \{\delta_j\}_{j\geq 0}$ satisfies $[\delta]_1<\infty$.  Then there exists $\sigma>0$, depending only on $\mu$ and $\Omega$, such that for any $0<\alpha\leq 1$,
  \begin{align}\label{conver_es_sub_rate}
    \|u_\varepsilon-u_0\|_{L^2(\Omega)}\leq C\Lambda [\delta]_\alpha \bigg( \sup_{k\geq 1} \frac{\e_k}{\e_{k-1}} \bigg)^{\alpha\sigma} \Big(\|f\|_{L^2(\Omega)}+\|g\|_{H^1(\partial\Omega)}\Big),
  \end{align}
  where $\Lambda$ is defined by \eqref{def.Lambda} and $C$ depends only on $d, \mu$ and $\Omega$.
\end{theorem}

\begin{proof}
    Let $\mathcal{E} = \sup_{k\ge 1} \frac{\e_k}{\e_{k-1}}$. As before, without loss of generality, assume $\mathcal{E} < 1/2$. Then
    \begin{equation*}
        \sup_{k\ge 1} \sum_{j=0}^{k-1} \frac{\e_k}{\e_j} \le 2\mathcal{E}.
    \end{equation*}
    Fix $\alpha \in (0,1]$. Let $m = \lfloor \mathcal{E}^{-\sigma} \rfloor$, where $\sigma \in (0,1)$ is given as in Theorem \ref{thm_suboptimal}. First, note that
    \begin{equation*}
        \| A^\e - A_m^\e \|_{L^\infty(\Omega)} \le \sum_{k=m+1}^\infty \delta_k \le [\delta]_\alpha (m+1)^{-\alpha} \le [\delta]_\alpha \mathcal{E}^{\alpha \sigma}.
    \end{equation*}
    By the stability of $H$-convergence in Lemma \ref{qual-homog_lem_stability}, we have
    \begin{equation*}
        \| \widehat{A} - \widehat{A}_m \|_{L^\infty(\Omega)} \le C [\delta]_\alpha \mathcal{E}^{\alpha \sigma}.
    \end{equation*}
    Therefore, by the energy estimates, we obtain
    \begin{equation*}
        \| u_\e - u_{m,\e} \|_{L^2(\Omega)} + \| u_0 - u_{m,0} \|_{L^2(\Omega)} \le C [\delta]_\alpha \mathcal{E}^{\alpha \sigma} (\|f\|_{L^2(\Omega)}+\|g\|_{H^1(\partial\Omega)}).
    \end{equation*}
    Hence, to prove the desired estimate, it suffices to show
    \begin{equation}\label{est.um}
        \| u_{m,\e} - u_{m,0} \|_{L^2(\Omega)} \le C [\delta]_\alpha \mathcal{E}^{\alpha \sigma} (\|f\|_{L^2(\Omega)}+\|g\|_{H^1(\partial\Omega)}).
    \end{equation}

    To this end, let $1\le k \le m$. As in the proof of Theorem \ref{thm.L2rates}, using Theorem \ref{thm_suboptimal} instead of Theorem \ref{thm_n=1_general}, we deduce that 
    \begin{align*}
        \|u_{m, \varepsilon}^k-u_{m, \varepsilon}^{k-1}\|_{L^2(\Omega)} & \leq C\big(\varepsilon_k^{\sigma}\|E_1\|_{L^\infty}+\varepsilon_k\|\nabla_x E_1\|_{L^\infty}+\varepsilon_k\|E_1\|_{L^\infty}\|\nabla_xA'\|_{L^\infty}\big) \\
        &\qquad\times (\|f\|_{L^2(\Omega)}+\|g\|_{H^1(\partial\Omega)})\\&\leq C\bigg\{\varepsilon_k^\sigma R_k+\Lambda_0R_k\sum_{j=0}^{k-1}\frac{\varepsilon_k}{\varepsilon_j}+R_k[\delta]_0\sum_{j=0}^{k-1}\frac{\varepsilon_k}{\varepsilon_j}\bigg\}\\
        &\qquad \times (\|f\|_{L^2(\Omega)}+\|g\|_{H^1(\partial\Omega)})\\&\leq C\big\{ R_k \mathcal{E}^\sigma+\Lambda R_k \mathcal{E} \big\} (\|f\|_{L^2(\Omega)}+\|g\|_{H^1(\partial\Omega)}).
\end{align*}
Summing in $k$ from $1$ to $m$ gives
\begin{align*}
  \|u_{m, \varepsilon}-u_{m, 0}\|_{L^2(\Omega)} & \leq C\Lambda \mathcal{E}^\sigma \sum_{k=1}^m R_k \Big(\|f\|_{L^2(\Omega)}+\|g\|_{H^1(\partial\Omega)}\Big)\\
  & \leq C \Lambda \mathcal{E}^\sigma [\delta]_\alpha m^{1-\alpha} \big(\|f\|_{L^2(\Omega)}+\|g\|_{H^1(\partial\Omega)}\big),
\end{align*}
where we have used the fact
\begin{equation*}
    \sum_{k=1}^m R_k = \sum_{j = 1}^\infty \delta_j \sum_{k=1}^{\min\{m, j \}} 1 \le [\delta]_\alpha m^{1-\alpha}.
\end{equation*}
In view of our choice of $m$, we obtain \eqref{est.um}. Note that the constant $C$ is independent of $\alpha$.
\end{proof}

\begin{remark}
    By mimicking the proof of Theorem \ref{conver_thm_rate_suboptimal}, we can also show a similar result in $C^{1,1}$ domains parallel to Theorem \ref{thm.L2rates}. In fact, under the assumptions of Theorem \ref{thm.L2rates}, for any $\alpha \in (0,1)$, we have
    \begin{equation*}
        \quad \|u_\varepsilon-u_0\|_{L^2(\Omega)}\leq C\Lambda [\delta]_\alpha \bigg( \sup_{k\geq 1} \frac{\varepsilon_k}{\varepsilon_{k-1}} \bigg)^\alpha \cdot(\|f\|_{L^2(\Omega)}+\|g\|_{H^{3/2}(\partial\Omega)}),
    \end{equation*}
    where $C$ depends only on $d, \mu$ and $\Omega$.
\end{remark}

\section{Interior Lipschitz estimates}\label{sec_interior-lip}

In this section we establish the interior Lipschitz estimates for $u_\varepsilon$. In the quantitative Campanato-type iteration, we approximate the solution by affine functions at any given scale with controlled errors. This involves the rescaling argument. Since we have infinitely many oscillating scales, we must ensure that certain properties of the coefficients are preserved when we zoom in from large scales to small scales.

\subsection{Rescaling}\label{sec_rescaling}
Let $u_\varepsilon$ be a weak solution to
\begin{align}\label{eq.inB1}
  -\mathrm{div}(A^\varepsilon(x)\nabla u_\varepsilon)=f\quad \text{in }B_1,
\end{align}
where $B_1=B(0, 1)$. We consider rescaling properties of the equation and show the rescaled equations have similar characters with universal upper bounds. 

Let $m\in \mathbb{N}_+$ and we zoom in to $\e_m$-scale. Let $v_m(x) = u_\e(\varepsilon_m x)$, which satisfies
\begin{align*}
  -\mathrm{div}\Big(A\Big(\frac{\varepsilon_mx}{\varepsilon_0}, \frac{\varepsilon_m x}{\varepsilon_1},\cdots, \frac{\varepsilon_mx}{\varepsilon_m}, \frac{x}{\varepsilon_{m+1}/\varepsilon_{m}}, \cdots\Big)\nabla v_m\Big)=\varepsilon_m^2f(\varepsilon_m x)\quad \text{in }B_{1/\varepsilon_m}.
\end{align*}
Note that if $j\le m$, then $\e_m/\e_j \le 1$. Thus, in the rescaled coefficient matrix, the first $m+1$ variables are all slow variables. 
Relabel $\widetilde{y}_j = y_{j+m}$ for $j\geq 1$ and define 
\begin{align*}
  \widetilde{A}(x, \widetilde{y}_{1}, \widetilde{y}_2, \cdots)=A\Big(\frac{\varepsilon_mx}{\varepsilon_0}, \frac{\varepsilon_m x}{\varepsilon_1},\cdots, \frac{\varepsilon_mx}{\varepsilon_m}, \widetilde{y}_1, \widetilde{y}_2, \cdots \Big).
\end{align*}
Then $\widetilde{A}$ possesses the same structure as $A$, which can be written as
\begin{align*}
  \widetilde{A}(x, \widetilde{y}_{1}, \cdots)=\sum_{\ell=0}^\infty\widetilde{B}_\ell(x, \widetilde{y}_{1}, \cdots, \widetilde{y}_{\ell}),
\end{align*}
where
\begin{align*}
  \widetilde{B}_0(x)&=\sum_{j=0}^m B_j \Big(\frac{\varepsilon_mx}{\varepsilon_0}, \frac{\varepsilon_m x}{\varepsilon_1},\cdots, \frac{\varepsilon_mx}{\varepsilon_j} \Big),\\ 
  \widetilde{B}_\ell(x, \widetilde{y}_1, \cdots, \widetilde{y}_\ell)&=B_{\ell+m} \Big( \frac{\varepsilon_mx}{\varepsilon_0}, \frac{\varepsilon_m x}{\varepsilon_1},\cdots, \frac{\varepsilon_mx}{\varepsilon_m}, \widetilde{y}_1, \cdots, \widetilde{y}_{\ell} \Big),
\end{align*}
for $\ell\geq 1$. Denote also $\widetilde{\delta}_\ell:=\max\{\|\widetilde{B}_\ell\|_{L^\infty}, \max_{0\leq j\le \ell}\|\nabla_{\widetilde{y}_j}\widetilde{B}_\ell\|_{L^\infty}\}$. By $\widetilde{y}_j = y_{j+m}$ and a direct computation, we have
\begin{align*}
  \widetilde{\delta}_0&\leq \max\bigg\{\sum_{\ell=0}^m\delta_\ell, \sum_{\ell=0}^m\sum_{j=0}^{\ell}\frac{\varepsilon_m}{\varepsilon_j}\delta_\ell\bigg\}\le \sum_{j=0}^m\frac{\varepsilon_m}{\varepsilon_j}\sum_{\ell=0}^m\delta_\ell,\\
  \widetilde{\delta}_{\ell}&\leq \max\bigg\{\delta_{\ell+m}, \sum_{j=0}^{m}\frac{\varepsilon_m}{\varepsilon_j}\delta_{\ell+m}\bigg\}=\sum_{j=0}^m\frac{\varepsilon_m}{\varepsilon_j}\delta_{\ell+m}\quad \text{for }\ell\geq 1.
\end{align*}
Therefore,
\begin{align}\label{int-Lip_es_tilde0}
  [\widetilde{\delta}]_0:=\sum_{k=0}^\infty \widetilde{\delta}_k\leq \sum_{j=0}^m\frac{\varepsilon_m}{\varepsilon_j}\sum_{k=0}^m\delta_k+\sum_{j=0}^m\frac{\varepsilon_m}{\varepsilon_j} \sum_{k=1}^\infty\delta_{k+m}=\sum_{j=0}^m\frac{\varepsilon_m}{\varepsilon_j}\sum_{k=0}^\infty \delta_k, 
\end{align}
and for $\alpha>0$
\begin{align}\label{int-Lip_es_tilde2}
  [\widetilde{\delta}]_\alpha:=\sum_{k=1}^\infty k^\alpha\widetilde{\delta}_k\leq \sum_{j=0}^m\frac{\varepsilon_m}{\varepsilon_j}\cdot\sum_{k=1}^\infty k^\alpha\delta_{k+m}.
\end{align}
Note that $\widetilde{A}, \widetilde{B}_\ell$ and $\widetilde{\delta}_\ell$ depend on $m$.

In the sequel, we impose a convenient scale-invariant condition for later calculation: there exists some fixed $K\in (1, 2]$ such that
\begin{align}
  \label{int-Lip_cond_sumbound}
  \sum_{j=1}^{m}\frac{\varepsilon_m}{\varepsilon_j}\leq K\quad \text{for each }m \in \N_+,
\end{align}
The constant $K$ will be chosen close to $1$ later.
It can be seen from \eqref{int-Lip_es_tilde0}--\eqref{int-Lip_es_tilde2} that for each $\alpha\geq 0$, $[\widetilde{\delta}]_\alpha$ has a universal upper bound independent of $m$, i.e.,
\begin{align}\label{int-lip_es_tildedelta_bound}
  [\widetilde{\delta}]_\alpha\leq (K+1)[\delta]_\alpha.
\end{align}
Moreover, the approximate matrix $\widetilde{A}_n=\sum_{\ell=0}^n\widetilde{B}_\ell$ of $n$ microscopic scales obviously satisfies condition \eqref{cond_ellipticity}--\eqref{cond_periodicity}.

From this perspective, if an estimate holds for $u_\varepsilon$ with constants depending on the characters of $A$, then the same estimate is also valid for $v_m$ with new constants depending on the corresponding characters of $\widetilde{A}$. Crucially, we are able to show that the constants in the estimates from Section \ref{sec_large-scale_u} depend monotonically on $[\delta]_\alpha$. Since $[\widetilde{\delta}]_\alpha$ admits a uniform upper bound via \eqref{int-lip_es_tildedelta_bound}, this allows us to establish a series of estimates for $v_m$ that are uniform in $m$.

As a preparation, we show that the condition \eqref{int-Lip_cond_sumbound} implies the exponential decay of $\e_k$.

\begin{lemma}\label{int-Lip_lem_varepsilon}
  Let $1=\varepsilon_0> \varepsilon_1>\varepsilon_2>\cdots$ be a sequence of positive numbers. Suppose \eqref{int-Lip_cond_sumbound} is fulfilled for $m\geq 1$. Then, for $1\leq k\leq m$,
  \begin{align}\label{est.em/ek}
    \frac{\varepsilon_m}{\varepsilon_{k}}\leq K(1-K^{-1})^{m-k}.
  \end{align}
\end{lemma}
\begin{proof}
  By considering instead the sequence $1=\varepsilon_{k-1}/\varepsilon_{k-1}>\varepsilon_{k}/\varepsilon_{k-1}>\cdots$, which still satisfies
  \begin{align*}
    \sum_{j=k}^{m} \frac{\varepsilon_m/\varepsilon_{k-1}}{\varepsilon_j/\varepsilon_{k-1}}=\sum_{j=k}^{m}\frac{\varepsilon_m}{\varepsilon_j}\leq K\quad \text{for each }m\geq k,
  \end{align*}
  it is sufficient to focus on the case $k=1$. 
  
  According to \eqref{int-Lip_cond_sumbound}, it holds $\varepsilon_m^{-1}\geq K^{-1}\sum_{j=1}^{m} \varepsilon_j^{-1}$. Thus, for each $m\geq 2$
\begin{align*}
  \sum_{j=1}^{m} \varepsilon_j^{-1}-\sum_{j=1}^{m-1} \varepsilon_j^{-1}&=\varepsilon_{m}^{-1}\geq K^{-1}\sum_{j=1}^{m} \varepsilon_j^{-1},\\
  \sum_{j=1}^{m} \varepsilon_j^{-1}&\geq (1-K^{-1})^{-1}\sum_{j=1}^{m-1} \varepsilon_j^{-1}.
\end{align*}
By iteration, we deduce that
\begin{align*}
  \sum_{j=1}^{m} \varepsilon_j^{-1}\geq (1-K^{-1})^{1-m}\varepsilon_1^{-1},
\end{align*}
which implies
\begin{align*}
  \varepsilon_m\leq K\bigg(\sum_{j=1}^{m} \varepsilon_j^{-1}\bigg)^{-1}\leq K(1-K^{-1})^{m-1}\varepsilon_1. 
\end{align*}
The case $m=1$ is trivial. 
\end{proof}

\begin{remark}\label{int-lip_remark_decay}
  Conversely, the decay of $\varepsilon_k$ in \eqref{est.em/ek} also implies the uniform boundedness of $\sum_{k=1}^m\frac{\varepsilon_m}{\varepsilon_k}$. In fact, we have
  \begin{align*}
      \sum_{k=1}^m\frac{\varepsilon_m}{\varepsilon_k}\leq \sum_{k=1}^m K(1-K^{-1})^{m-k}\leq K^2. 
  \end{align*}
  Furthermore, for $m\geq k$
 \begin{align*}
    \sum_{\ell=m}^\infty\frac{\varepsilon_\ell}{\varepsilon_k}\leq K\sum_{\ell=m}^\infty(1-K^{-1})^{\ell-k}\leq K^2(1-K^{-1})^{m-k}.
  \end{align*}
\end{remark}

\subsection{Approximation}
As a preliminary step, we establish an approximate result for $u_\varepsilon$ by $u_0$ based on the suboptimal convergence rate in Theorem \ref{conver_thm_rate_suboptimal}, following an approach used in \cite[Lemma 6.1]{Niu25_Optimal}.

\begin{theorem}\label{thm_approx}
Let $u_\varepsilon$ be a weak solution of \eqref{eq.inB1} with $f\in L^2(B_{1})$. Assume $[\delta]_1<\infty$. Then for each $r\in[\varepsilon_1, 1/2]$, there exists a weak solution $u_0$ to the equation
$$-\mathrm{div}(\widehat{A} \nabla u_0)=f \quad \text{ in } B_r,$$ 
such that for any $0<\alpha\leq 1$,
  \begin{align*}
    \bigg(\fint_{B_r}|u_\varepsilon-u_0|^2\bigg)^{1/2}&\leq C\Lambda [\delta]_\alpha\bigg(\frac{\varepsilon_1}{r}+\sup_{k\geq 2}\frac{\e_k}{\e_{k-1}}\bigg)^{\alpha\sigma} \\&\quad\times \bigg\{\bigg(\fint_{B_{2r}}|u_\varepsilon|^2\bigg)^{1/2}+r^2\bigg(\fint_{B_{2r}}|f|^2\bigg)^{1/2}\bigg\},
  \end{align*}
  where $\Lambda $ is given by \eqref{def.Lambda}, and $\sigma$ and $C$ depend only on $d$ and $\mu$.
\end{theorem}

\begin{proof}
  Since the desired estimate is scale-invariant for $r\in [\e_1, 1/2]$, without loss of generality, we can assume $r=1/2$. Indeed, by setting $\widetilde{\e} = \e/r, \widetilde{A}(y_0, y_1,\cdots) = A(ry_0, y_1,\cdots )$ and $v_{\widetilde{\e}}(x) = u_\varepsilon(rx)$ for $r\in [\varepsilon_1, 1/2]$, we have 
  $$-\mathrm{div}(\widetilde{A}^{\widetilde{\e}}(x)\nabla v_{\widetilde{\e}}) = \widetilde{f} \quad \text{in } B_1,$$ 
  where $\widetilde{f}(x)=r^2f(rx)$. It is not hard to see that 
  $$\widetilde{A}(y_0, y_1, \cdots)=\sum_{\ell=0}^\infty B_\ell(ry_0, y_1, \cdots, y_\ell)$$ 
  and each term satisfies the condition \eqref{cond_B} with the same bound $\delta_\ell$. This reduces the problem to the case $r = 1/2$.

  For $(1/2)<t<(3/4)$, let $u_0^t$ be the weak solution to
\begin{align*}
  \begin{cases}
    -\mathrm{div}(\widehat{A}(x)\nabla u_0^t)=f &\text{in }B_t\\
    u_0^t=u_\varepsilon&\text{on }\partial B_t.
  \end{cases}
\end{align*}
By Theorem \ref{conver_thm_rate_suboptimal},
\begin{align}\label{est.ue-e0t}
  \|u_\varepsilon-u_0^t\|_{L^2(B_{1/2})}\leq C\Lambda [\delta]_\alpha\Big(\sup_{k\geq 1} \frac{\e_k}{\e_{k-1}}\Big)^{\alpha\sigma}\Big\{\|u_\varepsilon\|_{H^{1}(\partial B_t)}+\|f\|_{L^2(B_t)}\Big\}.
\end{align}
Let
\begin{align*}
  u_0=4\int_{1/2}^{3/4} u_0^tdt. 
\end{align*}
Note that $u_0$ satisfies the homogenized equation 
$$-\mathrm{div}(\widehat{A}(x)\nabla u_0)=f \quad \text{in } B_{1/2}.$$ 
It follows from \eqref{est.ue-e0t} and an integration in $t$ that
\begin{align*}
  \|u_\varepsilon-u_0\|_{L^2(B_{1/2})}^2&\leq 4\int_{1/2}^{3/4}\|u_\varepsilon-u_0^t\|_{L^2(B_{1/2})}^2dt\\
  & \leq C^2\Lambda^2 [\delta]_\alpha^2 \Big(\sup_{k\geq 1} \frac{\e_k}{\e_{k-1}}\Big)^{2\alpha\sigma} \int_{1/2}^{3/4} \Big(\|u_\varepsilon\|_{H^{1}(\partial B_t)}^2+\|f\|_{L^2(B_t)}^2\Big) dt\\
  & \leq C^2\Lambda^2 [\delta]_\alpha^2 \Big(\sup_{k\geq 1} \frac{\e_k}{\e_{k-1}}\Big)^{2\alpha\sigma}\Big(\|u_\varepsilon\|_{H^{1}(B_{3/4})}^2 + \|f\|_{L^2(B_{1})}^2\Big) \\
  & \leq C^2 \Lambda^2  [\delta]_\alpha^2 \Big(\sup_{k\geq 1} \frac{\e_k}{\e_{k-1}}\Big)^{2\alpha\sigma}\Big(\|u_\varepsilon\|_{L^2(B_{1})}^2 + \|f\|_{L^2(B_{1})}^2\Big),
\end{align*}
where we have used the co-area formula in the third inequality and the Caccioppoli inequality in the last inequality.
\end{proof}

\subsection{Large-scale estimates}\label{sec_large-scale_u}

In this subsection, we establish the large-scale estimates for $u_\varepsilon$ at scales 
between $\e_1$ and $1$. This is a one-step result that will be iterated to all  scales.

Let $u_\varepsilon\in H^1(B_1)$ be a solution to $-\mathrm{div}(A^\varepsilon(x)\nabla u_\varepsilon)=f$ in $B_1$, where $f\in L^p(B_1)$ for some $p>d$. For $0<r\leq 1$, define
\begin{align*}
  H(r; u_\varepsilon)=\frac{1}{r}\inf_{P\in \mathcal{P}}\bigg(\fint_{B_r}|u_\varepsilon-P|^2\bigg)^{1/2}+r\bigg(\fint_{B_r}|f|^p\bigg)^{1/p},
\end{align*}
where $\mathcal{P}$ denotes the linear space of affine functions. 
Let $P_r$ (minimizer) be the affine function achieving the infimum in $H(r; u_\varepsilon)$, i.e.,
\begin{align*}
  H(r; u_\varepsilon)=\frac{1}{r}\bigg(\fint_{B_r}|u_\varepsilon-P_r|^2\bigg)^{1/2}+r\bigg(\fint_{B_r}|f|^p\bigg)^{1/p},
\end{align*}
and set
\begin{align*}
  h(r; u_\varepsilon)=|\nabla P_r|.
\end{align*}

For the sake of brevity, in the sequel we may omit the identifier $u_\varepsilon$ in $H(r; u_\varepsilon)$ and $h(r; u_\varepsilon)$ if it causes  no confusion.

\begin{theorem}\label{thm_Lip_1step}
  Suppose that $[\delta]_1<\infty$ and $0<\alpha\leq 1$. There exists a constant $c_0 \in (0, 1)$, depending on $d, \mu, p, [\delta]_0 $ and $[\delta]_1$, such that if for any $k\geq 2$
  \begin{align}\label{cond_logseparation}
    \bigg(\frac{\varepsilon_k}{\varepsilon_{k-1}}\bigg)^{\alpha\sigma}\leq c_0|\log \varepsilon_1|^{-1},
  \end{align}
then for $r\in[\varepsilon_1, 1]$,
  \begin{align*}
    H(r; u_\varepsilon) & \leq C_0 r^\lambda H(1; u_\varepsilon) \\
    & \quad +CM_0\bigg[\Big(\frac{\varepsilon_1}{r}\Big)^{\alpha\sigma}+r^\lambda+\sup_{k\geq 2}\bigg(\frac{\varepsilon_k}{\varepsilon_{k-1}}\bigg)^{\alpha\sigma}\bigg]\Big\{H(1; u_\varepsilon)+h(1; u_\varepsilon)\Big\},
  \end{align*}
  where 
  \begin{equation}\label{def.M0}
      M_0:= \max\bigg\{[\delta]_\alpha, {\sum_{k=0}^\infty \|\nabla_{y_0}B_k\|_{L^\infty}}\bigg\},
  \end{equation}
  $\sigma$ is given in Theorem \ref{thm_approx}, $\lambda\in(0, 1)$, $C\geq C_0\geq 1$. Furthermore, the constants $\lambda $ and $C_0$ depend only on $d, \mu, p$ and $[\delta]_0$, and $C$ depends additionally on $\alpha, [\delta]_1$. 
\end{theorem}

\begin{remark}
    {It is important to keep track of the dependence of the constants on $[\delta]_0$ and $[\delta]_1$ in our proofs. However, there is one rule they all follow: large constants (such as $C_0$ and $C$ in the above theorem) depend increasingly on $[\delta]_0$ and/or $[\delta]_1$; small constants (such as $c_0$ and $\lambda$ in the above theorem) depend decreasingly on $[\delta]_0$ and/or $[\delta]_1$. Therefore, if $[\delta]_0$ and $[\delta]_1$ are uniformly bounded above (which turns out to be true during rescaling), then large constants all have a uniform upper bound and small constants all have a uniform lower bound, as we expected.}
\end{remark}

\begin{proof}[Proof of Theorem \ref{thm_Lip_1step}]
    According to Theorem \ref{thm_approx}, for any $r\in [\varepsilon_1, 1/2]$, there exists a weak solution $u_0$ to $-\mathrm{div}(\widehat{A} \nabla u_0)=f$ in $B_r$ such that
    \begin{align}\label{int-Lip_es_approx}
    \begin{split}
        \bigg(\fint_{B_r}|u_\varepsilon-u_0|^2\bigg)^{1/2}&\leq C\Lambda [\delta]_\alpha\bigg(\frac{\varepsilon_1}{r}+\sup_{k\geq 2}\frac{\e_k}{\e_{k-1}}\bigg)^{\alpha\sigma} \\&\quad\times\bigg\{\bigg(\fint_{B_{2r}}|u_\varepsilon-b|^2\bigg)^{1/2}+r^2\bigg(\fint_{B_{2r}}|f|^2\bigg)^{1/2}\bigg\},
    \end{split}
    \end{align}
    for any $b\in\mathbb{R}$, as subtracting a constant $b$ from a solution would not change the equation. By the $C^{1, 1-d/p}$ regularity of $u_0$ (see \cite{Niu20_reiterated, Niu25_Optimal}), for any $\theta\in(0, 1/2)$, 
    \begin{align}\label{est.C1a}
    \begin{aligned}
        &\frac{1}{\theta r}\inf_{P\in \mathcal{P}}\bigg(\fint_{B_{\theta r}}|u_0-P|^2\bigg)^{1/2}+\theta r\bigg(\fint_{B_{\theta r}}|f|^p\bigg)^{1/p}\leq C\theta^{1-d/p}\\&\quad\times\bigg[\frac{1}{r}\bigg(\fint_{B_{r}}|u_0-P_r|^2\bigg)^{1/2}+r\bigg(\fint_{B_{r}}|f|^p\bigg)^{1/p}+r\|\nabla_x\widehat{A}\|_{L^\infty(B_r)} h(r)\bigg],
    \end{aligned}
    \end{align}
    where $P_r$ is the minimizer of $H(r)$ and $C$ depends only on $d, \mu, p $ and $\|\nabla_x\widehat{A}\|_{L^\infty(B_r)}$.  Thanks to Theorem \ref{qual-homog_thm_hatA}, we have
    \begin{align}\label{est.DAhat}
        \|\nabla_x\widehat{A}\|_{L^\infty(B_1)}\leq C\sum_{k=0}^\infty \|\nabla_{y_0}B_k\|_{L^\infty}\leq C[\delta]_0.
    \end{align} 
    However, we will keep the extra $\|\nabla_x\widehat{A}\|_{L^\infty(B_1)}$ in \eqref{est.C1a}, whose smallness will be crucial in a rescaling process.

    Now, we can find and fix a small constant $\theta\in (0, 1/8)$ in \eqref{est.C1a} such that $C\theta^{1-d/p} < 1/2$. Thus,
    \begin{align}
    \begin{aligned}\label{est.u0-P}
        & \frac{1}{\theta r}\inf_{P\in \mathcal{P}}\bigg(\fint_{B_{\theta r}}|u_0-P|^2\bigg)^{1/2}+\theta r\bigg(\fint_{B_{\theta r}}|f|^p\bigg)^{1/p}\\
        &\leq \frac{1}{2r}\bigg(\fint_{B_{r}}|u_0-P_r|^2\bigg)^{1/2}+\frac{r}{2}\bigg(\fint_{B_{r}}|f|^p\bigg)^{1/p}+\frac{r}{2}\|\nabla_x\widehat{A}\|_{L^\infty(B_r)} h(r)\\
        &\leq \frac{1}{2r}\bigg(\fint_{B_{r}}|u_\varepsilon-u_0|^2\bigg)^{1/2}+\frac{1}{2}H(r)+\frac{r}{2}\|\nabla_x\widehat{A}\|_{L^\infty(B_r)} h(r).
    \end{aligned}
    \end{align}
    Note that $\theta$ depends only on $d, \mu, p$ and $[\delta]_0$. 
    
    Combining \eqref{int-Lip_es_approx} and \eqref{est.u0-P}, we obtain
    \begin{align}
        H(\theta r)&\leq \frac{1}{\theta r}\bigg(\fint_{B_{\theta r}}|u_\varepsilon-u_0|^2\bigg)^{\frac{1}{2}}+\frac{1}{\theta r}\inf_{P\in \mathcal{P}}\bigg(\fint_{B_{\theta r}}|u_0-P|^2\bigg)^{\frac{1}{2}}+\theta r\bigg(\fint_{B_{\theta r}}|f|^p\bigg)^{\frac{1}{p}}\nonumber\\
        & \leq \frac{(1+\theta^{-1-d/2})}{r}\bigg(\fint_{B_r}|u_\varepsilon-u_0|^2\bigg)^{1/2}+\frac{1}{2}H(r)+\frac{r}{2}\|\nabla_x\widehat{A}\|_{L^\infty(B_r)} h(r)\nonumber\\
        &\leq C_\theta \Lambda [\delta]_\alpha\bigg(\frac{\varepsilon_1}{r}+\sup_{k\geq 2}\frac{\e_k}{\e_{k-1}}\bigg)^{\alpha\sigma} \bigg\{\frac{1}{r}\bigg(\fint_{B_{2r}}|u_\varepsilon-b|^2\bigg)^{1/2}+r\bigg(\fint_{B_{2r}}|f|^2\bigg)^{1/2}\bigg\}\nonumber\\
        &\qquad + \frac{1}{2}H(r)+\frac{r}{2}\|\nabla_x\widehat{A}\|_{L^\infty(B_r)} h(r)\nonumber\\&\leq C \Lambda [\delta]_\alpha\bigg(\frac{\varepsilon_1}{r}+\sup_{k\geq 2}\frac{\e_k}{\e_{k-1}}\bigg)^{\alpha\sigma}\{H(2r)+h(2r)\}\nonumber\\
        &\qquad +\frac{1}{2}H(r)+\frac{r}{2}\|\nabla_x\widehat{A}\|_{L^\infty(B_r)} h(r),\label{est.H-tr.long}
    \end{align}
    where we have let $b=P_r-\nabla P_r \cdot x$ for the last step and $C$ depends only on $d, \mu, p$ and $[\delta]_0$. It is easy to verify that $H$ and $h$ satisfy that for $r\in(0, 1/2]$ (see \cite[Theorem 6.4.1]{Shen_book} or \cite[Lemma 6.4]{Niu20_reiterated}),
    \begin{align}\label{int-Lip_cond_Hh}
    \begin{split}
        \max_{r\leq t\leq 2r} H(t)\leq C H(2r), \\\max_{r\leq t, s\leq 2r} | h(t)-h(s)|\leq C H(2r),
    \end{split}
    \end{align}
    with $C$ depending on $d$, which yields
    \begin{align*}
        h(r)\leq h(2r)+CH(2r).
    \end{align*}
    Thus, it follows from \eqref{est.H-tr.long} that for $r\in[\varepsilon_1, 1/2]$,
    \begin{align}\label{est.H-thetar}
    \begin{split}
        H(\theta r)&\leq \frac{1}{2}H(r)\\&\quad+CM_0\bigg[\Big(\frac{\varepsilon_1}{r}\Big)^{\alpha\sigma}+r+\sup_{k\geq 2}\bigg(\frac{\varepsilon_k}{\varepsilon_{k-1}}\bigg)^{\alpha\sigma}\bigg]\{H(2r)+h(2r)\},
        \end{split}
    \end{align}
    where $M_0$ is given by \eqref{def.M0} and we have used the first inequality in \eqref{est.DAhat}. The constant $C$ in \eqref{est.H-thetar} depends only on $d, \mu, p, [\delta]_0$ and $[\delta]_1$. 
    In summary, we have verified that $H(r)$ and $h(r)$ satisfy the conditions in Lemma \ref{shen.lem}. Now, let $c_0\leq 1$ be given in Lemma \ref{shen.lem} and
    \begin{align*}
        \sup_{k\geq 2}\bigg(\frac{\varepsilon_k}{\varepsilon_{k-1}}\bigg)^{\alpha\sigma}\leq c_0|\log \varepsilon_1|^{-1}.
    \end{align*}
    Note that $c_0$ is independent of $\alpha$. Consequently, Lemma \ref{shen.lem} tells us that
    \begin{align*}
        \max_{\varepsilon_1\leq r\leq 1}\{H(r)+h(r)\}\leq C\{H(1)+h(1)\}.
    \end{align*}
    Substituting this back into \eqref{est.H-thetar}, we arrive at
    \begin{align}\label{int-Lip_es_induction}
    \begin{split}
        H(\theta r)&\leq \frac{1}{2}H(r)\\
        &\quad + CM_0\bigg[\Big(\frac{\varepsilon_1}{r}\Big)^{\alpha\sigma}+r+\sup_{k\geq 2}\bigg(\frac{\varepsilon_k}{\varepsilon_{k-1}}\bigg)^{\alpha\sigma}\bigg]\{H(1)+h(1)\}. 
    \end{split}
    \end{align}
    Taking $r = \theta^{n-1} \in [\e_1, 1]$, we apply \eqref{int-Lip_es_induction} repeatedly and obtain
    \begin{align*}
        \begin{aligned}
            H(\theta^n) & \le \frac{1}{2^{n-1}}H(\theta) \\
            & + CM_0\sum_{k=1}^{n-1}2^{k-n+1}\bigg[\Big(\frac{\varepsilon_1}{\theta^{k}}\Big)^{\alpha\sigma}+\theta^{k}+\sup_{k\geq 2}\bigg(\frac{\varepsilon_k}{\varepsilon_{k-1}}\bigg)^{\alpha\sigma}\bigg] \{H(1)+h(1)\} \\
            & \le \frac{C_0}{2^{n}}H(1)+CM_0\bigg[\Big(\frac{\varepsilon_1}{\theta^{n}}\Big)^{\alpha\sigma}+\frac{1}{2^{n}}+\sup_{k\geq 2}\bigg(\frac{\varepsilon_k}{\varepsilon_{k-1}}\bigg)^{\alpha\sigma}\bigg]\{H(1)+h(1)\},
        \end{aligned}
    \end{align*}
    where \eqref{int-Lip_cond_Hh} is used in the last inequality. This yields that for any $r\in[\varepsilon_1, \theta]$,
    \begin{align*}
        H(r)\leq C_0 r^\lambda H(1)+CM_0\bigg[\Big(\frac{\varepsilon_1}{r}\Big)^{\alpha \sigma}+r^\lambda+\sup_{k\geq 2}\bigg(\frac{\varepsilon_k}{\varepsilon_{k-1}}\bigg)^{\alpha \sigma}\bigg]\{H(1)+h(1)\},
    \end{align*}
    where $\lambda=\frac{\log 2}{-\log \theta} \in (0,1)$. The case $r\in [\theta, 1]$ is trivial since $H(r)\leq CH(1)$ due to \eqref{int-Lip_cond_Hh}. The proof is complete. 
\end{proof}

\begin{corollary}\label{int_coro_firststep}
    Under the same  assumptions as in  Theorem \ref{thm_Lip_1step},  for $r\in [\varepsilon_1, 1]$, we have
  \begin{align*}
    & H(r; u_\varepsilon)\leq C_0 r^\lambda H(1; u_\varepsilon)+CM_0\{H(1; u_\varepsilon)+h(1; u_\varepsilon)\},\\
    & |h(r; u_\varepsilon)-h(1; u_\varepsilon)|\leq C_0H(1; u_\varepsilon)+CM_0\{H(1; u_\varepsilon)+h(1; u_\varepsilon)\},
  \end{align*}
  where $M_0 $ and $\lambda$ are given in Theorem \ref{thm_Lip_1step}, $C_0$ depends only on $d, \mu, p $ and $[\delta]_0$, while $C$ depends additionally on $\alpha$ and $[\delta]_1$. 
\end{corollary}
\begin{proof}
    The first estimate follows directly from Theorem \ref{thm_Lip_1step}. For the second estimate, let $n\in\mathbb{N}$ satisfy $\theta^{n+1}< r\leq \theta^n$. Without loss of generality, we suppose that $n\geq 1$; otherwise the estimate of $h$ follows trivially from \eqref{int-Lip_cond_Hh}. Note that by \eqref{int-Lip_cond_Hh}
    \begin{align*}
        |h(r)-h(1)|\leq |h(r)-h(\theta^n)|+\sum_{j=0}^{n-1}|h(\theta^{j+1})-h(\theta^j)|\leq C\sum_{j=0}^n H(\theta^j). 
    \end{align*}
    Applying Theorem \ref{thm_Lip_1step}, we obtain
    \begin{align*}
        & |h(r)-h(1)|\\
        & \leq C_0\sum_{j=0}^{n}\theta^{j\lambda} H(1)+CM_0\sum_{j=0}^{n}\bigg[\Big(\frac{\varepsilon_{1}}{\theta^j}\Big)^{\alpha\sigma}+\theta^{j\lambda}+\sup_{k\geq 2}\bigg(\frac{\varepsilon_k}{\varepsilon_{k-1}}\bigg)^{\alpha\sigma}\bigg]\{H(1)+h(1)\}\\&\leq C_0H(1)+CM_0\{H(1)+h(1)\},
    \end{align*}
    where we have used the fact that
    \begin{align*}
        (n+1)\sup_{k\geq 2}\bigg(\frac{\varepsilon_k}{\varepsilon_{k-1}}\bigg)^{\alpha\sigma}&\leq C|\log r|\sup_{k\geq 2}\bigg(\frac{\varepsilon_k}{\varepsilon_{k-1}}\bigg)^{\alpha\sigma}\\&\leq C|\log \varepsilon_1|\sup_{k\geq 2}\bigg(\frac{\varepsilon_k}{\varepsilon_{k-1}}\bigg)^{\alpha\sigma}\leq Cc_0,
    \end{align*}
    due to the scale-separation condition \eqref{cond_logseparation}.  
    \end{proof}

We finish this subsection by proving the iteration lemma (a variant of \cite[Lemma 8.5]{S17}) that has been used in the proof of Theorem \ref{thm_Lip_1step}.

\begin{lemma}\label{shen.lem}
    Let $H(r)$ and $h(r)$ be two nonnegative continuous functions on the interval $(0,1]$ and let $\kappa \in (0,1/4)$. Assume that
\begin{align}\label{cond.Hh}
\max_{r\leq t\leq 2r} H(t)\leq C_0H(2r), ~~~~~\max_{r\leq t, s\leq 2r} | h(t)-h(s)|\leq C_0H(2r),
\end{align}
for any $r\in [\kappa, 1/2]$,  and  for some fixed $\theta \in (0,1/8)$,
\begin{align}\label{cond.Htheta}
H( \theta  r) \leq \frac{1}{2} H(r) + C_0 \big(\omega_1 (\kappa/r)+\omega_2(r)+\gamma\big)\{ H(2r)+h(2r)\},
\end{align}
for any $r\in [\kappa, 1/2]$, where 
$\omega_1, \omega_2$ are nonnegative increasing functions on $[0, 1]$, such that $\omega_i(0)=0$, and 
\begin{align*}
  \int_0^1 \frac{\omega_i(s)}{s} ds < \infty.
\end{align*}
Then there  exists a small constant $c_0$, depending only on $C_0$ decreasingly,  such that if $\gamma \leq c_0 |\log \kappa|^{-1}$, we have
\begin{align*}
\max_{\kappa\leq r\leq 1} \{H(r)+h(r)\}\leq C \{H(1) +h(1)\} ,
\end{align*}
where $C$ depends only on $C_0,  \theta$, and $\omega_1, \omega_2$.
\end{lemma}
\begin{remark}
  The constant $C$ is increasing in $C_0$ and decreasing in $\theta$. 
\end{remark}

\begin{proof}
The lemma can be proved by a combination of \cite[Lemma 6.2]{Niu25_Optimal} and \cite[Lemma 6.7]{Gu22_incompressible}.

By the second inequality in \eqref{cond.Hh},
\begin{align*}
  h(r)\leq h(2r)+C_0H(2r)
\end{align*}
for $\kappa\leq r\leq 1/2$, from which we deduce that
\begin{align*}
  \int_s^{1/2}\frac{h(r))}{r}dr\leq \int_{2s}^1\frac{h(r)}{r}dr+C_0\int_{2s}^1\frac{H(r)}{r}dr
\end{align*}
for $\kappa\leq s\leq 1/4$. This implies that
\begin{align*}
  \int_{s}^{2s}\frac{h(r)}{r}dr&\leq \int_{1/2}^1\frac{h(r)}{r}dr+C_0\int_{2s}^1\frac{H(r)}{r}dr\\&\leq C_0(\log 2)[h(1)+H(1)]+C_0\int_{2s}^1\frac{H(r)}{r}dr. 
\end{align*}
By this estimate and the second inequality in \eqref{cond.Hh}, it follows that for $\kappa\leq s\leq 1/4$
\begin{align}\label{est.hs}
    \begin{aligned}
    h(s)&\leq (\log 2)^{-1}\int_s^{2s}\frac{|h(r)-h(s)|+h(r)}{r}dr\\&\leq C_0(\log 2)^{-1}\Big\{H(2s)+h(1)+H(1)+\int_{2s}^1\frac{H(r)}{r}dr\Big\}. 
    \end{aligned}
\end{align}
Moreover, by the first inequality in \eqref{cond.Hh}, we have
\begin{align*}
  H(2s)&\leq C_0(\log 2)^{-1}\int_{2s}^1\frac{H(r)}{r}dr,\qquad \text{if }\kappa\leq s\leq 1/8,\\
  H(2s)&\leq C_0^2H(1), \qquad \text{if }1/8\leq s\leq 1/4,
\end{align*}
which, together with \eqref{est.hs} and \eqref{cond.Hh}, gives
\begin{align}\label{est.Hshs}
  H(s)+h(s)\leq C_1\bigg\{h(1)+H(1)+\int_s^1\frac{H(r)}{r}dr\bigg\}
\end{align}
for $\kappa\leq s\leq 1/4$, where $C_1$ depends only on $C_0$ increasingly. Combined with \eqref{cond.Htheta}, this implies that
\begin{align*}
H(\theta r)&\leq \frac{1}{2}H(r)+C_2(\omega_1(\kappa/r)+\omega_2(r)+\gamma)\{h(1)+H(1)\}\\&\qquad+C_2(\omega_1(\kappa/r)+\omega_2(r)+\gamma)\int_r^1\frac{H(t)}{t}dt,
\end{align*}
where $C_2$ depends only on $C_0$ increasingly. Dividing this inequality by $r$ and integrating over $r\in(a \kappa, b)$ for some $a>1>b>0$, we have
\begin{align}
  \int_{a \theta \kappa}^{b\theta}\frac{H(r)}{r}dr&\leq \frac{1}{2}\int_{a\kappa}^{b}\frac{H(r)}{r}dr+C_2\int_{a\kappa}^{b}\frac{\omega_1(\kappa/r)+\omega_2(r)+\gamma}{r}dr \{h(1)+H(1)\}\nonumber\\&\qquad+C_2\int_{a\kappa}^{b}\frac{\omega_1(\kappa/r)+\omega_2(r)+\gamma}{r}dr\int_r^1\frac{H(t)}{t}dt. \label{est.SH/r}
\end{align}
Using the integrability of $\omega_i$, we can find $a$ large enough and $b$ small enough, depending only on $C_2$ and $\omega_i$, such that
\begin{align*}
\int_{a\kappa}^{b}\frac{\omega_1(\kappa/r)+\omega_2(r)+\gamma}{r}dr&=\int_{\kappa/b}^{1/a}\frac{\omega_1(r)}{r}dr+\int_{a \kappa}^b\frac{\omega_2(r)}{r}dr+\int_{a \kappa}^b\frac{\gamma}{r}dr\\&\leq (8C_2)^{-1}+\gamma|\log\kappa|,
\end{align*}
in which process the choices of $a$ and $b$ are assumed to satisfy $a\kappa<b$, otherwise the desired result is trivial using \eqref{cond.Hh} as $b/a\leq \kappa\leq 1$. It follows that, 
\begin{align*}
\int_{a\kappa}^{b}\frac{\omega_1(\kappa/r)+\omega_2(r)+\gamma}{r}dr\int_r^1\frac{H(t)}{t}dt
\leq \{(8C_2)^{-1}+\gamma|\log\kappa|\}\int_{a\kappa}^1\frac{H(t)}{t}dt. 
\end{align*}
As a result, setting $c_0=(8C_2)^{-1}$ and using the assumption $\gamma<c_0|\log\kappa|^{-1}$, we deduce from \eqref{est.SH/r} that
\begin{align*}
  \int_{a \theta \kappa}^{b\theta}\frac{H(r)}{r}dr \leq \frac{1}{2}\int_{a\kappa}^{b}\frac{H(r)}{r}dr+\frac{1}{4}\{h(1)+H(1)\}+\frac{1}{4}\int_{a\kappa}^1\frac{H(t)}{t}dt,
\end{align*}
and thus
\begin{align*}
    \int_{a \theta \kappa}^{b\theta}\frac{H(r)}{r}dr \leq 3\int_{b\theta}^{1}\frac{H(r)}{r}dr+\{h(1)+H(1)\},
\end{align*}
which yields
\begin{align*}
  \int_{a \theta \kappa}^{1}\frac{H(r)}{r}dr & = \int_{a \theta \kappa}^{b\theta}\frac{H(r)}{r}dr+\int_{b\theta}^{1}\frac{H(r)}{r}dr\\
  & \leq 4\int_{b\theta}^{1}\frac{H(r)}{r}dr+\{h(1)+H(1)\}\\
  & \leq C\{h(1)+H(1)\}. 
\end{align*}
Substituting this integral into \eqref{est.Hshs}, we obtain the desired estimate for $a\theta \kappa\leq r\leq 1/4$. The other cases can be derived easily by using \eqref{cond.Hh}. 
\end{proof}

\subsection{Iteration across all scales}\label{sec_iteration}
By the analyses in Section \ref{sec_rescaling}, the results in Section \ref{sec_large-scale_u} are valid for rescaled equations with universal constants $c_0, \lambda, C_0, C$ under condition \eqref{int-Lip_cond_sumbound}. Precisely, under \eqref{int-Lip_cond_sumbound} and the scale-separation condition 
\begin{align}\label{cond_logseparation_m}
  \bigg(\frac{\varepsilon_k}{\varepsilon_{k-1}}\bigg)^{\alpha\sigma}\leq c_0\Big|\log \frac{\varepsilon_{m+1}}{\varepsilon_m}\Big|^{-1} \quad\text{for any } k\geq m+2,
\end{align}
where $\alpha\in(0, 1]$ and $\sigma>0$ is a constant depending only on $d$ and $\mu$, the rescaled solution $v_m(x) = u_\e(\varepsilon_m x)$ defined in Section \ref{sec_rescaling} satisfies, as in Corollary \ref{int_coro_firststep}, that for $r\in [\varepsilon_{m+1}/\varepsilon_m, 1]$,
\begin{align*}
  & H(r; v_m)\leq C_0 r^\lambda H(1; v_m)+CM_m\{H(1; v_\varepsilon)+h(1; v_m)\},\\
  & |h(r; v_m)-h(1; v_m)|\leq C_0H(1; v_\varepsilon)+CM_m\{H(1; v_\varepsilon)+h(1; v_m)\},
\end{align*}
where, due to the definition of $M_0$ in \eqref{def.M0} (applied to the rescaled matrices $\widetilde{B}_\ell$ in Section \ref{sec_rescaling}), 
\begin{align}\label{int-Lip_def_MN}
\begin{aligned}
  M_m: & = \max \Big\{ [\widetilde{\delta}]_\alpha, \sum_{k=0}^\infty \| \nabla_{x} \widetilde{B}_\ell \|_{L^\infty} \Big\} \\
  & \le \max\bigg\{\sum_{j=0}^{m}\frac{\varepsilon_m}{\varepsilon_j}\sum_{k=j}^{\infty}\delta_k, \sum_{j=0}^m\frac{\varepsilon_m}{\varepsilon_j}\cdot\sum_{k=1}^\infty k^\alpha\delta_{k+m}\bigg\}.
  \end{aligned}
\end{align}
Observe that $H(r; v_m)= \varepsilon_m H(r\varepsilon_m; u_\varepsilon)$ and $h(r; v_m)=\varepsilon_m h(r \varepsilon_m; u_\varepsilon)$. Rescaling it back, we obtain the following lemma.

\begin{lemma}\label{int-Lip_prop_Hh}
    Suppose that $[\delta]_1<\infty$, $0<\alpha\leq 1$ and \eqref{int-Lip_cond_sumbound} holds. There exists a constant $c_0\leq 1$, depending on $d, \mu, p, [\delta]_0 $ and $[\delta]_1$, such that if the scale-separation condition \eqref{cond_logseparation_m} holds for $m\geq 0$, then for all $r\in [\varepsilon_{m+1}, \varepsilon_m]$
    \begin{align*}
    \begin{split}
        & H(r; u_\varepsilon)\leq C_0 \Big(\frac{r}{\varepsilon_m}\Big)^\lambda H(\varepsilon_m; u_\varepsilon)+C_1M_m\{H(\varepsilon_m; u_\varepsilon)+h(\varepsilon_m; u_\varepsilon)\},\\
        & |h(r; u_\varepsilon)-h(\varepsilon_m; u_\varepsilon)|\leq C_0H(\varepsilon_m; u_\varepsilon)+C_1M_m\{H(\varepsilon_m; u_\varepsilon)+h(\varepsilon_m; u_\varepsilon)\},
    \end{split}
    \end{align*}
    where $M_m$ is given by \eqref{int-Lip_def_MN}, $\lambda\in(0, 1)$ and $C_1\geq C_0\geq 1$. Furthermore, $\lambda$ and $C_0$ depend only on $d, \mu, p $ and $[\delta]_0$, while $C_1$ depends additionally on $\alpha$ and $[\delta]_1$. 
\end{lemma}

The next lemma shows that $M_m$ is summable --- a key property that allows us to control the errors across all scales.

\begin{lemma}\label{lem.sumMm}
    Assume $[\delta]_{1+\rho} < \infty$ for some $\rho>0$ and \eqref{int-Lip_cond_sumbound}. Let $\alpha = \rho/2$ in \eqref{int-Lip_def_MN}. Then for any $T > 0$, there exists $m_0 \in \N$, depending only on $\rho, T, [\delta]_0$ and $[\delta]_{1+\rho}$, such that
    \begin{equation*}
        T \sum_{\ell = m_0}^\infty M_\ell \le 1.
    \end{equation*}
\end{lemma}

\begin{proof}
    Clearly, by \eqref{int-Lip_def_MN}, we have
    \begin{equation}\label{eq.SMn}
        \sum_{\ell = m}^\infty M_\ell \le \sum_{\ell=m}^\infty \sum_{j=0}^{\ell}\frac{\varepsilon_\ell} {\varepsilon_j}\sum_{k=j}^\infty\delta_k + \sum_{\ell=m}^\infty\sum_{j=0}^\ell\frac{\varepsilon_\ell}{\varepsilon_j}\sum_{k=1}^\infty k^\alpha\delta_{k+\ell}.
    \end{equation}
    We estimate the two terms on the right-hand side separately.
    
    By condition \eqref{int-Lip_cond_sumbound}, we have
    \begin{equation*}
        \sum_{j=0}^k \frac{1}{\e_j} \le \frac{3}{\e_k},
    \end{equation*}
    where we have used the fact $K \le 2$.
    It follows that 
    \begin{equation}\label{est.SMm-1}
    \begin{aligned}
    \sum_{\ell=m}^\infty \sum_{j=0}^{\ell}\frac{\varepsilon_\ell} {\varepsilon_j}\sum_{k=j}^\infty\delta_k & = \sum_{\ell=m}^\infty\sum_{k=0}^{\ell}\sum_{j=0}^{k}\frac{\varepsilon_\ell}{\varepsilon_j}\delta_k+\sum_{\ell=m}^\infty\sum_{k=\ell+1}^{\infty}\sum_{j=0}^{\ell}\frac{\varepsilon_\ell}{\varepsilon_j}\delta_k\\
    & \leq 3\sum_{\ell=m}^\infty\sum_{k=0}^{\ell}\frac{\varepsilon_\ell}{\varepsilon_{k}}\delta_k+3\sum_{k=m+1}^\infty (k-1)\delta_k.\\
        \end{aligned}
    \end{equation}
    By Lemma \ref{int-Lip_lem_varepsilon} and Remark \ref{int-lip_remark_decay}, we calculate that for $m\in \mathbb{N}$,
    \begin{equation*}
        \begin{aligned}
        &\sum_{\ell=m}^\infty\sum_{k=0}^{\ell} \frac{\varepsilon_\ell}{\varepsilon_{k}}\delta_k = \sum_{\ell=m}^\infty\varepsilon_\ell\delta_0+\sum_{k=1}^{m-1}\sum_{\ell=m}^\infty\frac{\varepsilon_\ell}{\varepsilon_{k}}\delta_k+\sum_{k=m}^{\infty}\sum_{\ell=k}^\infty\frac{\varepsilon_\ell}{\varepsilon_{k}}\delta_k \\
        &\leq 9\bigg\{[\delta]_0(1-K^{-1})^{m-1}+\sum_{k=1}^{m-1}(1-K^{-1})^{m-k}\delta_k+\sum_{k=m}^{\infty}\delta_k\bigg\} \\&\leq 9 \bigg\{[\delta]_0(1-K^{-1})^{m-1}+\max\{m^{-1}, (1-K^{-1})^{m-1}\}\sum_{k=1}^{m-1}k\delta_k\\&\qquad\qquad+m^{-1}\sum_{k=m}^{\infty}k\delta_k\bigg\} \\
        &\leq 9 ([\delta]_0+[\delta]_1)\max\{m^{-1}, (1-K^{-1})^{m-1}\}.
        \end{aligned}
    \end{equation*}
    On the other hand, by the assumption $[\delta]_{1+\rho} < \infty$, we have
    \begin{equation*}
        \sum_{k=m+1}^\infty (k-1)\delta_k \le \frac{1}{m^\rho}[\delta]_{1+\rho}.
    \end{equation*}
    Bringing the last two displayed inequalities back into \eqref{est.SMm-1}, we obtain
    \begin{equation*}
        \sum_{\ell=m}^\infty \sum_{j=0}^{\ell}\frac{\varepsilon_\ell} {\varepsilon_j}\sum_{k=j}^\infty\delta_k \le 27([\delta]_0+[\delta]_1)\max\{m^{-1}, (1-K^{-1})^{m-1}\}+\frac{3}{m^\rho}[\delta]_{1+\rho}.
    \end{equation*}

    Next, we turn to the second term on the right-hand side of \eqref{eq.SMn}. By using $[\delta]_{1+\rho} < \infty$ and $\alpha = \rho/2$, we obtain
    \begin{align*}
        &\sum_{\ell=m}^\infty \sum_{j=0}^\ell \frac{\varepsilon_\ell}{\varepsilon_j}\sum_{k=1}^\infty k^\alpha\delta_{k+\ell}\leq 3\sum_{\ell=m}^\infty\sum_{k=1}^\infty k^\alpha\delta_{k+\ell}\\
        & \le 3\sum_{\ell=m}^\infty\sum_{k=\ell+1}^\infty k^\alpha\delta_{k}= 3\sum_{k=m+1}^\infty\sum_{\ell=m}^{k-1} k^\alpha\delta_{k} \\
        &\leq 3\sum_{k=m+1}^\infty k^{1+\alpha}\delta_k\leq \frac{3}{m^{\rho-\alpha}}\sum_{k=m+1}^\infty k^{1+\rho}\delta_k\leq \frac{3}{m^{\rho/2}}[\delta]_{1+\rho}.
    \end{align*}
    As a result, we get
    \begin{equation*}
        \sum_{\ell = m}^\infty M_\ell \le 27 ([\delta]_0+[\delta]_1)\max\{m^{-1}, (1-K^{-1})^{m-1}\}+\frac{6}{m^{\rho/2}}[\delta]_{1+\rho}.
    \end{equation*}
    Obviously, the right-hand side of the above inequality converges to zero as $m\to \infty$ and there exists $m_0$, depending on $\rho, T, [\delta]_0$ and $[\delta]_{1+\rho}$ such that $\sum_{\ell = m_0}^\infty M_\ell \le T^{-1}$. This ends the proof.
\end{proof}

Now we are ready to prove our main theorems in this section.

\begin{theorem}\label{thm_Hh_bound}
  Let $c_0$ be given by Lemma \ref{int-Lip_prop_Hh}. Suppose $[\delta]_{1+\rho}<\infty$ for some $\rho>0$ and the conditions \eqref{int-Lip_cond_sumbound} and \eqref{cond_logseparation_m} hold with $\alpha=\rho/2$ for all $m \in \N$. There exists $K_0\in(1, 2]$, depending only on $d, \mu, p$ and $[\delta]_0$, such that if $1<K\leq K_0$, then for any $r\in(0, 1]$,
  \begin{align*}
    H(r; u_\varepsilon)+h(r; u_\varepsilon)\leq C\{H(1; u_\varepsilon)+h(1; u_\varepsilon)\},
  \end{align*}
  where $C$ depends on $d, \mu, p, [\delta]_0, \rho$ and $[\delta]_{1+\rho}$.
\end{theorem}


\begin{proof}
Let $C_1$ be given by Lemma \ref{int-Lip_prop_Hh}. By Lemma \ref{lem.sumMm}, there exists $m_0 \in \N$ such that
\begin{align}\label{int-Lip_es_M_remainder}
    64C_1^2\sum_{\ell=m_0}^\infty M_\ell\leq 1,
\end{align}
and therefore,
\begin{align}\label{int-Lip_property_Mj}
    64C_1^2 M_\ell\leq 1\quad \text{for all }\ell \geq m_0.
\end{align}
We claim that for $\ell\geq m_0 \geq 0$
  \begin{align}\label{int-Lip_H_claim}
    \begin{split}
      H(\varepsilon_{\ell+1}) & \leq C_0^{\ell-m_0+1}\Big(\frac{\varepsilon_{\ell+1}}{\varepsilon_{m_0}}\Big)^\lambda H(\varepsilon_{m_0})\\
      & \qquad+\sum_{j=m_0}^{\ell}C_0^{\ell-j}\Big(\frac{\varepsilon_{\ell+1}}{\varepsilon_{j+1}}\Big)^\lambda C_1M_j\{H(\varepsilon_j)+h(\varepsilon_j)\}.
    \end{split}
  \end{align}
  This can be proved by induction on $\ell$. The case $\ell=m_0$ is given by Lemma \ref{int-Lip_prop_Hh}. Suppose that it is true for $\ell+1$. Then, by Lemma \ref{int-Lip_prop_Hh} and using the inductive hypothesis, we have
  \begin{align*}
    H(\varepsilon_{\ell+2})&\leq C_0\Big(\frac{\varepsilon_{\ell+2}}{\varepsilon_{\ell+1}}\Big)^\lambda H(\varepsilon_{\ell+1})+C_1M_{\ell+1}\{H(\varepsilon_{\ell+1})+h(\varepsilon_{\ell+1})\}\\
    &\leq C_0\Big(\frac{\varepsilon_{\ell+2}}{\varepsilon_{\ell+1}}\Big)^\lambda \bigg\{C_0^{\ell-m_0+1}\Big(\frac{\varepsilon_{\ell+1}}{\varepsilon_{m_0}}\Big)^\lambda H(\varepsilon_{m_0})\\
    &\qquad \qquad+\sum_{j=m_0}^\ell C_0^{\ell-j}\Big(\frac{\varepsilon_{\ell+1}}{\varepsilon_{j+1}}\Big)^\lambda C_1M_j\big\{H(\varepsilon_j)+h(\varepsilon_j)\big\}\bigg\}\\
    &\qquad +C_1M_{\ell+1}\{H(\varepsilon_{\ell+1})+h(\varepsilon_{\ell+1})\}\\&\leq C_0^{\ell-m_0+2}\Big(\frac{\varepsilon_{\ell+2}}{\varepsilon_{m_0}}\Big)^\lambda H(\varepsilon_{m_0})\\&\qquad+\sum_{j=m_0}^{\ell+1} C_0^{\ell+1-j}\Big(\frac{\varepsilon_{\ell+2}}{\varepsilon_{j+1}}\Big)^\lambda C_1M_j\{H(\varepsilon_j)+h(\varepsilon_j)\},
  \end{align*}
  which is exactly the claim.

Taking the summation of \eqref{int-Lip_H_claim} from $m_0$ to $k > m_0$, we obtain by Lemma \ref{int-Lip_lem_varepsilon} that
  \begin{align*}
    \sum_{\ell=m_0}^k H(\varepsilon_{\ell+1}) &\leq \sum_{\ell=m_0}^kC_0^{\ell-m_0+1}\Big(\frac{\varepsilon_{\ell+1}}{\varepsilon_{m_0}}\Big)^\lambda H(\varepsilon_{m_0})\\&\qquad+ \sum_{\ell=m_0}^k\sum_{j=m_0}^{\ell}C_0^{\ell-j}\Big(\frac{\varepsilon_{\ell+1}}{\varepsilon_{j+1}}\Big)^\lambda C_1M_j\{H(\varepsilon_j)+h(\varepsilon_j)\}\\&\leq 2 C_0(1-K^{-1})^{\lambda} H(\varepsilon_{m_0})\sum_{\ell={m_0}}^k[C_0(1-K^{-1})^{\lambda}]^{\ell-m_0}\\&\quad+2 \sum_{j=m_0}^kC_1M_j\{H(\varepsilon_j)+h(\varepsilon_j)\}\sum_{\ell=j}^k[C_0(1-K^{-1})^{\lambda}]^{\ell-j}\\&\leq 2H(\varepsilon_{m_0})+\sum_{\ell=m_0}^k 4C_1M_\ell\{H(\varepsilon_\ell)+h(\varepsilon_\ell)\},
  \end{align*}
  where in the last step we have required $1 < K \leq K_0$ and $K_0 \in (1,2]$ is close to $1$ such that
\begin{align}\label{int-Lip_cond_K0}
  C_0(1-K_0^{-1})^{\lambda}\le 1/2.
\end{align}

Therefore, by \eqref{int-Lip_property_Mj},
\begin{align*}
    \sum_{\ell=m_0}^{k+1}H(\varepsilon_{\ell}) & = H(\varepsilon_{m_0}) + \sum_{\ell=m_0}^{k}H(\varepsilon_{\ell+1})\\
    & \leq 3H(\varepsilon_{m_0})+\frac{1}{2} \sum_{\ell = {m_0}}^k H(\varepsilon_\ell) + \sum_{\ell = m_0}^k 4C_1M_\ell h(\varepsilon_\ell),
\end{align*}
which leads to
\begin{align}\label{int-Lip_es_H}
  \sum_{\ell = m_0}^{k+1}H(\varepsilon_{\ell})\leq 6H(\varepsilon_{m_0})+\sum_{\ell = m_0}^k8C_1M_\ell h(\varepsilon_\ell).
\end{align}

On the other hand, using the estimate of $h$ in Lemma \ref{int-Lip_prop_Hh}, we have for $k\geq m_0$,
\begin{align*}
  |h(\varepsilon_{k+1})-h(\varepsilon_{m_0})| & \leq \sum_{\ell=m_0}^k|h(\varepsilon_{\ell+1})-h(\varepsilon_\ell)|\\
  &\leq C_0\sum_{\ell=m_0}^k H(\varepsilon_\ell)+\sum_{\ell=m_0}^kC_1M_\ell\{H(\varepsilon_\ell)+h(\varepsilon_\ell)\}\\
  &\leq 2C_1\sum_{\ell=m_0}^k H(\varepsilon_\ell)+\sum_{\ell=m_0}^kC_1M_\ell h(\varepsilon_\ell),
\end{align*}
where we have used the fact $M_\ell\leq 1$ for $\ell\geq m_0$ derived from \eqref{int-Lip_property_Mj}. Substituting \eqref{int-Lip_es_H} into the inequality above and using \eqref{int-Lip_es_M_remainder}, this gives
\begin{align*}
  h(\varepsilon_{k+1})&\leq h(\varepsilon_{m_0})+ 12C_1H(\varepsilon_{m_0})+\sum_{\ell=m_0}^{k}16C_1^2M_\ell h(\varepsilon_\ell)+\sum_{\ell=m_0}^kC_1M_\ell h(\varepsilon_\ell)\\
  & \leq h(\varepsilon_{m_0})+ 12C_1H(\varepsilon_{m_0})+\frac{1}{2}\sup_{m_0 \leq \ell\leq k}h(\varepsilon_\ell).
\end{align*}
Taking the supremum over $k$, we obtain
\begin{align*}
  \sup_{m_0 \leq \ell\leq k}h(\varepsilon_{\ell})\leq h(\varepsilon_{m_0})+ 12C_1H(\varepsilon_{m_0})+\frac{1}{2}\sup_{m_0 \leq \ell\leq k}h(\varepsilon_\ell),
\end{align*}
which implies, 
\begin{align*}
  \sup_{\ell \ge m_0} h(\varepsilon_{\ell})&\leq 2\{h(\varepsilon_{m_0})+ 12C_1H(\varepsilon_{m_0})\}. 
\end{align*}
Combining this estimate with \eqref{int-Lip_es_H} and \eqref{int-Lip_es_M_remainder}, and taking $k\to \infty$, we also have
\begin{align*}
  \sum_{\ell=m_0}^{\infty} H(\varepsilon_\ell)\leq 6H(\varepsilon_{m_0})+\sum_{\ell=m_0}^{\infty}8C_1M_\ell h(\varepsilon_\ell)\leq h(\varepsilon_{m_0})+8C_1H(\varepsilon_{m_0}). 
\end{align*}
Consequently, we conclude that
\begin{align}\label{est.m0-infity}
  \sum_{\ell = m_0}^\infty H(\varepsilon_\ell) + \sup_{\ell \ge m_0} h(\varepsilon_\ell)\leq C\{h(\varepsilon_{m_0})+H(\varepsilon_{m_0})\},
\end{align}
where $C$ depends only on $C_1$.

Furthermore, using the estimates in Lemma \ref{int-Lip_prop_Hh} recursively, it holds for $1\leq \ell\leq m_0$
  \begin{align}\label{est.1-m_0}
    H(\varepsilon_{\ell})+h(\varepsilon_{\ell})\leq C\{H(\varepsilon_{\ell-1})+h(\varepsilon_{\ell-1})\}\leq\cdots\leq  C^{\ell}\{H(1)+h(1)\}, 
  \end{align}
  where $C$ depends only on $d, \mu, p, \rho, [\delta]_0$ and $[\delta]_1$. According to Lemma \ref{int-Lip_lem_varepsilon}, we know $\varepsilon_k\rightarrow 0$ as $k\rightarrow\infty$. This means that for any $r\in (0, 1]$ there exists $k\in\mathbb{N}$ such that $\varepsilon_{k+1}<r\leq \varepsilon_k$. Using Lemma \ref{int-Lip_prop_Hh} again, we have
  \begin{align}\label{est.in2scale}
    H(r)+h(r)\leq C\{H(\varepsilon_k)+h(\varepsilon_k)\}. 
  \end{align}
Combining \eqref{est.m0-infity}, \eqref{est.1-m_0} and \eqref{est.in2scale}, we finally obtain
  \begin{align*}
    H(r)+h(r)\leq C\{H(1)+h(1)\}\quad\text{for all } r\in (0,1), 
  \end{align*}
  where $C$ also depends on $m_0$. This completes the proof.
\end{proof}

\begin{remark}
    The estimate \eqref{est.m0-infity} actually implies a stronger result: 
    $$\lim_{r\to 0} H(r) \to 0, $$
    which reflects the flatness of the solution $u_\e$, beyond the boundedness of $\nabla u_\e$.  
\end{remark}

\begin{theorem}\label{int-lip_lem_Lip}
  Let $c_0, K_0$ be given in Theorem \ref{thm_Hh_bound}. Suppose $[\delta]_{1+\rho}<\infty$ for some $\rho>0$ and the conditions \eqref{int-Lip_cond_sumbound} and \eqref{cond_logseparation_m} hold with $\alpha=\rho/2$ for each $m\geq 0$. Let $u_\varepsilon$ be a solution to $-\mathrm{div}(A^\varepsilon(x)\nabla u_\varepsilon)=f$ in $B_{R}$, $0<R\leq 1$ and $f\in L^p(B_{R})$ with $p>d$. Then if $K\leq K_0$, 
  \begin{align*}
    \| \nabla u_\e \|_{L^\infty(B_{R/2})} \leq C\bigg\{\bigg(\fint_{B_{R}}|\nabla u_\varepsilon|^2\bigg)^{1/2}+R\bigg(\fint_{B_{R}}|f|^p\bigg)^{1/p}\bigg\},
  \end{align*}
  where $C$ depends only on $d, \mu, p, [\delta]_0, \rho$ and $[\delta]_{1+\rho}$. 
\end{theorem}
\begin{proof}
  By translation and dilation, we only need to prove the bound of $|\nabla u_\e(0)|$ with $R=1$. In fact, since $\varepsilon_k \rightarrow 0$ as $k\rightarrow \infty$, for any $R\in (0,1)$, there exists $m$ such that $\varepsilon_{m+1}<R\leq \varepsilon_m$. By setting $v(x)=u_\varepsilon(Rx)$, we obtain
  \begin{align*}
    -\mathrm{div}\Big( A \Big( \frac{Rx}{\e_0} , \cdots, \frac{Rx}{\varepsilon_m}, \frac{x}{\varepsilon_{m+1}/R}, \cdots \Big) \nabla v \Big) = R^2f(Rx)\quad \text{in }B_{1}. 
  \end{align*}
  By a similar argument as Section \ref{sec_rescaling}, it is not hard to verify that all the conditions keep fulfilled for this rescaled equation.   
  
  Now assume $R = 1$. For any $0<r\leq 1$, by the Caccioppoli inequality, we have
  \begin{align}\label{est.Du-Hh}
  \begin{aligned}
    \bigg(\fint_{B_{r/2}}|\nabla u_\varepsilon|^2\bigg)^{1/2} & \leq C\inf_{q\in \mathbb{R}}\frac{1}{r}\bigg(\fint_{B_r}|u_\varepsilon-q|^2\bigg)^{1/2} + Cr\bigg(\fint_{B_r}|f|^p\bigg)^{1/p} \\
    & \leq C\{H(r; u_\varepsilon)+h(r; u_\varepsilon)\},
    \end{aligned}
  \end{align}
  where $C$ depends only on $d$ and $\mu$. On the other hand, observing that
  \begin{align*}
      h(r; u_\varepsilon)&= |\nabla P_r| = \frac{C}{r}\inf_{q\in\mathbb{R}}\bigg(\fint_{B_r}|P_r-q|^2\bigg)^{1/2}\\&\leq \frac{C}{r}\bigg(\fint_{B_r}|u_\varepsilon-P_r|^2\bigg)^{1/2}+\frac{C}{r}\inf_{q\in\mathbb{R}}\bigg(\fint_{B_r}|u_\varepsilon-q|^2\bigg)^{1/2}\\& \leq \frac{C}{r}\inf_{q\in\mathbb{R}}\bigg(\fint_{B_r}|u_\varepsilon-q|^2\bigg)^{1/2},
  \end{align*}
we obtain from the Poincar\'{e} inequality that
  \begin{align}\label{est.Hh-Du}
  \begin{aligned}
    H(r; u_\varepsilon)+h(r; u_\varepsilon)&\leq \frac{C}{r}\inf_{q\in\mathbb{R}}\bigg(\fint_{B_r}|u_\varepsilon-q|^2\bigg)^{1/2}+Cr\bigg(\fint_{B_r}|f|^p\bigg)^{1/p}\\&\leq C\bigg(\fint_{B_{r}}|\nabla u_\varepsilon|^2\bigg)^{1/2}+Cr\bigg(\fint_{B_r}|f|^p\bigg)^{1/p}.
    \end{aligned}
  \end{align}
  Combining \eqref{est.Du-Hh} and \eqref{est.Hh-Du}, as well as Theorem \ref{thm_Hh_bound}, we know
  \begin{align*}
    \bigg(\fint_{B_{r/2}}|\nabla u_\varepsilon|^2\bigg)^{1/2}&\leq C\{H(r; u_\varepsilon)+h(r; u_\varepsilon)\}\\&\leq C\{H(1; u_\varepsilon)+h(1; u_\varepsilon)\}\\&\leq C\bigg\{\bigg(\fint_{B_1}|\nabla u_\varepsilon|^2\bigg)^{1/2}+\bigg(\fint_{B_{1}}|f|^p\bigg)^{1/p}\bigg\}.
  \end{align*}
  Since this estimate holds for any $r\in (0,1)$, by
  letting $r\rightarrow0$, we obtain
  \begin{align*}
    |\nabla u_\varepsilon(0)|\leq C\bigg\{\bigg(\fint_{B_1}|\nabla u_\varepsilon|^2\bigg)^{1/2}+\bigg(\fint_{B_1}|f|^p\bigg)^{1/p}\bigg\}.
  \end{align*}
  This is the desired estimate.
\end{proof}


\begin{proof}[Proof of Theorem \ref{thm.IntLip}]
  By Theorem \ref{int-lip_lem_Lip}, it suffices to verify that \eqref{int-Lip_cond_sumbound} holds with $K\leq K_0$ and \eqref{cond_logseparation_m} holds with $\alpha=\rho/2$. Let $\nu>0$ be small enough such that $(1-\nu^{1/N})^{-1}\leq K_0$ and $|\log \nu|\nu^{\frac{\alpha\sigma}{N}}\leq c_0$. By condition \eqref{int-lip_cond_separation_ratio}, we have
  \begin{align*}
    \Big(\frac{\varepsilon_{k+1}}{\varepsilon_{k}}\Big)^N\leq \varepsilon_1\leq \nu\quad \text{for }k\geq 1,
  \end{align*}
  from which we deduce that
  \begin{align*}
    \sum_{j=1}^k\frac{\varepsilon_k}{\varepsilon_j}=\sum_{j=1}^k\prod_{\ell=j}^{k-1}\frac{\varepsilon_{\ell+1}}{\varepsilon_{\ell}}\leq \sum_{j=1}^k\nu^{\frac{k-j}{N}}\leq (1-\nu^{1/N})^{-1}\leq K_0.
  \end{align*}
On the other hand, for any $m\geq 0$ and $k\geq m+2$
  \begin{align*}
    \bigg(\frac{\varepsilon_k}{\varepsilon_{k-1}}\bigg)^{\alpha\sigma}\leq \Big(\frac{\varepsilon_{m+1}}{\varepsilon_m}\Big)^{\frac{\alpha\sigma}{N}}\leq c_0\Big|\log\frac{\varepsilon_{m+1}}{\varepsilon_m}\Big|^{-1},
  \end{align*}
where we have used \eqref{int-lip_cond_separation_ratio} in the first step and the fact
\begin{align*}
  \Big|\log\frac{\varepsilon_{m+1}}{\varepsilon_m}\Big|\Big(\frac{\varepsilon_{m+1}}{\varepsilon_m}\Big)^{\frac{\alpha\sigma}{N}}\leq |\log \nu|\nu^{\frac{\alpha\sigma}{N}}\leq c_0
\end{align*}
in the last one. In summary, all the conditions in Theorem \ref{int-lip_lem_Lip} are fulfilled and we obtain \eqref{int-lip_es_Lip}.
\end{proof}

\begin{remark}\label{rmk.illExample}
    As pointed out in Remark \ref{rmk.smallnu}, the requirement for $\e_1 < \nu$ is due to a technical reason in our method. Here we construct an example that will cause a problem. Let $\e = (\e_1, \e_2,\cdots)$ be given by
    \begin{equation*}
        \e_j = 
        \begin{cases}
            e^{-j}, \quad & \text{if } 1\le j\le m,\\
            e^{-m} \tau^{j-m}, \quad & \text{if } j > m.
        \end{cases}
    \end{equation*}
    Assume $0<\tau \ll e^{-1}$. It is easy to see that such $\e$ satisfies the scale-separation condition \eqref{int-lip_cond_separation_ratio}. However, the first $m$ scales are not separated well and homogenization does not take place above $\e_m$-scale. Our method is not able to show a Lipschitz estimate uniformly in both $m$ and $\tau$. This situation needs a more careful analysis, which will not be discussed in the present paper.
\end{remark}

The smallness assumption of $\e_1$ in the uniform Lipschitz estimate can be removed for the most interesting cases.

\begin{lemma}\label{int-lip_thm_Dini}
  Suppose that $[\delta]_{1+\rho}<\infty$ for some $\rho>0$ and there exists a constant $\vartheta\in(0, 1)$ such that
  \begin{align}\label{cond.ej>cj}
    \varepsilon_j\geq \vartheta^j\quad \text{for each }j. 
  \end{align}
  Let $u_\varepsilon$ be the same as in Theorem \ref{thm.IntLip}. Then for any $0<R\le 1$,
  \begin{align*}
    \| \nabla u_\e \|_{L^\infty(B_{R/2})} \leq C\bigg\{\bigg(\fint_{B_{R}}|\nabla u_\varepsilon|^2\bigg)^{1/2}+R\bigg(\fint_{B_{R}}|f|^p\bigg)^{1/p}\bigg\},
  \end{align*}
  where $C$ depends only on $d, \mu, p, \vartheta, \rho, [\delta]_0 $ and $[\delta]_{1+\rho}$.
\end{lemma}

\begin{proof}
By dilation, it suffices to consider $R=1$. Denote the continuity modulus of $A^\varepsilon$ by $$\omega(r):=\sup_{|x_1-x_2|<r}|A^\varepsilon(x_1)-A^\varepsilon(x_2)|.$$
For $0<r\leq \frac{1}{2}$, set $n= n(r) = \lfloor \varrho |\log r|\rfloor$ where $\varrho=\frac{1}{2|\log \vartheta|}$. We calculate that
  \begin{align*}
    \omega(r)&\leq \|\nabla A_n^\varepsilon(x)\|_{L^\infty}\cdot r+2\sum_{\ell=n+1}^\infty \| B_\ell \|_{L^\infty} \leq \sum_{\ell=0}^n\sum_{j=0}^\ell \varepsilon_j^{-1}\delta_\ell r+2\sum_{\ell=n+1}^\infty \delta_\ell\\&\leq \sum_{\ell=0}^n(\ell+1)\delta_\ell\varepsilon_n^{-1} r+2\sum_{\ell=n+1}^\infty \delta_\ell\leq ([\delta]_0+[\delta]_{1})\varepsilon_n^{-1}r+2[\delta]_{1+\rho}(n+1)^{-1-\rho}\\&\leq C\vartheta^{-\lfloor \varrho |\log r|\rfloor}r+C(\lfloor \varrho |\log r|\rfloor+1)^{-1-\rho}\leq Cr^{1/2}+C|\log r|^{-1-\rho},
  \end{align*}
  where we have used the assumption \eqref{cond.ej>cj} and $C$ depends only on $[\delta]_0, [\delta]_{1+\rho}$ and $\vartheta$. Therefore,
  \begin{align*}
    \int_0^{1/2} \frac{\omega(r)}{r}dr<\infty,
  \end{align*}
  i.e., $A^\varepsilon(x)$ is Dini continuous. By the Lipschitz estimate of elliptic equations with Dini continuous coefficients (see \cite{Lieb86}; also see \cite{Li2017} for elliptic systems), we have
  \begin{align*}
    \| \nabla u_\e\|_{L^\infty(B_{1/2})} \leq C\bigg\{\bigg(\fint_{B_1}|\nabla u_\varepsilon|^2\bigg)^{1/2}+\bigg(\fint_{B_1}|f|^p\bigg)^{1/p}\bigg\},
  \end{align*}
  where $C$ depends only on $d, \mu, p $ and $\omega$.
\end{proof}

\begin{theorem}\label{thm.typicalLipEst}
    Let $\{ A_n\}_{n\in \N}$, $\delta$ and $u_\e$ satisfy the same assumptions as in Theorem \ref{thm.IntLip}. Suppose that there exists $N \ge 1$ such that
    \begin{equation}\label{cond.two-side cond}
        \bigg(\frac{\varepsilon_{k+1}}{\varepsilon_{k}}\bigg)^N\leq \frac{\varepsilon_{i+1}}{\varepsilon_{i}} \le \bigg(\frac{\varepsilon_{k+1}}{\varepsilon_{k}}\bigg)^{1/N}, \quad \text{for any } k > i \ge 0.
    \end{equation}
    Then \eqref{int-lip_es_Lip} holds with a constant $C$ depends on the same parameters.
\end{theorem}
\begin{proof}
    We only need to consider two cases. If $\e_1 = \e_1/\e_0 < \nu$ with $\nu$ given by Theorem \ref{thm.IntLip}, then Theorem \ref{thm.typicalLipEst} follows directly from Theorem \ref{thm.IntLip}. Now if $\e_1 \ge \nu$, then by the second inequality of \eqref{cond.two-side cond}, we have
    \begin{equation*}
        \frac{\e_{k+1}}{\e_k} \ge \e_1^N \ge \nu^N,
    \end{equation*}
    for all $k > 0$.
    Let $\vartheta = \nu^N$. By iterating the above inequality, we have
    $
        \e_k \ge \vartheta^k.
    $
    Hence, Lemma \ref{int-lip_thm_Dini} gives the desired estimate.
\end{proof}

\begin{remark}\label{rmk.typical}
    Theorem \ref{thm.typicalLipEst} covers the interesting case $\e_j  = \e_1^j$ for any $\e_1 \in (0,1)$. More generally, it covers the case $\e_j = \e_1^{\beta_j}$, where  $C^{-1} \le \beta_j - \beta_{j-1} \le C$ for some constant $C>0$, and  $\e_1 \in (0,1)$. If $\e_j = \e_1^{\beta_j}$ and $\beta_j - \beta_{j-1} \simeq j^m$ for some $m>0$, then we may still establish some analog of Theorem \ref{thm.typicalLipEst}, under a faster decay assumption of $\delta = \{\delta_j\}_{j\in \N}$, such as $[\delta]_{m+2} < \infty$.
\end{remark}

\section{Boundary Lipschitz estimates}\label{sec_ext}

This section is devoted to proving the boundary Lipschitz estimates. The proofs are more or less parallel to those in the previous section, and we present only the sketch for the Dirichlet boundary value problem. 

Let $\gamma>0$ and $\psi:\mathbb{R}^{d-1}\rightarrow\mathbb{R}$ be a $C^{1,\gamma}$ function with
\begin{align}\label{ext_cond_psi}
  \psi(0)=0\quad\text{and}\quad\|\nabla\psi\|_{L^\infty(\R^{n-1})} +[\nabla\psi]_{C^{0,\gamma}(\mathbb{R}^{d-1})}\leq L.
\end{align}
Set
\begin{align*}
  D_r=D(r,\psi)=\{(x',x_d)\in\mathbb{R}^d: |x'|<r\text{~and~}\psi(x')<x_d<10(L+d)r\},
\end{align*}
\begin{align*}
  \Delta_{r}=\Delta(r,\psi)=\{(x',\psi(x'))\in\mathbb{R}^{d}:|x'|<r\}.
\end{align*}
Note that $D_r$ is a Lipschitz domain for any $r>0$ with a Lipschitz constant depending only on $L$ and $d$.
We introduce a scaling-invariant norm for functions on $\Delta_r$,
\begin{align*}
  \|g\|_{C^{1, \gamma}(\Delta_r)}=\|g\|_{L^\infty(\Delta_r)}+r\|\nabla_{\mathrm{tan}} g\|_{L^\infty(\Delta_r)}+r^{1+\gamma}[\nabla_{\mathrm{tan}}g]_{C^{0, \gamma}(\Delta_r)}, 
\end{align*}
where $\nabla_{\mathrm{tan}}g$ denotes the tangential gradient of $g$ and
\begin{align*}
  [g]_{C^{0, \gamma}(\Delta_r)}=\sup_{x, y\in \Delta_r, x\neq y}\frac{|g(x)-g(y)|}{|x-y|^\gamma}.
\end{align*}

Let $u_\varepsilon$ be a weak solution to
\begin{align}\label{eq.extue}
  -\mathrm{div}(A^\varepsilon(x)\nabla u_\varepsilon)=f \quad \text{in }D_r,\quad u_\varepsilon=g\ \ \text{on }\Delta_r. 
\end{align}
For $r\in (0,1)$, we can rescale the equation to a domain of size $1$. In fact, by setting $v(x)=u_\varepsilon(rx)$, we have
\begin{align*}
  -\mathrm{div}(\widetilde{A}^\varepsilon(x)\nabla v)=\widetilde{f}\ \ \text{in }D(1, \psi_r),\quad v=\widetilde{g}\ \ \text{on }\Delta(1, \psi_r),
\end{align*}
where $\widetilde{A}(y_0, y_1, \cdots)=A(ry_0, y_1, \cdots)$, $\widetilde{f}(x)=r^2f(rx)$, $\widetilde{g}(x)=g(rx)$ and $\psi_r(x')=r^{-1}\psi(rx')$. It is easy to see that the rescaling argument works well for the boundary estimates as in the interior case. Especially, the function $\psi_r$ satisfies condition \eqref{ext_cond_psi} with the same $L$ and $\|g\|_{C^{1, \gamma}(\Delta(r, \psi))}=\|\widetilde{g}\|_{C^{1, \gamma}(\Delta(1, \psi_r))}$. 

\begin{theorem}\label{ext_lem_approx}
    Let $f\in L^2(D_1)$ and $g\in C^1(\Delta_1)$, with $\|g\|_{C^1(\Delta_r)}=\|g\|_{L^\infty(\Delta_r)}+r\|\nabla_{\mathrm{tan}}g\|_{L^\infty(\Delta_r)}$. Assume that $[\delta]_1<\infty$. Let $u_\e$ be a weak solution of \eqref{eq.extue}. Then for each $r\in[\varepsilon_1, 1/2]$, there exists a weak solution $u_0$ to 
    \begin{equation}\label{eq.extu0}
        -\mathrm{div}(\widehat{A}\nabla u_0)=f \quad \text{in } D_r, \quad u_0 = g \quad \text{on } \Delta_r,
    \end{equation}
    such that for $0<\alpha\leq 1$,
    \begin{align*}
        \bigg(\fint_{D_r}|u_\varepsilon-u_0|^2\bigg)^{1/2}& \leq C\Lambda [\delta]_\alpha\bigg(\frac{\varepsilon_1}{r}+\sup_{k\geq 2}\frac{\e_k}{\e_{k-1}}\bigg)^{\alpha\sigma} \\
        &\quad\times\bigg\{\bigg(\fint_{D_{2r}}|u_\varepsilon|^2\bigg)^{1/2}+r^2\bigg(\fint_{D_{2r}}|f|^2\bigg)^{1/2}+\|g\|_{C^1(\Delta_{2r})}\bigg\},
    \end{align*}
    where $\Lambda $ is given by \eqref{def.Lambda}, and $\sigma$ and $C$ depend only on $d, \mu$ and $(\gamma, L)$ in \eqref{ext_cond_psi}.
\end{theorem}

\begin{proof}
    The proof is similar to that of Theorem \ref{thm_approx}.
  By rescaling, we may assume $r=1/2$.   For $(1/2)<t<(3/4)$, let $u_0^t$ be the weak solution to
\begin{align*}
  \begin{cases}
    -\mathrm{div}(\widehat{A}(x)\nabla u_0^t)=f &\text{in }D_t\\
    u_0^t=u_\varepsilon&\text{on }\partial D_t,
  \end{cases}
\end{align*}
and set
\begin{align*}
  u_0=4\int_{1/2}^{3/4} u_0^tdt. 
\end{align*}
Note that $u_0$ satisfies the equation \eqref{eq.extu0} with $r = 1/2$. By Theorem \ref{conver_thm_rate_suboptimal}, for any $\alpha \in (0,1)$,
\begin{align*}
  \|u_\varepsilon-u_0^t\|_{L^2(D_{1/2})}\leq C\Lambda [\delta]_\alpha\Big(\sup_{k\geq 1} \frac{\e_k}{\e_{k-1}}\Big)^{\alpha\sigma}\{\|u_\varepsilon\|_{H^{1}(\partial D_t)}+\|f\|_{L^2(D_t)}\}
\end{align*}
for some $\sigma>0$ depending only on $d, \mu, \gamma $ and $L$. 
Using the co-area formula as in the proof of Theorem \ref{thm_approx}, we obtain
\begin{align*}
  \|u_\varepsilon-u_0\|_{L^2(D_{1/2})}^2
  &\le C^2\Lambda^2 [\delta]_\alpha^2 \Big(\sup_{k\geq 1} \frac{\e_k}{\e_{k-1}}\Big)^{2\alpha\sigma}\\&\quad\times\Big(\|u_\varepsilon\|_{L^2(D_{1})}^2+\|f\|_{L^2(D_{1})}^2+\|g\|_{C^1(\Delta_1)}^2\Big).
\end{align*}
This implies the desired estimate as in Theorem \ref{thm_approx}. 
\end{proof}

Now, similar to the interior estimates,  for $0<r\leq 1$, 
we define
\begin{align*}
  H(r; u_\varepsilon)=\frac{1}{r}\inf_{P\in \mathcal{P}}\bigg\{\bigg(\fint_{D_r}|u_\varepsilon-P|^2\bigg)^{1/2}+\|g-P\|_{C^{1, \gamma}(\Delta_r)}\bigg\}+r\bigg(\fint_{D_r}|f|^p\bigg)^{1/p},
\end{align*}
and
\begin{align*}
  h(r; u_\varepsilon)=|\nabla P_r|,
\end{align*}
where $P_r$ is the function achieving the infimum in $H(r; u_\varepsilon)$. As in Theorem \ref{thm_Lip_1step}, we have
\begin{lemma}\label{ext_lem_firststep}
  Suppose that $[\delta]_1<\infty$ and $0<\alpha\leq 1$. There exists a constant $0< c_0\leq 1$, depending on $d, \mu, p, [\delta]_0, [\delta]_1, \gamma $ and $L$, such that, if \eqref{cond_logseparation} holds for $k\geq 2$, then for $r\in[\varepsilon_1, 1]$,
  \begin{align*}
    H(r; u_\varepsilon)&\leq C_0 r^\lambda H(1; u_\varepsilon)+CM_0\bigg[\Big(\frac{\varepsilon_1}{r}\Big)^{\alpha\sigma}+r^\lambda+\sup_{k\geq 2}\bigg(\frac{\varepsilon_k}{\varepsilon_{k-1}}\bigg)^{\alpha\sigma}\bigg]\\&\qquad\times\{H(1; u_\varepsilon)+h(1; u_\varepsilon)\},
  \end{align*}
  where $M_0$ is given in Theorem \ref{thm_Lip_1step}, $\sigma$ is given in Lemma \ref{ext_lem_approx}, $\lambda$ and $C_0$ depend only on $d, \mu, p, [\delta]_0, \gamma$ and $L$, and $C$ depends additionally on $\alpha$ and $[\delta]_1$.
\end{lemma}
\begin{proof}
    Using the boundary $C^{1, \gamma}$ estimates of the elliptic equations with Lipschitz coefficients, one can deduce as in Theorem \ref{thm_Lip_1step} that there exists a small constant $\theta\in (0, 1/8)$, depending only on $d, \mu, p, \|\nabla \widehat{A}\|_{L^\infty}, \gamma $ and $ L$, such that
  \begin{align*}
    &\frac{1}{\theta r}\inf_{P\in \mathcal{P}}\bigg\{\bigg(\fint_{D_{\theta r}}|u_0-P|^2\bigg)^{1/2}+\|g-P\|_{C^{1, \gamma}(\Delta_{\theta r})}\bigg\}+\theta r\bigg(\fint_{D_{\theta r}}|f|^p\bigg)^{1/p}\\
    &\leq \frac{1}{2r}\bigg(\fint_{D_{r}}|u_0-P_r|^2\bigg)^{1/2}+\frac{1}{2r}\|g-P_r\|_{C^{1, \gamma}(\Delta_r)}+\frac{r}{2}\bigg(\fint_{D_{r}}|f|^p\bigg)^{1/p}\\
    &\qquad\qquad+\frac{r}{2}\|\nabla_x\widehat{A}\|_{L^\infty(D_r)} h(r; u_\varepsilon)\\
    &\leq \frac{1}{2r}\bigg(\fint_{D_{r}}|u_\varepsilon-u_0|^2\bigg)^{1/2}+\frac{1}{2}H(r)+\frac{r}{2}\|\nabla_x\widehat{A}\|_{L^\infty(D_r)} h(r),
  \end{align*}
  which, together with Lemma \ref{ext_lem_approx}, gives
  \begin{align*}
    H(\theta r)&\leq C\Lambda [\delta]_\alpha\bigg(\frac{\varepsilon_1}{r}+\sup_{k\geq 2}\frac{\e_k}{\e_{k-1}}\bigg)^{\alpha\sigma}\{H(2r)+h(2r)\}\\
    &\qquad+\frac{1}{2}H(r)+\frac{r}{2}\|\nabla_x\widehat{A}\|_{L^\infty(D_r)} h(r).
  \end{align*}
  The rest of the proof is exactly the same as that of Theorem \ref{thm_Lip_1step}. 
\end{proof}

Based on Lemma \ref{ext_lem_firststep}, following the same procedures of Corollary \ref{int_coro_firststep}, Lemma \ref{int-Lip_prop_Hh} and Theorem \ref{thm_Hh_bound}, we arrive at
\begin{theorem}\label{ext_lem_Hh_bound}
     Let $c_0$ be the universal constant given by Lemma \ref{ext_lem_firststep}. Suppose that $[\delta]_{1+\rho}<\infty$ for some $\rho>0$ and that conditions \eqref{int-Lip_cond_sumbound} and \eqref{cond_logseparation_m} hold with $\alpha=\rho/2$ for each $m\geq 0$. There exists $K_0\in(1, 2]$, depending only on $d, \mu, p, [\delta]_0, \gamma $ and $L$, such that if $1<K\leq K_0$, then for any $r\in(0, 1]$,
  \begin{align*}
    H(r; u_\varepsilon)+h(r; u_\varepsilon)\leq C\{H(1; u_\varepsilon)+h(1; u_\varepsilon)\},
  \end{align*}
  where $C$ depends on $d, \mu, p, [\delta]_0, \gamma, L, \rho$ and $[\delta]_{1+\rho}$.
\end{theorem}
This finally leads to the following boundary estimate (exactly Theorem \ref{thm.BdryLip}) by the same proofs as in Theorems \ref{int-lip_lem_Lip} and  \ref{thm.IntLip}.
\begin{theorem}\label{Lip-B}
    Assume that $[\delta]_{1+\rho}<\infty$ for some $\rho>0$. Let $u_\varepsilon$ be a weak solution to 
    \begin{equation*}
        -\mathrm{div}(A^\varepsilon(x)\nabla u_\varepsilon)=f \quad \text{in }D_R,\quad u_\varepsilon=g\ \ \text{on }\Delta_R,
    \end{equation*}
    where $0<R\leq 1$, $f\in L^p(B_{R})$ with $p>d$, $g\in C^{1, \gamma}(\Delta_R)$. Suppose that the scale-separation condition \eqref{int-lip_cond_separation_ratio} holds for some $N\geq 1$. Then there exists a constant $\nu>0$, depending only on $d, \mu, p, \rho, N, [\delta]_0, [\delta]_1, \gamma$ and $L$, such that whenever $\varepsilon_1\leq \nu$, we have
  \begin{align*}
    \| \nabla u_\e\|_{L^\infty(D_{R/2})} \leq C\bigg\{\bigg(\fint_{D_{R}}|\nabla u_\varepsilon|^2\bigg)^{1/2}+R\bigg(\fint_{D_{R}}|f|^p\bigg)^{1/p}+R^{-1}\|g\|_{C^{1, \gamma}(\Delta_R)}\bigg\},
  \end{align*}
  where $C$ depends only on $d, \mu, p, [\delta]_0, \gamma, L, \rho$ and $[\delta]_{1+\rho}$. 
\end{theorem}

Combining the interior and boundary estimates, we obtain the uniform Lipschitz regularity in $\Omega$.

\begin{theorem}\label{Lip-BB}
  Let $\Omega$ be a bounded $C^{1, \gamma}$ domain in $\R^d$, $f\in L^p(\Omega)$ with $p>d$, $g\in C^{1, \gamma}(\partial\Omega)$. Assume that $[\delta]_{1+\rho}<\infty$ for some $\rho>0$ and the separation condition \eqref{int-lip_cond_separation_ratio} holds for some $N\geq 1$. Let $u_\varepsilon$ be the solution to \eqref{eq_infinite}. There exists a constant $\nu>0$, depending only on $d, \mu, p, \rho, N, [\delta]_0, [\delta]_1$ and the $C^{1, \gamma}$ character of $\partial\Omega$, such that whenever $\varepsilon_1\leq \nu$, we have
  \begin{align*}
    \|\nabla u_\varepsilon\|_{L^\infty(\Omega)}\leq C\{\|f\|_{L^p(\Omega)}+\|g\|_{C^{1, \gamma}(\partial\Omega)}\},
  \end{align*}
  where $C$ depends only on $d, \mu, p, [\delta]_0, \rho, [\delta]_{1+\rho}$ and $\Omega$. 
\end{theorem}

By a real-variable argument involving reverse H\"older inequalities (see e.g. \cite{Shen_book}), the uniform Lipschitz estimates in Theorems \ref{int-lip_lem_Lip} and \ref{Lip-B} imply the uniform $W^{1, p}$ estimates.

\begin{theorem}\label{W1p-1}
    Let $\Omega$ be a bounded $C^{1, \gamma}$ domain in $\R^d$.
    Suppose $\delta$ and $\e$ satisfy the same conditions as in Theorem \ref{Lip-BB}.
    For $f\in L^p(\Omega)$ with $1< p< \infty$, let $u_\e$ be the solution of the boundary value problem,
    \begin{equation*}
    \left\{
    \aligned
    -\text{\rm div} (A^\e (x) \nabla u_\e)
    &=\text{\rm div}(f) & \quad & \text{ in } \Omega,\\
    u_\e & =0 & \quad & \text{ on } \partial\Omega.
    \endaligned
    \right.
    \end{equation*}
    Then,
    \begin{equation*}
        \|\nabla u_\e \|_{L^p (\Omega)}
        \le C_p \| f \|_{L^p(\Omega)},
    \end{equation*}
    where $C_p$ depends only on $d, \mu, p, [\delta]_0, \rho, [\delta]_{1+\rho}$ and $\Omega$. 
\end{theorem}

   Finally, we point out that with minor modifications of their proofs,
   one may extend the uniform estimates
   in Theorems \ref{Lip-BB} and \ref{W1p-1} to solutions with Neumann boundary conditions,
   under the same assumptions on $\e$, $\delta$, and $\Omega$.
 
\bibliographystyle{alpha}
\bibliography{ref}
\end{document}